\documentclass[a4paper,reqno]{amsart}

\usepackage[active]{srcltx}
\usepackage[all]{xy}
\usepackage{subfig}
\usepackage{hyperref}

\usepackage{amssymb,latexsym}
\usepackage{mathrsfs}
\usepackage{graphics}
\usepackage{latexsym}
\usepackage{psfrag}
\usepackage{verbatim}
\usepackage[dvips]{graphicx}
\usepackage{color}
\usepackage{pifont,marvosym}
\usepackage{bbm}
\usepackage{stmaryrd}

\theoremstyle{plain}
\newtheorem{lemma}{Lemma}[section]
\newtheorem{proposition}[lemma]{Proposition}
\newtheorem{theorem}[lemma]{Theorem}
\newtheorem{corollary}[lemma]{Corollary}

\theoremstyle{definition}
\newtheorem{definition}[lemma]{Definition}

\theoremstyle{remark}
\newtheorem{remark}[lemma]{Remark}
\newtheorem{example}[lemma]{Example}

\numberwithin{equation}{section}

\newcommand{\R}{\mathbb{R}}
\newcommand{\Q}{\mathbb{Q}}
\newcommand{\N}{\mathbb{N}}

\newcommand{\TV}{\text{\rm Tot.Var.}}

\newcommand{\dist}{\text{\rm dist}}
\newcommand{\supp}{\text{\rm supp}}

\newcommand{\clos}{\text{\rm clos}}

\newcommand{\Id}{\mathbb{I}}

\newcommand{\ind}{\mathbbm{1}}

\newcommand{\A}{\mathfrak{A}}

\newcommand{\inter}{\mathrm{int}}

\newcommand{\Graph}{\textrm{graph}}

\newcommand{\epi}{\mathrm{epi}}
\newcommand{\aff}{\mathrm{aff}}
\newcommand{\conv}{\mathrm{conv}}
\newcommand{\convd}{\mathrm{conv}_\mathrm{d}}

\newcommand{\nfrac}[2]{\genfrac{}{}{0pt}{}{#1}{#2}}
\newcommand{\interr}{\inter_{\mathrm{rel}}}

\def\a{\mathfrak a}
\def\c{\mathfrak c}
\def\b{\mathfrak b}

\def\bD{\mathbf D}

\def\tRd{[0,\infty) \times \R^d}
\def\tpt{\mathtt p^{\bar t}}

\newcommand{\mintheta}{{\bar\theta}}
\newcommand{\minvartheta}{{\bar\vartheta}}
\newcommand{\RR}{[0,+\infty)\times \mathbb R}
\newcommand{\extphi}{\bar\phi}
\newcommand{\extpsi}{\bar\psi}
\newcommand{\extc}{\bar{\mathtt c}}
\newcommand{\disth}{ \mathtt d_{\mathrm H} }
\newcommand{\segment}[1]{\llbracket #1 \rrbracket}	
\def\mhC{\mathcal C}
\def\mhA{\mathcal A}
\def\prj{\mathtt p}

\title[The decomposition of optimal transportation problems with convex cost]{The decomposition of optimal transportation problems with convex cost}

\author{Stefano Bianchini}
\address{SISSA, via Bonomea, 263, IT-34136 Trieste (ITALY)}
\email{bianchin@sissa.it}

\author{Mauro Bardelloni}
\address{SISSA, via Bonomea, 263, IT-34136 Trieste (ITALY)}
\email{mbarde@sissa.it}

\date{\today\text{}}

\begin{document}

\bibliographystyle{plain}


\begin{abstract}
Given a positive l.s.c. convex function $\mathtt c : \R^d \to \R^d$ and an optimal transference plane $\underline{\pi}$ for the transportation problem
\begin{equation*}
\int \mathtt c(x'-x) \pi(dxdx'),
\end{equation*}
we show how the results of \cite{biadan} on the existence of a \emph{Sudakov decomposition} for norm cost $\mathtt c= |\cdot|$ can be extended to this case.

More precisely, we prove that there exists a partition of $\R^d$ into a family of disjoint sets $\{S^h_\a\}_{h,\a}$ together with the projection $\{O^h_\a\}_{h,\a}$ on $\R^d$ of proper extremal faces of $\epi\, \mathtt c$, $h = 0,\dots,d$ and $a \in \A^h \subset \R^{d-h}$, such that
\begin{itemize}
\item $S^h_\a$ is relatively open in its affine span, and has affine dimension $h$;
\item $O^h_\a$ has affine dimension $h$ and is parallel to $S^h_\a$;
\item $\mathcal L^d(\R^d \setminus \cup_{h,\a} S^h_\a) = 0$, and the disintegration of $\mathcal L^d$, $\mathcal L^d = \sum_h \int \xi^h_\a \eta^h(d\a)$, w.r.t. $S^h_\a$ has conditional probabilities $\xi^h_\a \ll \mathcal H^h \llcorner_{S^h_\a}$;
\item the sets $S^h_\a$ are essentially cyclically connected and cannot be further decomposed.
\end{itemize}
The last point is used to prove the existence of an optimal transport map.

The main idea is to recast the problem in $(t,x) \in [0,\infty] \times \R^d$ with an $1$-homogeneous norm $\bar{\mathtt c}(t,x) := t \mathtt c(- \frac{x}{t})$ and to extend the regularity estimates of \cite{biadan} to this case.
\end{abstract}

\thanks{This work is supported by ERC Starting Grant 240385 "ConsLaw".}

\maketitle

\centerline{Preprint SISSA 45/2014/MATE}

\tableofcontents



\section{Introduction}
\label{S_introduction}


Let $\mathtt c : \R^d \to \R$ a non negative convex l.s.c. function with superlinear growth,
and consider the following optimal transportation problem: given $\mu,\nu \in \mathcal P(\R^d)$, find a minimizer $\pi$ of the problem
\begin{equation}
\label{E_original_transp}
\inf \bigg\{ \int_{\R^d\times \R^d} \mathtt c(x'-x)\, \pi({\rm d}x{\rm d}x'), \quad \pi \in \Pi (\mu,\nu ) \bigg\},
\end{equation}
where $\Pi(\mu,\nu)$ is the set of transference plans $\pi \in \mathcal P(\R^d \times \R^d)$ with marginals $\mu,\nu$ respectively. W.l.o.g. we can assume that the above minimum (the transference cost $\mathcal C(\mu,\nu)$) is not $\infty$.

It is well known that in this setting the problem \eqref{E_original_transp} has a solution (\emph{optimal transference plan}). A standard question is if there exists an optimal plan given by a map (\emph{optimal transport map}): this problem is the \emph{Monge transportation problem}, while the existence of an optimal solution to \eqref{E_original_transp} is the \emph{Kantorovich transportation problem}. Simple examples show that if $\mu$ is not a.c. w.r.t. $\mathcal L^d$, then in general the Monge problem has no solution. In the following we will thus assume that $\mu \ll \mathcal L^d$. 

The aim of this paper is to prove a decomposition result from which one deduces the existence of an optimal transport map. The result is actually stronger, showing that for any fixed optimal plan $\bar \pi$ it is possible to give a partition of the space $\R^d$ into sets $S^h_\a$ which are essentially indecomposable (a precise definition will be given in the following): it is standard from this property of the partition to deduce the existence of an optimal map.

In the case of norm cost, there is a large literature on the existence of optimal maps: see for example \cite{conf:optcime,Car:strictly,caffafeldmc,champdepasc:Monge,champdepasc:MongeS,trudiwang}. The original Sudakov strategy has been finally implemented in the norm case in \cite{biadan}. In the case of convex cost, an attempt to use a similar  approach of decomposing the transport problems can be found in \cite{cardpsan:strat}.

In order to state the main result, in addition to the standard family of transference plans $\Pi(\mu,\nu)$ we introduce the notion of \emph{transference plan subjected to a partition}: given $\underline \pi \in \Pi(\mu,\nu)$ and a partition $\{S_\a = \mathtt f^{-1}(\a)\}_{\a \in \A}$ of $\R^d$, with $\mathtt f : \R^d \to \A$ Borel and $\mu(f^{-1}(\A)) = 1$, let $\underline \pi_\a$ be the conditional probabilities of the disintegration of $\underline \pi$ w.r.t. $\{S_\a \times \R^d\}_\a$,
\begin{equation*}
\underline \pi = \int \underline \pi_\a m(d\a), \qquad m := \mathtt f_\sharp \mu.
\end{equation*}
Define the family of probabilities $\underline \nu_\a$ as the second marginal of $\underline \pi_a$ (the first being the conditional probability $\mu_\a$ of $\mu$ when disintegrated on $\{S_\a\}_\a$, $\mu = \int \mu_\a m(d\a)$). Then set
\begin{equation*}
\Pi(\mu,\{\underline \nu_\a\}) := \bigg\{ \pi : \pi = \int \pi_\a m(d\a) \ \text{with} \ \pi_a \in \Pi(\mu_a,\underline \nu_a) \bigg\}.
\end{equation*}
Clearly this is a nonempty subset of $\Pi(\mu,\nu)$.

A second definition is the notion of \emph{cyclically connected sets}. We recall that given a cost $\mathtt c : \R^d \times \R^d \to [0,\infty]$ and a set $\Gamma \subset \{\mathtt c < \infty\}$, the set $S \in \R^d$ is \emph{$\mathtt c$-cyclically connected} if for every couple of point $x,x' \in S$ there are a family $(x_i,y_i) \in \Gamma$, $i=0,\dots,N-1$, such that
\begin{equation*}
\mathtt c(x_{i+1 \mod N},y_i) < \infty \qquad \text{and} \qquad x_0 = x, \ x' = x_j \ \text{for some} \ j \in \{0,\dots,N-1\}.
\end{equation*}
When the cost $\mathtt c$ is clear from the  setting, we will only say \emph{cyclically connected}.

We will need to define the disintegration of the Lebesgue measure on a partition. The formula of the disintegration of a $\sigma$-finite measure $\varpi$ w.r.t. a partition $\{S_\a = \mathtt f^{-1}(\a)\}_\a$ is intended in the following sense: fix a strictly positive function $f$ such that $\varpi' := f \varpi$ is a probability and write
\begin{equation*}
\varpi = f^{-1} \varpi' = \int \big( f^{-1} \varpi'_\a \big) \sigma(d\a),
\end{equation*}
where
\begin{equation*}
\varpi' = \int \varpi'_\a \sigma(d\a), \qquad \sigma = \mathtt f_\sharp \varpi',
\end{equation*}
is the disintegration of $\varpi'$. It clearly depends on the choice of $f$, but not the property of being absolutely continuous as stated below. 

We say that a set $S \subset \R^d$ is \emph{locally affine} if it is open in its affine span $\aff\,S$. If $\{S_\a\}_\a$ is a partition into disjoint locally affine sets, we say that the disintegration is \emph{Lebesgue regular} (or for shortness \emph{regular}) if the disintegration of $\mathcal L^d$ w.r.t. the partition satisfies
\begin{equation*}
\mathcal L^d \llcorner_{\cup_\a S_a} = \int_{\A} \xi_\a \eta(d\a), \qquad \xi_\a \ll \mathcal H^h \llcorner_{S_\a}, \ h = \dim\,S_\a.
\end{equation*}

At this point we are able to state the main result.
%
%
%

\begin{figure}
\centerline{\resizebox{14cm}{7cm}{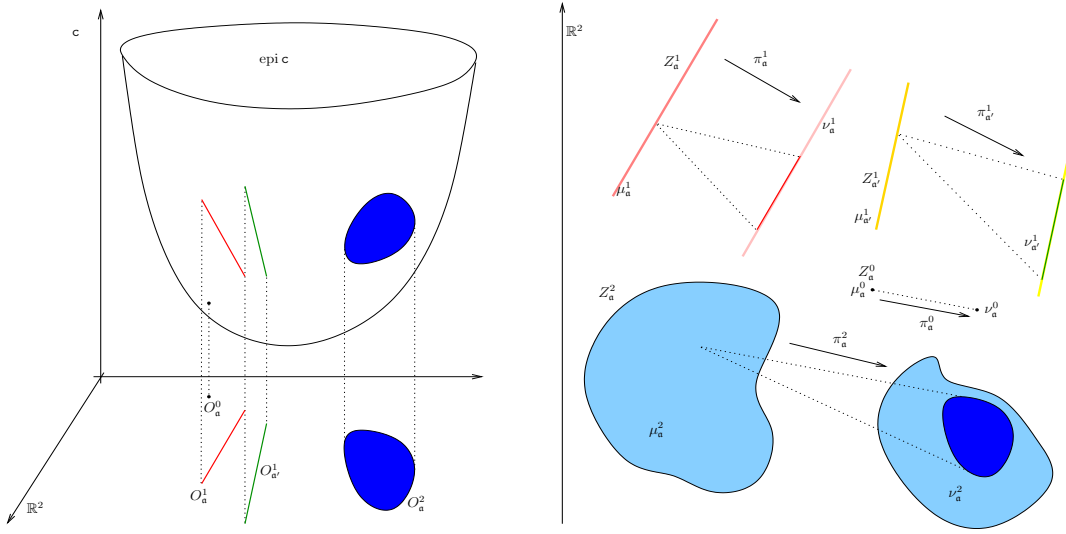}}
\caption{Theorem \ref{T_main_theor} in the plane case, $h=0,1,2$: the projection $O^h_\a$ of the extremal face of $\epi\,\mathtt c$ provides the convex cost $\ind_{O^h_\a}$ used in the decomposed transportation problems $\pi^h_\a \in \mathcal P(Z^h_\a \times \R^d)$.}
\end{figure}

%
%
%

\begin{theorem}
\label{T_main_theor}
Let $\underline \pi \in \Pi(\mu,\nu)$ be an optimal transference plan, with $\mu \ll \mathcal L^d$. Then there exists a family of sets $\{S^h_\a,O^h_\a\}_{\nfrac{h = 0,\dots,d}{\a \in \A^h}}$, $S^h_\a,O^h_\a \subset \R^d$, such that the following holds:
\begin{enumerate}
\item $S^h_\a$ is a locally affine set of dimension $h$;
\item $O^h_\a$ is a $h$-dimensional convex set contained in an affine subspace parallel to $\aff\,S^h_\a$ and given by the projection on $\R^d$ of a proper $h$-dimensional extremal face of $\epi\,\mathtt c$; 
\item $\mathcal L^d(\R^d \setminus \cup_{h,\a} S^h_\a) = 0$;
\item 
the partition is Lebesgue regular;
\item if $\pi \in \Pi(\mu,\{\underline \nu^h_\a\})$ then optimality in \eqref{E_original_transp} is equivalent to 
\begin{equation}
\label{E_equiv_optim}
\sum_h \int \bigg[ \int \ind_{O^h_\a}(x'-x) \pi^h_\a(dxdx') \bigg] m^h(d\a) < \infty,
\end{equation}
where $\pi = \sum_h \int_{\A^h} \pi^h_\a m^h(d\a)$ is the disintegration of $\pi$ w.r.t. the partition $\{S^h_\a \times \R^d\}_{h,\a}$;
\item \label{Point_indecomp_main_th} for every carriage $\Gamma$ of $\pi \in \Pi(\mu,\{\underline \nu^h_\a\})$ there exists a $\mu$-negligible set $N$ such that each $S^h_\a \setminus N$ is $\ind_{O^h_\a}$-cyclically connected.
\end{enumerate}
\end{theorem}

We will often refer to the last condition by $\{S^h_\a\}_{h,\a}$ is \emph{($\mu$-) essentially cyclically connected}, i.e. the set of the partitions are cyclically connected up to a ($\mu$-) negligible set. Usually the measure is clear from the context.

Using the fact that $\mathtt c \llcorner_{O^h_\a}$ is affine, a simple computation allows to write
\begin{equation*}
\begin{split}
\int \mathtt c(x'-x) \pi(dxdx') =&~ \sum_h \int_{\A^h} \mathtt c(x'-x) \pi^h_\a m^h(d\a) \crcr
=&~ \sum_h \int_{\A^h} \mathtt c(x'-x) \ind_{O^h_\a}(x'-x) \pi^h_\a m^h(d\a) \crcr  
=&~ \sum_h \int \bigg[ a^h_\a + b^h_\a \cdot \bigg( \int x' \underline \nu^h_\a - \int x \mu^h_\a \bigg) \bigg] m^h(d\a),
\end{split}
\end{equation*}
where $a^h_\a + b^h_\a \cdot x$ is a support plane of the face $O^h_\a$. In particular, for $\underline{\pi}$-a.e. $(x,x')$, the set $O^h_\a$ can be determined from $S^h_\a$ as the projection of the extremal face of $\epi\,\mathtt c$ containing $x'-x$ and whose projection is parallel to $x' - x + \mathrm{span}(S^h_\a - x)$.

Following the analysis of \cite{biadan}, the decomposition $\{S^h_\a,O^h_\a\}_{h,\a}$ will be called \emph{Sudakov decomposition subjected to the plan $\underline \pi$}. Note that the indecomposability of $S^h_\a$ yields a \emph{uniqueness} of the decomposition in the following sense: if $\{S^k_\b,O^k_\b\}_{k,\b}$ is another partition, then by \eqref{E_equiv_optim} one obtains that $O^k_\b \subset O^h_\a$ (resp. $O^h_\a \subset O^k_\b$) on $S^k_\b \cap S^h_\a$ (up to $\mu$-negligible sets), and then Point \eqref{Point_indecomp_main_th} of the above theorem gives that $S^k_\b \subset S^h_\a$ (resp. $S^h_\a \subset S^k_\b$). But then the indecomposability condition for $\{S^h_a,O^h_\a\}_{h,\a}$ (resp. $\{S^h_a,O^h_\a\}_{h,\a}$) is violated.

We remark again that the indecomposability is valid only in the convex set $\Pi(\mu,\{\underline \nu^h_\a\}) \subset \Pi(\mu,\nu)$: in general by changing the plan $\underline{\pi}$ one obtains another decomposition. In the case $\nu \ll \mathcal L^d$, this decomposition is independent on $\underline \pi$: this is proved at the end of Section \ref{S_transl_orig}, Theorem \ref{T_main_theor_ac}.


In order to illustrate the main result, we present some special cases. A common starting point is the existence of a couple of potentials $\phi$, $\psi$ (see \cite[Theorem 1.3]{villa:topics}) such that
\begin{equation*}
\psi(x') - \phi(x) \leq \mathtt c(x'-x) \qquad \textrm{for all }x,x'\in\R^d
\end{equation*}
and
\begin{equation}
\label{E_pi_optimal_couples}
\psi(x') - \phi(x) = \mathtt c(x'-x) \qquad \textrm{for } \pi\text{\rm-a.e. } (x,x') \in \R^d\times\R^d,
\end{equation}
where $\pi$ is an arbitrary optimal transference plan. Assuming for simplicity that $\mu,\nu$ have compact support and observing that $\mathtt c$ is locally Lipschitz, we can take $\phi$, $\psi$ Lipschitz, in particular $\mathcal L^d$-a.e. differentiable. By \eqref{E_pi_optimal_couples} and the assumption $\mu \ll \mathcal L^d$ one obtains that for $\pi$-a.e. $(x,x')$ the gradient $\nabla \phi$ satisfies the inclusion
\begin{equation}
\label{E_subdiff_inclu}
\nabla \phi(x) \in \partial^- \mathtt c(x'-x),
\end{equation}
being $\partial^- \mathtt c$ the subdifferential of the convex function $\mathtt c$.

Assume now $\mathtt c$ strictly convex. Being the proper extremal faces of $\epi\,\mathtt c$ only points, the statement of Theorem \eqref{T_main_theor} gives that the decomposition is trivially $\{\{x\},O_x\}_x$, where $O_x$ is some vector in $\R^d$. In this case for all $p = \nabla \phi(x)$ there exists a unique $q = x'-x$ such that \eqref{E_subdiff_inclu} holds. Then one obtains that $O_x = \{q\}$.

\noindent The second case is when $\mathtt c$ is a norm: in this case the sets $O^h_\a$ become cones $C^h_\a$. This case has been studied in \cite{biadan}: in the next section we will describe this result more deeply, because our approach is based on their result.

\noindent The cases of convex costs with convex constraints or of the form $h(\|x'-x\|)$, with $h : \R^+ \to \R^+$ strictly increasing and $\|\cdot\|$ an arbitrary norm in $\R^2$ are studied in \cite{cardpsan:strat}.

As an application of these reasonings, we show how \eqref{E_subdiff_inclu} can be used in order to construct of an optimal map, i.e. a solution of the Monge transportation problem with convex cost (see \cite{Car1,cardpsan:strat}): indeed, one just minimize among $\pi \in \Pi(\mu,\{\underline \nu^h_\a\})$ the secondary cost $|\cdot|^2/2$ ($|\cdot|$ being the standard Euclidean norm), and by the cyclically connectedness of $S^h_\a$ one obtains a couple of potentials $\{\phi^h_\a,\phi^h_\a\}_{h,\a}$. Since $\mu$, $\nu$ have compact support, then again these potentials are $\mu^h_\a$-a.e. differentiable, and a simple computation shows that $x' - x$ is the unique minimizer of
\begin{equation*}
\frac{|p|^2}{2} - \nabla \phi(x) \cdot p + \ind_{O^h_\a}(p).
\end{equation*}
The fact that this construction is Borel regular w.r.t. $h,\a$ is standard (\cite{biacar:cmono,biacav:mongemetric,biadan,Car1,cardpsan:strat}), and follows by the regularity properties of the map $h,\a \to S^h_\a,O^h_\a$ in appropriate Polish spaces, see the definitions at the beginning of Section \ref{S_directed_locally_affine_partitions}.

\begin{corollary}
\label{C_Monge_sol}
There exists an optimal map $\mathtt T : \R^d \to \R^d$ such that $(\Id,\mathtt T)_\sharp \mu$ is an optimal transference plan belonging to $\pi(\mu,\{\underline \nu^h_\a\})$.
\end{corollary}

Note that by varying $\underline{\pi}$ and the secondary cost one obtains infinitely many different optimal maps. An analysis of the regularity properties of the set of maps can be found in \cite{biacav:mongemetric}.


\begin{remark}
\label{R_only_compact}
In the proof we will only consider the case of $\mu$, $\nu$ compactly supported. This assumption avoids some technicalities, and it is fairly easy to recover the general case.

Indeed, let $K_n \nearrow \R^d$ be a countable family of compact sets and consider $\pi_n := \pi \llcorner_{K_n \times K_n}$. Assume that Theorem \ref{T_main_theor} is proved for all $\pi_n$: let $(S^{h,n}_\a,O^{h,n}_\a)$ be the corresponding decomposition. Up to reindexing and regrouping the sets, one can take $S^{h,n}_\a \nearrow S^h_\a$ and since $\dim\,O^{h,n}_\a$ is increasing with $n$, then $O^{h,n} = O^h_\a$ for $n$ sufficiently large. Hence $\{S^h_\a,O^h_\a\}_{h,\a}$ is the desired decomposition.
\end{remark}

\subsection{Description of the approach}
\label{Ss_descr_appro}

The main idea of the proof is to recast the problem in $\R^{d+1}$ with a $1$-homogeneous cost $\bar{\mathtt c}$ and use the strategy developed in \cite{biadan}.

Define
\begin{equation*}
\bar \mu := (1,\Id)_\sharp \mu, \qquad \bar \nu := (0,\Id)_\sharp \nu,
\end{equation*}
and the cost
\begin{equation}
\label{E_def_bar_tt_c}
\bar{\mathtt c}(t,x) := \left\{\begin{array}{ccc} 
t\,  \mathtt c\big(- \frac{x}{t}\big) & t>0,\\
\ind_{(0,0)} & t=0, \\
+\infty &\rm{otherwise,}
\end{array}\right.
\end{equation}
where $(t,x) \in \R^+ \times \R^d$. It is clear that the minimisation problem \eqref{E_original_transp} is equivalent to
\begin{equation}\label{E_new_transp}
\int_{\big( \R^+\times\R^d \big)\times \big( \R^+\times\R^d \big)} \bar{\mathtt c}(t-t', x-x') \, \overline \pi({\rm d}t{\rm d}x{\rm d}t'{\rm d}x'), \qquad \bar\pi\in\Pi(\bar\mu,\bar{\nu}).
\end{equation}
In particular, every optimal plan $\pi$ for the problem \eqref{E_original_transp} selects an optimal $\bar\pi := \big( (1,\Id) \times (0,\Id) \big)_\sharp \pi$ for the problem \eqref{E_new_transp} and viceversa.

The potentials $\extphi$, $\extpsi$ for \eqref{E_new_transp} can be constructed by the Lax formula from the potentials $\phi$, $\psi$ of the problem \eqref{E_original_transp}: define
\begin{equation}
\label{HJ_def_extphi}
\extphi(t,x) := \min_{x'\in\R^d} \big\{ - \psi(x') + \bar{\mathtt c}(t,x-x') \big\}, \qquad t \geq 0
\end{equation}
\begin{equation*}
\extpsi(t,x) := \max_{x'\in\R^d} \big\{ - \phi(x') - \bar{\mathtt c}(1-t,x'-x) \big\}, \qquad t \leq 1.
\end{equation*}
It clearly holds
\begin{equation*}
\extphi(0,x) = \bar \psi(0,x) = -\psi(x) \qquad \text{\rm and} \qquad \extphi(1,x) = \bar \psi(1,x) = -\phi(x), 
\end{equation*}
so that the function $\extphi, \extpsi$ are at $t=0,1$ conjugate forward/backward solutions of the Hamilton-Jacoby equation
\begin{equation}
\label{E_HJ_intro}
\partial_t u + H(\nabla u) = 0,
\end{equation}
with Hamiltonian $H = (\mathtt c)^*$, the Legendre transform of $\mathtt c$. (This is actually the reason for the choice of the minus sign in the definition of \eqref{E_def_bar_tt_c}.)

\begin{figure}
\centering{\resizebox{14cm}{7cm}{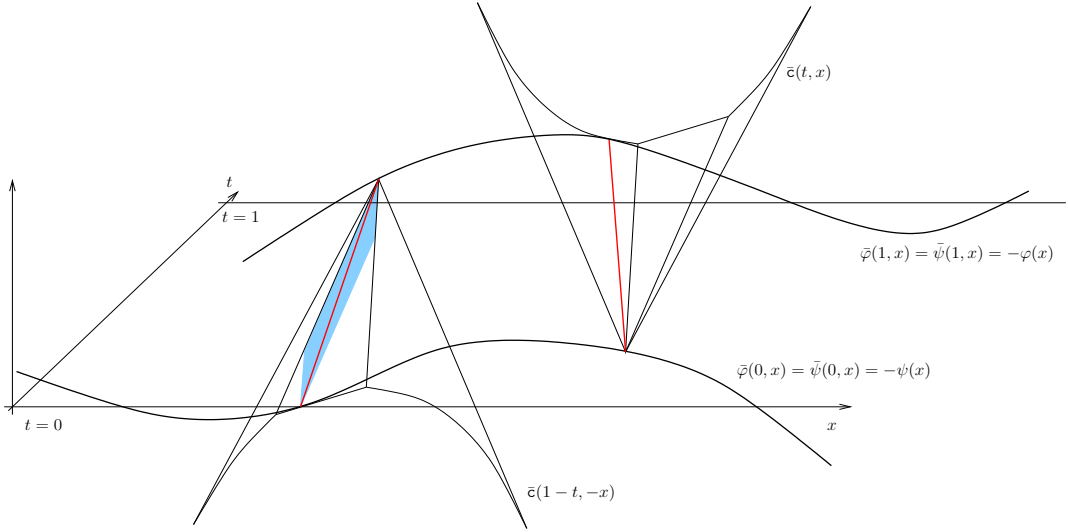}}
\caption{The formulation in $\R^{d+1}$ as a HJ equation: in general in the common region $0 \leq t \leq 1$ it holds $\bar \psi \leq \bar{\phi}$, but in the (red) optimal rays and the depicted region the equality holds.}
\end{figure}

	

By standard properties of solutions to \eqref{E_HJ_intro} one has
\begin{equation*}
\extphi(t,x) - \extphi(t',x') \leq \bar{\mathtt c}(t-t',x-x'),\quad \textrm{ for every }t \geq t' \geq 0,\ x,x'\in\mathbb R^d,
\end{equation*}	 
and for all $\bar \pi$ optimal
\begin{equation*}
\extphi(z) - \extphi(z') = \bar{\mathtt c}(z-z'), \qquad \textrm{for }\overline\pi\textrm{-a.e. } z=(t,x),z'=(t',x') \in \RR^d.
\end{equation*}

Being $\bar{\mathtt c}$ a $1$-homogeneous cost, one can use the same approach of \cite{Dan:PhD} in order to obtain a first \emph{directed locally affine partition} $\{Z ^h_\a,C^h_\a\}_{h,\a}$, where $Z^h_\a$ is a relatively open (in its affine span) set of affine dimension $h+1$, $h \in \{0,\dots,d\}$, and $C^h_\a$ is the projection of an $(h+1)$-dimensional convex extremal face of $\epi\,\bar{\mathtt c}$ (a \emph{cone} due to $1$-homogeneity) given by
\begin{equation*}
C^h_\a = \R^+ \cdot \partial^+ \bar \phi(z), \qquad z = (t,x) \in Z^h_\a.
\end{equation*}
The definition of $\partial^+ \bar \phi$ is the standard formula
\begin{equation*}
\partial ^+  {\extphi}(z) := \Big\{ z'\in \RR^d :  {\extphi}(z')- {\extphi}(z)=\bar{\mathtt c}(z'-z) \Big\}.
\end{equation*}
By the results of \cite{Dan:PhD}, this first decomposition satisfies already many properties stated in Theorem \ref{T_main_theor}:
\begin{enumerate}
\item $Z^h_\a$ is locally affine of dimension $h+1$;
\item $C^h_\a$ is the projection of an extremal cone of $\epi\,\bar{\mathtt c}$ of dimension $h+1$ parallel to $Z^h_\a$;
\item $\mathcal L^{d+1}(\RR^d \setminus \cup_{h,a} Z^h_\a) = 1$;
\item $\bar \pi' \in \Pi(\bar \mu,\bar \nu)$ is optimal iff
\begin{equation*}
\sum_{h = 0}^d \int \bigg[ \int \ind_{C^h_\a}(z-z') (\bar \pi')^h_\a(dzdz') \bigg] m(d\a) < \infty,
\end{equation*}
being $\bar \pi' = \sum_h \int (\bar \pi')^h_\a m(d\a)$ the disintegration of $\bar \pi'$ w.r.t. $\{Z^h_\a \times \R^d\}_{h,\a}$.
\end{enumerate}
We note here that this decomposition is \emph{independent} on $\bar \pi$, because it is only based on the potentials $\bar \phi$, $\bar \psi$. Observe that the choice of the signs in \eqref{HJ_def_extphi} yields that $z$ and $z'$ are exchanged w.r.t. $x,x'$ in \eqref{E_equiv_optim}.

A family of sets $\{Z^h_\a,C^h_\a\}_{h,\a}$ satisfying the first two points above (plus some regularity properties) will be called \emph{directed locally affine partition}; the precise definition can be found in Definition \ref{D_loca_aff_part}, where a Borel dependence w.r.t. $h,\a$ is required.

While the indecomposability stated in Point \eqref{Point_indecomp_main_th} is know not to be true also in the norm cost case, the main problem we face here is that the regularity of the partition is stated in terms of the Lebesgue measure $\mathcal L^{d+1}$, and this has no direct implication on the structure of the disintegration of $\bar \mu$, being the latter supported on $\{t=1\}$. The obstacle that the "norm" $\bar{\mathtt c}$ is unbounded can be easily overtaken due to the assumptions that $\supp\,\bar \mu$ is a compact subset of $\{t=1\}$ and $\supp\,\bar{\nu}$ is a compact subset of $t=0$.\\
The first new result is thus the fact that, due to the transversality of the the cones $C^h_\a$ w.r.t. the plane $\{t=1\}$, $\cup_{h,\a} Z^h_\a \cap \{t=1\}$ is $\mathcal H^d\llcorner_{\{t=1\}}$-conegligible and the disintegration of $\mathcal H^d \llcorner_{\{t=\bar t\}}$ w.r.t. $Z^h_\a$ is \emph{regular} for all $\bar t > 0$, i.e.
\begin{equation*}
\mathcal H^d \llcorner_{\{t = \bar t\}} = \sum_{h = 0}^d \int \xi^h_\a \eta^h(d\a) \qquad \text{with} \qquad \xi^h \ll \mathcal H^h \llcorner_{Z^h_\a \cap \{t = \bar t\}}.
\end{equation*}
Note that since $C^h_\a$ is transversal to $\{t = \bar t\}$ by the definition of $\bar{\mathtt c}$, then $Z^h_\a \cap \{t = \bar t\}$ has affine dimension $h$ (and this is actually the reason for the notation). We thus obtain the first result of the paper, which is a decomposition into a directed locally affine partition which on one hand is \emph{independent} on the optimal transference plan, on the other hand it elements are not indecomposable in the sense of Point \eqref{Point_indecomp_main_th} of Theorem \ref{T_main_theor}.

\begin{theorem}
\label{T_potential_deco}
There exists a directed locally affine partition $\{Z^h_\a,C^h_\a\}_{h,\a}$ such that
\begin{enumerate}
\item \label{Point_1_pote_deco} $\mathcal H^d(\{t=1\} \setminus \cup_{h,\a} Z^h_\a) = 0$;
\item \label{Point_2_pote_deco} the disintegration of $\mathcal H^d \llcorner_{\{t = 1\}}$ w.r.t. the partition $\{Z^h_\a\}_{h,\a}$ is regular; 
\item \label{Point_3_pote_deco} $\bar \pi$ is an optimal plan iff
\begin{equation*}
\sum_h \int \ind_{C^h_\a}(z-z') \pi^h_\a(dzdz') m^h(d\a) < \infty,
\end{equation*}
where $\pi = \sum_h \int \pi^h_\a m^h(d\a)$ is the disintegration of $\pi$ w.r.t. the partition $\{Z^h_\a \times \R^{d+1}\}_{h,\a}$.
\end{enumerate}
\end{theorem}

 
Now the technique developed in \cite{biadan} can be applied to each set $Z^h_\a$ with the cost $C^h_\a$ and marginals $\bar \mu^h_\a$ and $\bar \nu^h_\a$. As it is shown in \cite{biadan} and in Example \ref{Ex_2ndmarg} of Section \ref{S_disintegration_locally_affine}, the next steps \emph{depend} on the marginal $\bar \nu^h_\a$, so that one need to fix a transference plan $\underline {\bar \pi}$ in Theorem \ref{T_main_theor}.

For simplicity in this introduction we fix the indexes $h,\a$, while in general in order to obtain a Borel construction one has to consider also the dependence $h,\a \mapsto Z^h_\a,C^h_\a$ in suitable Polish spaces.

In each $Z^h_\a$ the problem becomes thus a transportation problem with marginals $\bar \mu^h_\a$, $\underline{\bar \nu}^h_\a$ and cost $\ind_{C^h_\a}$, where $\bar \mu^h_\a,\underline{\bar \nu}^h_\a$ are the marginals of $\underline{\bar \pi}$ w.r.t. the partition $\{Z^h_\a,C^h_\a\}_{h,\a}$ (the first marginal being independent of $\underline{\bar \pi}$). The analysis of \cite{biadan} yields a decomposition of $Z^h_\a$ into locally affine sets $Z^{\ell}_{\b}$ of affine dimension $\ell + 1$, together with extremal cones $C^{\ell}_\b$ such that $\{Z^{\ell}_\b,C^{\ell}_\b\}_{\ell,\b}$ is a locally affine directed partition of $Z^h_\a$. \\
The main problem is that the \emph{regularity} of the partition refers to the measure $\mathcal H^{h+1} \llcorner_{Z^{h}_\a}$, while we need to disintegrate $\bar \mu^h_\a \ll \mathcal H^h \llcorner_{Z^h_\a \cap \{t = 1\}}$. The novelty is thus that we exploit the transversality of the cone $C^h_\a$ w.r.t. the plane $\{t = \bar t\}$ is order to deduce the regularity of the partition.

The approach is similar to the one used in the decomposition with the potentials above, and we outline below.

\subsubsection{Refined partition with cone costs}
\label{Sss_refi_cone_cost}

To avoid heavy notations, in this section we set $\breve Z = Z^h_\a$, $\breve C = C^h_\a$ and with a slight abuse of notation $\breve{\mu} = \bar \mu^h_\a$, $\breve\nu = \underline{\bar \nu}^h_\a$.

Fix a carriage
\begin{equation*}
\breve \Gamma \subset \big\{ w - w' \in \breve C \big\} \cap \big( \{t = 1\} \times \{t = 0\} \big)
\end{equation*}
of a transport plan $\breve \pi \in \Pi(\breve \mu,\breve \nu)$ of $\ind_{\breve C}$-finite cost, and let $\mathtt w_n$ be countably many points such that
\begin{equation*}
\{\mathtt w_n\}_n \subset \mathtt p_1 \Gamma \subset \clos\{\mathtt w_n\}_n,
\end{equation*}
where $\mathtt p_i$ denotes the projection on the $i$-th component of $(w,w') \in \R^h \times \R^h$, $i=1,2$.

For each $n$ define the set $H_n$ of points which can be reached from $\mathtt w_n$ with an axial path of finite cost,
\begin{equation*}
H_n := \Big\{ w: \exists I \in \N, \big\{ (w_i,w'_i) \big\}_{i = 1}^I \subset \breve \Gamma \ \big( w_1 = \mathtt w_n \ \wedge \ w_{i+1} - w'_i  \in \breve C \big) \Big\},
\end{equation*}
and let the function $\theta'$ be given by
\begin{equation*}
\theta'(w) := \sum_n 3^{-n} \chi_{H_n}(w).
\end{equation*}
Notice that $\theta'$ depends on the set $\Gamma$ and the family $\{\mathtt w_n\}_n$.

The fact that $C \cap \{t = \bar t\}$ is a compact convex set of linear dimension $h$ allows to deduce that the sets $H_n$ are of finite perimeter, more precisely the topological boundary $\partial H_n \cap \{t = \bar t\}$ is $\mathcal H^{h-1}$-locally finite, and that $\theta'$ is SBV in $\R^+ \times \R^h$.

The first novelty of the paper is to observe that we can replace $\theta'$ with two functions which make explicit use of the transversality of $C$: define indeed
\begin{equation}
\label{E_intro_def_theta}
\theta(w) := \sup \Big\{ \theta'(w'), w' \in \mathtt p_2 \Gamma \cap \{w - C\} \Big\}
\end{equation}
and let $\vartheta$ be the u.s.c. envelope of $\theta$. It is fairly easy to verify that $\theta'(w) = \theta'(w') = \theta(w) = \theta(w')$ for $(w,w') \in \breve \Gamma$ (Lemma \ref{L_prop_theta_prim}), and moreover \eqref{E_intro_def_theta} can be seen as a Lax formula for the HJ equation with Lagrangian $\ind_{C}$. \\
Again simple computations imply that $\theta$ is SBV, and moreover being each level set a union of cones it follows that $\partial \{\theta \geq \vartheta\} \cap \{t = \bar t\}$ is of locally finite $\mathcal H^{h-1}$-measure. Hence in each slice $\{t = \bar t\}$, $\vartheta > \theta$ only in $\mathcal H^h$-negligible set, and for $\vartheta$ the Lax formula becomes
\begin{equation*}
\vartheta(w) := \max \Big\{ \vartheta(w'), w' \in \mathtt p_2 \Gamma \cap \{w - C\} \Big\}.
\end{equation*}

We now start the analysis of the decomposition induce by the level sets of $\theta$ or $\vartheta$. The analysis of \cite{biadan} yields that up to a negligible set $N$ there exists a locally affine partition $\{Z^{h'}_\beta,C^{h'}_\beta\}_{h',\beta}$: the main point is the proof is to show  that the set of the so-called \emph{residual points} are $\mathcal H^h$-negligible is each plane $\{t = \bar t\}$ and that the disintegration is $\mathcal H^h \llcorner_{\{t=1\}}$-regular. Since the three functions differ only on a $\breve \mu$-negligible set, we use $\theta$ to construct the partition and $\vartheta$ for the estimate of the residual set and the disintegration: the reason is that if $(w,w') \in \breve \Gamma$ then $\theta(w) = \theta(w')$, relation which is in general false for $\vartheta$ (however they clearly differ on a $\breve \pi$-negligible set, because $\breve \mu \ll \mathcal H^h \llcorner_{\{t=1\}}$).

The strategy we use can be summarized as follows: first prove regularity results for $\vartheta$ and then deduce the same properties for $\theta$ up to a $\mathcal H^h_{\{t = \bar t\}}$-negligible set. We show how this reasoning works in order to prove that optimal rays of $\theta$ can be prolonged for $t > 1$: for $\mathcal H^h \llcorner_{\{t=1\}}$-a.e. $w$ there exist $\varepsilon > 0$ and $w'' \in w + \breve C \cap \{t = 1 + \varepsilon\}$ such that $\theta(w'') = \theta(w)$. This property is known in the case of HJ equations, see for example the analysis in \cite{biaglo:HJ2} (or the reasoning in Section \ref{Ss_back_forw_regul_phi}).

The advantage of having a Lax formula for $\vartheta$ is that for \emph{every} point $w \in \R^+ \times \R^h$ there exists at least one optimal ray connecting $w$ to $t = 0$: the proof follows closely the analysis for the HJ case. Moreover the non-degeneracy of the cone $C$ implies that it is possible to make (several) selections of the initial point $\R^+ \times \R^d \ni w \mapsto w'(w) \in \{t = 0\}$ in such a way along the optimal ray $\segment{w,w'(w)}$ the following \emph{area estimate} holds:
\begin{equation*}
\mathcal H^h(A_t) \geq \bigg( \frac{t}{\underline t} \bigg)^h \mathcal H^h(A_{\underline t}), \qquad A_t = \bigg\{ \bigg( 1 - \frac{t}{\underline t} \bigg) w + \frac{t}{\underline t} w'(w), w \in A_{\underline t} \bigg\},
\end{equation*}
where $A_{\underline t} \subset \{t = \bar t\}$ (see \cite{biaglo:HJ2,Dan:PhD} for an overview of this estimate). In particular by letting $\underline t \searrow \bar t$ one can deduce that $\mathcal H^h \llcorner_{\{t = \bar t\}}$-a.e. point $w$ belongs to a ray starting in $\{t > \bar t\}$. Since $\theta$ differs from $\vartheta$ in a $\mathcal H^h_{\{t = \bar t\}}$-negligible set, one deduce that the same property holds also for optimal rays of $\theta$.

The property that the optimal rays can be prolonged is the key point in order to show that the residual set $N$ is $\mathcal H^h \llcorner_{\{t = \bar t\}}$-negligible for all $\bar t> 0$ and that the disintegration is regular.

The technique to obtain the indecomposability of Point \eqref{Point_indecomp_main_th} is now completely similar to the approach in \cite{biadan}. For every $\breve \Gamma$, $\{\mathtt w_n\}_n$ one construct the function $\theta_{\Gamma,\mathtt w_n}$ and the equivalence relation
\begin{equation*}
E_{\Gamma,\mathtt w_n} := \big\{ \theta_{\Gamma,\mathtt w_n}(w) = \theta_{\Gamma,\mathtt w_n}(w') \big\},
\end{equation*}
then prove that there is a minimal equivalence relation $\bar E$ given again by some function $\bar \theta$, and deduce from the minimality that the sets of \emph{positive} $\breve \mu$-measure are not further decomposable. Since $\breve \mu \ll \mathcal H^h \llcorner_{\{t = 1\}}$, one can prove that Point \eqref{Point_indecomp_main_th} of Theorem \ref{T_main_theor} holds.

We thus obtain the following theorem.

\begin{theorem}
\label{T_step_theorem}
Given a directed locally affine partition $\{Z^h_\a,C^h_\a\}_{h,\a}$ and a transference plan $\underline{\bar \pi} \in \Pi(\bar \mu,\bar \nu)$ such that
\begin{equation}
\label{E_finite_cone_cost}
\bar{\underline{\pi}} = \sum_h \int \bar{\underline{\pi}}^h_\a m^h(d\a), \qquad \int \ind_{C^h_\a}(z-z') \bar{\underline{\pi}}^h_\a(dzdz') < \infty,
\end{equation}
then there exists a directed locally affine partition $\{Z^{h,\ell}_{\a,\b},C^{h,\ell}_{\a,\b}\}_{h,\a,\ell,\b}$ such that
\begin{enumerate}
\item $Z^{h,\ell}_{\a,\b} \subset Z^h_\a$ has affine dimension $\ell+1$ and $C^{h,\ell}_{\a,\b}$ is an $(\ell+1)$-dimensional extremal cone of $C^h_\a$; moreover $\aff\,Z^h_a = \aff(z + C^h_\a)$ for all $z \in Z^h_\a$;
\item $\mathcal H^d(\{t=1\} \setminus \cup_{h,\a,\ell,\b} Z^{h,\ell}_{\a,\b}) = 0$;
\item the disintegration of $\mathcal H^d \llcorner_{\{t = 1\}}$ w.r.t. the partition $\{Z^\ell_\c\}_{\ell,\c}$, $\c = (\a,\b)$, is regular, i.e.
\begin{equation*}
\mathcal H^d \llcorner_{\{t=1\}} = \sum_\ell \int \xi^\ell_\c \eta^\ell(d\c), \qquad \xi^\ell_\c \ll \mathcal H^\ell \llcorner_{Z^\ell_\c \cap \{t=1\}};
\end{equation*}
\item if $\bar \pi \in \Pi(\bar \mu,\{\underline{\bar \nu}^h_\a\})$ with $\underline{\bar \nu}^h_\a = (\mathtt p_2)_\sharp \underline{\bar \pi}^h_\a$, then $\bar \pi$ satisfies \eqref{E_finite_cone_cost} iff
\begin{equation*}
\bar \pi  = \sum_{\ell} \int \bar \pi^{\ell}_{\c} m^{\ell}(d\c), \qquad \int \ind_{C^{\ell}_{\c}}(z-z') \bar \pi^\ell_\c < \infty;
\end{equation*}
\item \label{Point_step_th_indeco} if $\ell = h$, then for every carriage $\Gamma$ of any $\bar \pi \in \Pi(\bar \mu,\{\underline{\bar \nu}^h_\a\})$ there exists a $\bar \mu$-negligible set $N$ such that each $Z^{h,h}_{\a,\b} \setminus N$ is $\ind_{C^{h,h}_{\a,\b}}$-cyclically connected.
\end{enumerate}
\end{theorem}

The proof of Theorem \ref{T_main_theor} is now accomplished by repeating the reasoning at most $d$ times as follows.

First one uses the decomposition of Theorem \ref{T_potential_deco} to get a first directed locally affine partition. \\
Then starting with the sets of maximal dimension $d$, one uses Theorem \ref{T_step_theorem} in order to obtain (countably many) indecomposable sets of affine dimension $d+1$ as in Point \eqref{Point_step_th_indeco} of Theorem \ref{T_step_theorem}. The remaining sets forms a directed locally affine partition with sets of affine dimension $h \leq d$. Note that if $C^h_\a$ is an extremal face of $\bar{\mathtt c}$ and $C^{h,\ell}_{\a,\b}$ is an extremal fact of $C^h_\a$, then clearly $C^{h,\ell}_{\a,\b}$ is an extremal face of $\bar{\mathtt c}$. \\
Applying Theorem \ref{T_step_theorem} to this remaining locally affine partition, one obtains indecomposable sets of dimension $h+1$ and a new locally affine partition made of sets with affine dimension $\leq h$, and so on.

The last step is to project the final locally affine partition $\{Z^h_\a,C^h_\a\}_{h,\a}$ of $\R^+ \times \R^d$ made of indecomposable sets (in the sense of Point \eqref{Point_step_th_indeco} of Theorem \ref{T_step_theorem}) in the original setting $\R^d$. By the definition of $\bar{\mathtt c}$ it follows that $c(-x) = \bar{\mathtt c}(1,x)$, so that any extremal cone $C^h_\a$ of $\bar{\mathtt c}$ corresponds to the extremal face $O^h_\a = - C^h_\a \cap \{t=1\}$ of $c$. Thus the family
\begin{equation*}
S^h_\a := Z^h_\a \cap \{t=1\}, \qquad O^h_\a := - C^h_\a \cap \{t=1\}
\end{equation*}
satisfies the statement, because $\bar \mu(\{t=1\}) = \bar \nu(\{t=0\}) = 1$.

\begin{remark}
\label{R_further_cases}
As a concluding remark, we observe that similar techniques work also without the assumption of superlinear growth and allowing $c$ to take infinite values. Indeed, first of all one decomposes the space $\R^d$ into indecomposable sets $S_\gamma$ w.r.t. the convex cost
\begin{equation*}
C := \clos\, \{c < \infty\},
\end{equation*}
using the analysis on the cone cost case. Notice that since w.l.o.g. $C$ has dimension $d$, this partition is countable.

Next in each of these sets one studies the transportation problem with cost $c$. Using the fact that these sets are essentially cyclically connected for all carriages $\Gamma$, then one deduces that there exist potentials $\phi_\beta,\psi_\beta$, and then the proof outlined above can start.

The fact that the intersection of $C$ (or of the cones $C^h_\a$) is not compact in $\{t = \bar t\}$ can be replaced by the compactness of the support of $\mu,\nu$, while the regularity of the functions $\theta'$, $\theta$ and $\vartheta$ depends only on the fact that $C \cap \{t = \bar t\}$ is a convex closed set of dimension $d$ (or $h$ for $C^h_\a$).
\end{remark}

%
%
%

\subsection{Structure of the paper}
\label{Ss_struct}

The paper is organized as follows.

In Section \ref{S_setting} we introduce some notations and tools we use in the next sections. Apart from standard functional spaces, we recall some definitions regarding multifunctions and linear/affine subspaces, adapted to our setting. Finally some basic notions on optimal transportation are presented.

In Section \ref{S_directed_locally_affine_partitions} we state the fundamental definition of \emph{directed locally affine partition} $\mathbf D = \{Z^h_\a,C^h_\a\}_{\nfrac{h = 0,\dots,d}{\a \in A^h}}$: this definition is the natural adaptation of the same definition in \cite{biadan}, with minor variation due to the presence of the preferential direction $t$. Proposition \ref{P_countable_partition_in_reference_directed_planes} shows how to decompose $\mathbf D$ into a countable disjoint union of directed locally affine partitions $\mathbf D(h,n)$ such that
\begin{list}{-}{2pt}
\item the sets $Z^h_\a$ in $\mathbf D(h,n)$ have fixed affine dimension,
\item the sets $Z^h_\a$ are almost parallel to a given $h$-dimensional plane $V^h_n$,
\item their projections on $V^h_n$ contain a given $h$-dimensional cube,
\item the projection of $C^h_\a$ on $V^h_n$ is close a given cone $C^h_n$.
\end{list}
As in \cite{biadan} the sets $\mathbf D(h,k)$ are called \emph{sheaf sets} (Definition \ref{D_shaef_set}).

As we said in the introduction, the line of the proof is to refine a directed locally affine partition in order to obtain either indecomposable sets or diminish their dimension by at least $1$: in Section \ref{S_fdlap} we show how the potentials $\bar \phi$, $\bar \psi$ can be used to construct a first directed locally affine partition. The approach is to associate forward and backward optimal rays to each point in $\R^+ \times \R^d$, and then define the \emph{forward/backward regular and regular transport set}: the precise definition is given in Definition \ref{def_regular_points}, we just want to observe that the regular points are in some sense \emph{generic}. After proving some regularity properties, Theorems \ref{T_decomp_phi_R-}, \ref{T_decomp_phi_R+} and Proposition \ref{P_equal_R-R+} show how to construct a directed locally affine partition $\mathbf D_{\bar \phi} = \{Z^h_\a,C^h_\a\}_{h,\a}$, formula \eqref{dlap_phi}. \\
The second part of the section proves that the partition induced by $Z^h_\a$ covers all $\{t=1\}$ up to a $\mathcal H^d$-negligible set and that the disintegration of $\mathcal H^d$ w.r.t. $Z^h_\a$ is regular. Here we need to refine the approach of \cite{Dan:PhD}, which gives only the regularity of the disintegration of $\mathcal L^{d+1} \llcorner_{\{t > 0\}}$. Proposition \ref{P_back_regul_phi} shows that $\mathcal H^d\llcorner_{\{t=1\}}$-a.e. point $z$ belongs to some $Z^h_\a$ (i.e. it is regular), and Proposition \ref{P_ac_disi} completes the analysis proving that the conditional probabilities of the disintegration of $\mathcal H^d \llcorner_{\{t=1\}}$ are a.c. with respect to $\mathcal H^h \llcorner_{Z^h_\a \cap \{t=1\}}$.

The next four sections describe how the iterative step work: given a directed locally affine partition $\mathbf D = \{Z^h_\a,C^h_\a\}_{h,\a}$ such that the disintegration of $\mathcal H^d \llcorner_{\{t=1\}}$ is regular, obtain a refined locally affine partition $\mathbf D' = \{Z^{h,\ell}_{\a,\b},C^{h,\ell}_{\a,\b}\}_{h,\ell,\a,\b}$, again with a regular disintegration and such that $Z^{h,\ell}_{\a,\b} \subset Z^h_\a$ for some $h \geq \ell$, but such that the sets of maximal dimension $h = \ell$ are indecomposable in the sense of Point \eqref{Point_indecomp_main_th} of Theorem \ref{T_main_theor}.

First of all, in Section \ref{S_disintegration_locally_affine} we define the notion of optimal transportation problems in a sheaf set $\{Z^h_\a,C^h_\a\}_\a$, with $h$-fixed: the key point is that the transport can occur only along the directions in the cone $C^h_\a$, see the transport cost \eqref{E_mathtt_c_mathbf_Z}. For the directed locally affine partition obtained from $\bar \phi$, this property is equivalent to the optimality of the transport plan. We report a simple example which shows why from this point onward we need to fix a transference plan, Example \ref{Ex_2ndmarg}. \\
The fact that the elements of a sheaf set are almost parallel to a given plane makes natural to map them into \emph{fibration}, which essentially a sheaf set whose elements $Z^h_\a$ are parallel. This is done in Section \ref{Ss_mapping_sheaf_to_fibration}, and Proposition \ref{P_equivalence_transference_sheaf_fibration} shown the equivalence of the transference problems.

The proof outlined in Section \ref{Sss_refi_cone_cost} is developed starting from Section \ref{S_monotone_relation_on_fibration}. For any fixed carriage $\tilde \Gamma \subset \{t=1\} \times \{t=0\}$ we construct in Section \ref{Ss_linear_prorder_fibration} first the family of sets $H_n$, and then the partition functions $\theta'$, $\theta$: the properties we needs (mainly the regularity of the level sets) are proved in Section \ref{Sss_constr_linear_preord}. In Section \ref{Sss_constr_family_linear_preord} we show how by varying $\tilde \Gamma$ we obtain a family of equivalence relations (whose elements are the level sets of $\theta$) closed under countable intersections. \\
The next section (Section \ref{Ss_properties_minimal_equivalence_relation}) uses the techniques developed in \cite{BiaCar} in order to get a minimal equivalence relation: the conclusion is that there exists a function $\bar \theta$, constructed with a particular carriage $\bar{\tilde \Gamma}$, which is finer that all other partitions, up to a $\bar \mu$-negligible set. The final example (Example \ref{Ex_no_cycl_conn_inner}) address a technical point: it shows that differently from \cite{biadan}, it is not possible to identify the sets of cyclically connected points with the Lebesgue points of the equivalence classes.

Section \ref{S_decomposition_of_fibration} strictly follows the approach of \cite{biadan} in order to obtain from the fibration a refined locally affine partition. Roughly speaking the construction is very similar to the construction with the potential $\bar \phi$: one defines the optimal directions and the regular points more or less as in the potential case. After listing the necessary regularity properties of the objects introduced at the beginning of this section, in Section \ref{Ss_partition_transport_set} we give the analogous partition function of the potential case and obtain the refined locally affine partition $\tilde{\mathbf D}' = \{Z^{h,\ell}_{\a,\b},C^{h,\ell}_{\a,\b}\}_{h,\ell,\a,\b}$.

Section \ref{S_disintegration_on_affine} addresses the regularity problem of the disintegration. As said in the introduction, the main idea is to replace $\bar \theta$ with its u.s.c. envelope $\bar \vartheta$, which has the property that its optimal rays reach $t=0$ for \emph{all} point in $\R^+ \times \R^d$. A slight variation of the approach used with the potential $\bar \phi$ gives that $\mathcal H^d \llcorner_{\{t = \bar t\}}$-a.e. point is regular (Proposition \ref{P_back_regul_vartheta}) for the directed locally affine partition given by $\bar \vartheta$. Using the fact that $\bar \theta = \bar \vartheta$ $\mathcal H^h \llcorner_{\{t=\bar t\}}$-a.e., one obtains the regularity of $\mathcal H^d \llcorner_{\{t = \bar t\}}$-a.e. point for the directed locally affine partition induced by $\bar \theta$ (Corollary \ref{P_back_regul_theta}). The area estimate for optimal rays of $\bar \vartheta$ (Lemma \ref{L_inner_area_estimate}) allows with an easy argument to prove the regularity of the disintegration, Proposition \ref{P_ac_disi_bar_theta}. 

The final Section \ref{S_transl_orig} explains how the steps outlined in the last four sections can be used in order to obtain the proof of Theorem \ref{T_main_theor}.

Finally in Appendix \ref{A_appendix_1} we recall the result of \cite{biacar:cmono} concerning linear preorders and the existence of minimal equivalence relations and their application to optimal transference problems.

\section{General notations and definitions}
\label{S_setting}

As standard notation, we will write $\N$ for the natural numbers, $\N_0 = \N \cup \{0\}$, $\Q$ for the rational numbers, $\R$ for the real numbers. The set of positive rational and real numbers will be denoted by $\Q^+$ and $\R^+$ respectively. To avoid the analysis of different cases when parameters are in $\R$ or $\N$, we set $\R^0 := \N$.
The first infinite ordinal number will be denoted by $\omega$, and the first uncountable ordinal number is denoted by $\Omega$.

The $d$-dimensional real vector space will be denoted by $\R^d$. The euclidian norm in $\R^d$ will be denoted by $|\cdot|$. For every $k\leq d$, the open unit ball in $\RR^h$ with center $z$ and radius $r$ will be denoted with $B(z,r)$ and for every $x\in\R^h$, $\bar t \geq 0$,  $B^h(\bar t,x,r):= B(\bar t,x,r)\cap \{t=\bar t\}$. 

Moreover, for every $a,b\in \RR^d$ define the close segment, the open segment, and the section at $t = \bar t$ respectively as :
\[ \segment{a,b}:= \{ \lambda a+ (1-\lambda) b: \lambda\in[0,1]  \},\quad \rrbracket a,b\llbracket:= \{ \lambda a+ (1-\lambda) b: \lambda\in]0,1[  \}, \quad \segment{a,b}(\bar t):= \segment{a,b}\cap \{ t =\bar t\}.\]

The closure of a set $A$ in a topological space $X$ will be written $\clos\,A$, and its interior by $\inter\,A$. If $A \subset Y \subset X$, then the relative interior of $A$ in $Y$ is $\interr A$: in general the space $Y$ will be clear from the context. The topological boundary of a set $A$ will be denoted by $\partial A$, and the relative boundary is $\partial_\mathrm{rel} A$. The space $Y$ will be clear from the context.

If $A$, $A'$ are subset of a real vector space, we will write
\begin{equation*}
A + A' := \big\{ z + z', z \in A, z' \in A' \big\}.
\end{equation*}
If $T \subset \R$, then we will write
\begin{equation*}
T A := \big\{ t z, t \in T, z \in A \big\}.
\end{equation*}

The convex envelope of a set $A \subset \RR^d$ will be denoted by $\conv\,A$. If $A \subset \RR^d $, its  convex direction envelope is defined as
\begin{equation*}
	\convd A := \{ t=1\} \cap \big( \R^+ \cdot \conv\,A \big).
\end{equation*}

If $x \in \prod_i X_i$, where $\prod_i X_i$ is the product space of the spaces $X_i$, we will denote the projection on the $\bar i$-component as $\mathtt p_{\bar i} x$ or $\mathtt p_{x_{\bar i}} x$: in general no ambiguity will occur. Similarly we will denote the projection of a set $A \subset \prod_i X_i$ as $\mathtt p_{\bar i} A$, $\mathtt p_{x_{\bar i}} A$. In particular for every $\bar t\geq 0$ and $x\in \R^d$, $\prj_t (\bar t, x):= x$.


\subsection{Functions and multifunctions}
\label{Ss_souslin_multifunction}

A multifunction $\mathbf f$ will be considered as a subset of $X \times Y$, and we will write
\[
	\mathbf f(x) = \big\{ y \in Y: (x,y) \in \mathbf f \big\}.
\]
The inverse will be denoted by
\begin{equation*}
\mathbf f^{-1} = \big\{ (y,x) : (x,y) \in \mathbf f \big\}.
\end{equation*}

With the same spirit, we will not distinguish between a function $\mathtt f$ and its graph $\Graph\, \mathtt f$, in particular we say that the function $\mathtt f$ is \emph{$\sigma$-continuous} if $\Graph\,\mathtt f$ is $\sigma$-compact. Note that we do not require that its domain is the entire space.

If $\mathtt f$, $\mathtt g$ are two functions, their composition will be denoted by $\mathtt g \circ \mathtt f$.

The epigraph of a function $\mathtt f : X \to \R$ will be denoted by
\begin{equation*}
\epi\,\mathtt f := \big\{ (x,t) : \mathtt f(x) \leq t \big\}.
\end{equation*}

The identity map will be written as $\mathbb I$, the characteristic function of a set $A$ will be denoted by
\begin{equation*}
\chi_A(x) :=
\begin{cases}
1 & x \in A, \crcr
0 & x \notin A,
\end{cases}
\end{equation*}
and the indicator function of a set $A$ is defined by
\begin{equation*}
\ind_A(x) := \begin{cases}
             0 & x \in A, \crcr
	     \infty & x \notin A.
             \end{cases}
\end{equation*}

\subsection{Affine subspaces and cones}
\label{Ss_intro_affine_subspaces_cones}

We now introduce some spaces needed in the next sections: we will consider these spaces with the topology given by the Hausdorff distance $\mathtt d_{\mathrm H}$ of their elements in every closed ball $\clos\, B(0,r)$ of $\R^d$, i.e.
\[
\mathtt d(A,A') := \sum_n 2^{-n} \disth \big( A \cap B(0,n), A' \cap B(0,n) \big).
\]
for two generic elements $A$, $A'$.

We will denote points in $\RR^d$ as $z = (t,x)$.

For $h,h',d \in \N_0$, $h' \leq h \leq d$, define $\mathcal G(h,\RR^d)$ to be the set of $(h+1)$-dimensional subspaces of $\RR^d$ such that their slice at $t=1$ is a $h$-dimensional subspace of $\{t=1\}$, and let $\mathcal A(h,\RR^d)$ be the set of $(h+1)$-dimensional affine subspaces of $\RR^d$ such that their slice at $t=1$ is a $h$-dimensional affine subspace of $\{t=1\}$. If $V \in \mathcal A(h,\RR^d)$, we define $\mathcal A(h',V) \subset \mathcal A(h',\RR^d)$ as the $(h'+1)$-dimensional affine subspaces of $V$ such that their slice at time $t=1$ is a $h'$-dimensional affine subset.

We define the projection on $A \in \mhA(h,[0,+\infty) \times \R^d)$ with $\bar t$ fixed as $\tpt_A$:
\[
\tpt_A (\bar t, x) = \big( \bar t, \mathtt p_{A \cap \{t=\bar t\}} x \big).
\]

If $A \subset \RR^d$, then define its affine span as
\begin{equation*}
\aff\,A := \bigg\{ \sum_i t_i z_i, i \in \N, t_i \in \R, z_i \in A, \sum_i t_i = 1 \bigg\}.
\end{equation*}
The \emph{linear dimension} of the set $\aff\,A \subset \RR^d$ is denoted by $\dim\,A$. The orthogonal space to $\mathrm{span}\, A := \aff(A \cup \{0\})$ will be denoted by $A^\perp$.
For brevity, in the following the dimension of $\aff A\cap\{t =\bar t\}$ will be called \emph{dimension at time $\bar t$} (or if there is no ambiguity \emph{time fixed dimension}) and denoted by $\dim_{\bar t} A$.

Let $\mathcal C(h,\RR^d)$ be the set of closed convex non degenerate cones in $\RR^d$ with vertex in $(0,0)$ and dimension $h+1$: non degenerate means that their linear dimension is $h+1$ and their intersection with $\{ t=1 \}$ is a compact convex set of dimension $h$. Note that if $C \in \mhC(h,\tRd)$, then $\aff\,C \in \mhA(h,\tRd)$ and conversely if $\aff\,C \in \mhA(h,\tRd)$ and $C\cap\{t=1\}$ is bounded then $C \in \mhC(h,\tRd)$. 

Set also for $C \in \mathcal C(h,\RR^d)$
\begin{equation*}
	\mathcal D C := C \cap \{t=1\},
\end{equation*}
and
	\begin{align*}
		\mathcal{DC}(h,\RR^d) :=&~ \big\{ \mathcal D C : C \in \mathcal C(h,\RR^d) \} \crcr
							=&~ \big \{ K\subset \{t=1\} : K\textrm{ is convex and compact}\big\}.
	\end{align*}
The latter set is the set of \emph{directions} of the cones $C \in \mathcal C(h,\RR^d)$. We will also write for $V \in \mathcal G(h,\RR^d)$
\begin{equation*}
\mathcal C(h',V) := \big\{ C \in \mathcal C(h',\RR^d): \aff\,C  \subset V \big\}, \quad \mathcal{DC}(h,V) := \big\{ \mathcal D C : C \in \mathcal C(h,V) \}.
\end{equation*}

Define $\mathcal K(h)$ as the set of all $h$-dimensional compact and convex subset of $\{t=1\}$.
If $K\in\mathcal K(h)$, set the \emph{open set}
\begin{equation}
	\label{E_epsilon_neigh_of_Cka_section}
	\mathring K(r) := \big( K + B^{d+1}(0,r) \big) \cap \aff\,K .
\end{equation}

\noindent Define $K(r) := \clos\,K(r) \in \mathcal K(h)$. Notice that $K = \cap_n \mathring K(2^{-n})$. 

For $r < 0$ we also define the open set
\begin{equation}
	\label{E_inverse_neigh_Cone_section}
	\mathring K(-r) := \Big\{ z \in \{t=1\} : \exists \epsilon > 0 \ \big( B^{d+1}(z,r+\epsilon) \cap \aff\,K \subset K \big) \Big\},
\end{equation}
so that $\interr K = \cup_n \mathring K(-2^{-n})$: as before $K(-r) := \clos\,\mathring K(-r) \in \mathcal K(h,\RR^d)$ for $0 < -r \ll 1$. 

If $V$ is a $h$-dimensional subspace of $\{t=1\}$, $K\in\mathcal K(h)$ such that $K\subset V$ and given two real numbers $r,\lambda > 0$, consider the subsets $L_{\rm d}(h,K,r,\lambda)$ of $\mathcal K$ defined by
\begin{align}
\label{E_L_V_r_def_section}
	L_{\rm d}(h,K,r,\lambda) := \Big\{ K' \in \mathcal K(h) : \ 	(i)&~ K(-r) \subset \mathtt p_V \mathring K', \crcr
													(ii)&~ \mathtt p_V K' \subset \mathring K, \crcr
													(iii)&~ \disth(\mathtt p_V K',K') < \lambda \Big\}.
\end{align}
The subscript $\rm d$ refers to the fact that we are working in $\{t=1\} \times \R^d$.

\begin{figure}
	\centering
		\subfloat[][The sets $\mathring{K}(-r),\mathring K(r)$ for a given $K \in \mathcal K(h)$.]
			{\includegraphics[width=.45\columnwidth]{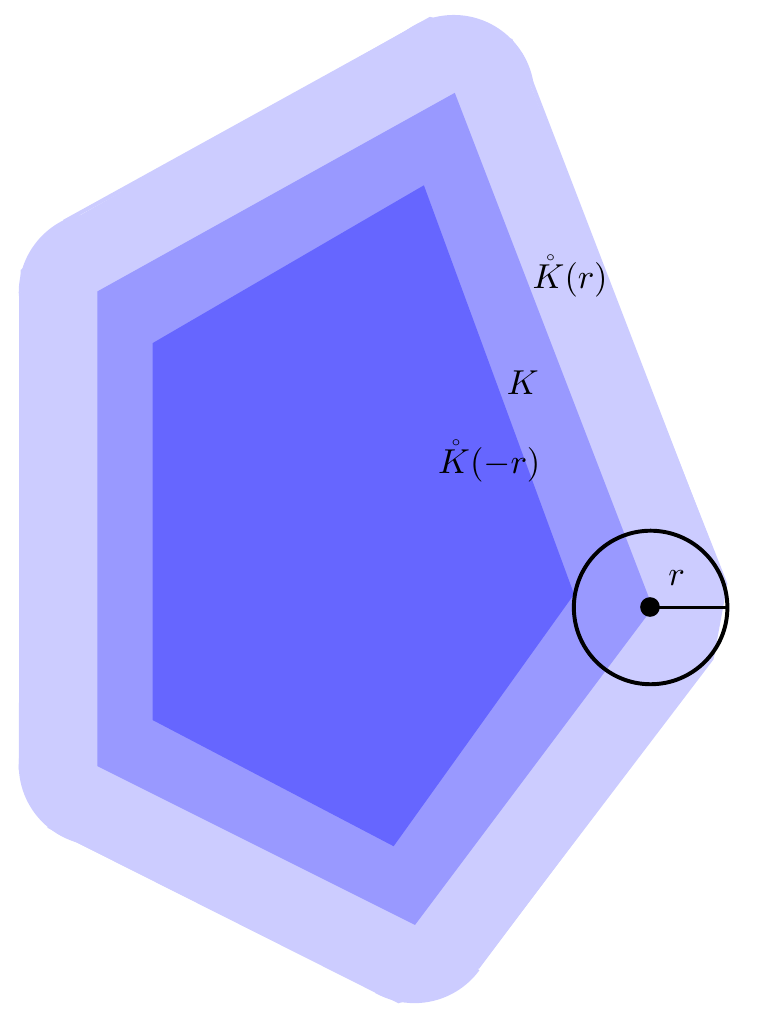}} \quad
		\subfloat[][The set $L_{\rm d}(h,K,r,\lambda)$.]
			{\includegraphics[width=.45\columnwidth]{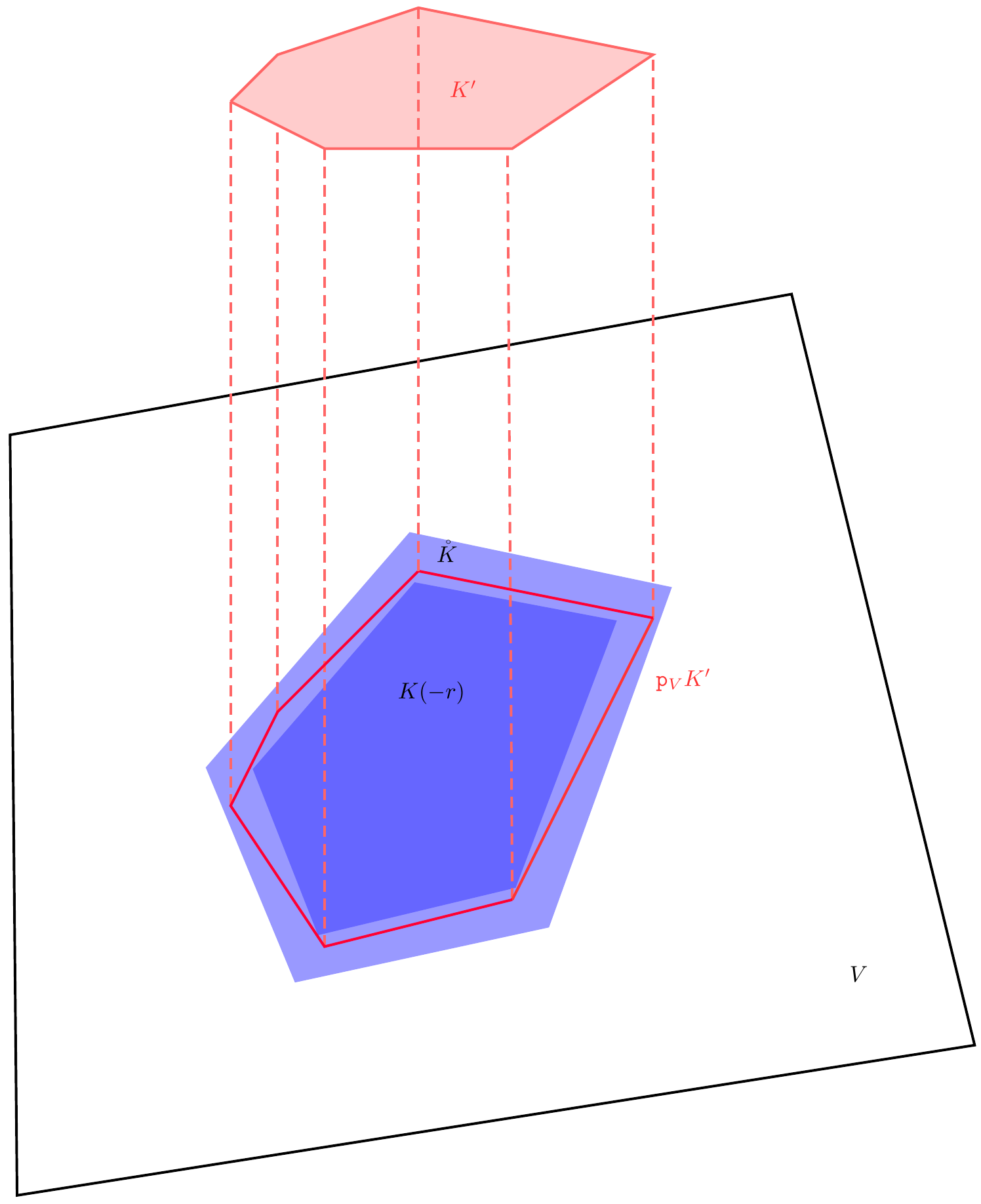}}
		\caption{The sets defined in \eqref{E_epsilon_neigh_of_Cka_section}, \eqref{E_inverse_neigh_Cone_section} and \eqref{E_L_V_r_def_section}.}
		\label{fig:subfig}
\end{figure}

Recall that according to the definition of $C\in\mathcal C(h,\RR^d)$, $C\cap\{t=1\}$ is compact. Define
\begin{equation*}
	L(h,C,r,\lambda) := \Big\{ C\in \mathcal C(h,\RR^d): C\cap\{t=1\} \in L_{\rm d}(h,K,r,\lambda) \Big\}.
\end{equation*}
It is fairly easy to see that for all $r,\lambda > 0$ the family
\begin{equation}
	\label{E_base_of_C_k_R_ell_section}
	\mathfrak L(h,r,\lambda) := \Big\{ L(h,C,r',\lambda'), C \in \mathcal C(h,\RR^d), 0 < r' < r, 0 < \lambda' < \lambda \Big\}
\end{equation}
generates a prebase of neighborhoods of $\mathcal C(h,\RR^d)$. In particular, being the latter separable, we can find countably many sets $L(h,C_n,r_n,\lambda_n)$, $n \in \N$, covering $\mathcal C(h,\RR^d)$, and such that $\big(C_n\cap\{t=1\}\big)(-r_n) \in \mathcal K(h)$. 

Let $C\in\mathcal C(h, \RR^d)$ and $r>0$.
For simplicity, we define 
\begin{equation*}
	\mathring C(r) := \{0\} \cup \R^+ \cdot \Big( \big( \mathcal D C + B^{d+1}(0,r) \big) \cap \aff\,\mathcal D C \Big),
\end{equation*}
\begin{equation*}
	\mathring C(-r) := \{0\} \cup \R^+ \cdot \Big\{ z \in \{t=1\} : \exists \epsilon > 0 \ \big( B^{d+1}(z,r+\epsilon) \cap \aff\,\mathcal D C \subset \mathcal D C \big) \Big\},
\end{equation*}
\[C(r) := \clos\,C(r) \quad \textrm{and}\quad C(-r) := \clos\,C(-r). \]
Notice that \eqref{E_base_of_C_k_R_ell_section} can be rewritten using these new definitions.


\subsection{Partitions}
\label{Ss_partitions_intro}

We say that a subset $Z \subset \RR^d$ is \emph{locally affine} if there exists $h \in \{0,\dots,d\}$ and $V \in \mathcal A(h,\RR^d)$ such that $Z \subset V$ and $Z$ is relatively open in $V$, i.e. $\interr Z \not= \emptyset$. Notice that we are not considering here $0$-dimensional sets (points), because we will not use them in the following.

A \emph{partition} in $\RR^d$ is a family $\mathcal Z = \{Z_{\mathfrak a}\}_{\mathfrak a \in \mathfrak A}$ of disjoint subsets of $\RR^d$. We do not require that $\mathcal Z$ is a \emph{covering} of $\RR^d$, i.e. $\cup_\mathfrak a Z_\mathfrak a = \RR^d$.

A \emph{locally affine partition} $\mathcal Z = \{Z_{\mathfrak a}\}_{\mathfrak a \in \mathfrak A}$ is a partition such that each $Z_{\mathfrak a}$ is locally affine. We will often write
\[
\mathcal Z = \bigcup_{k=0}^d \mathcal Z^h, \quad \mathcal Z^h = \big\{ Z_\mathfrak a, \mathfrak a \in \mathfrak A: \dim\,Z_\mathfrak a = h + 1 \big\},
\]
and to specify the dimension of $Z_\mathfrak a$ we will add the superscript $(\dim\,Z_\mathfrak a - 1)$: thus, the sets in $\mathcal Z^h$ are written as $Z^h_\mathfrak a$, and $\mathfrak a$ varies in some set of indexes $\mathfrak A^{d-h}$ (the reason of this notation will be clear in the following. In particular $\mathfrak A^{d-h}\subseteq\R^{d-h}$). 

\subsection{Measures and transference plans}
\label{Ss_measure_transference}

We will denote the Lebesgue measure of $\RR^d$ as $\mathcal L^{d+1}$, and the $k$-dimensional Hausdorff measure on an affine $k$-dimensional subspace $V$ as $\mathcal H^h \llcorner_V$. In general, the restriction of a function/measure to a set $A \in \RR^d$ will be denoted by the symbol $\llcorner_A$ following the function/measure.

The Lebesgue points $\mathrm{Leb}(A)$ of a set $A \subset \RR^d$ are the points $z \in A$ such that
\begin{equation*}
\lim_{r \to 0} \frac{\mathcal L^{d+1}(A \cap B(z,r))}{\mathcal L^{d+1}(B(z,r))} = 1.
\end{equation*}
If $\varpi$ is a locally bounded Borel measure on $\RR^d$, we will write $\varpi \ll \mathcal L^{d+1}$ if $\varpi$ is a.c. w.r.t. $\mathcal L^{d+1}$, and we say that $z$ is a \emph{Lebesgue point of $\varpi \ll \mathcal L^{d+1}$} if
\begin{equation*}
\mathtt f(z) > 0 \quad \wedge \quad \lim_{r \to 0} \frac{1}{\mathcal L^{d+1}(B(z,r))} \int_{B(z,r)} \big| \mathtt f(z') - \mathtt f(z) \big| \mathcal L^{d+1}(dz') = 0,
\end{equation*}
where we denote by $\mathtt f$ the Radon-Nikodym derivative of $\varpi$ w.r.t. $\mathcal L^{d+1}$, i.e. $\varpi = \mathtt f \mathcal L^{d+1}$.
We will denote this set by $\mathrm{Leb}\,\varpi$.

The set of probability measure on a measurable space $X$ will be denoted by $\mathcal P(X)$. In general the $\sigma$-algebra is clear from the context.

If $\varpi$ is a measure on the measurable space $X$ and $\mathtt f : X \mapsto Y$ is a $\varpi$-measurable map, then we will denote the \emph{push-forward of $\varpi$} by $\mathtt f_\sharp \varpi$.

We will also use the following notation: for a generic Polish space $(X,\mathtt d)$, measures $\mu,\nu \in \mathcal P(X)$ and Borel cost function $\mathtt c : X \times X \to [0,\infty]$ we set
\begin{equation*}
\Pi(\mu,\nu) := \Big\{ \pi \in \mathcal P(X \times X) : (\mathtt p_1)_\sharp \pi = \mu, (\mathtt p_2)_\sharp \pi = \nu \Big\}.
\end{equation*}
\begin{equation*}
\Pi^f_{\mathtt c}(\mu,\nu) := \bigg\{ \pi \in \Pi(\mu,\nu) : \int_{X \times X} \mathtt c \pi < +\infty \bigg\},
\end{equation*}
\begin{equation*}
\Pi^{\mathrm{opt}}_{\mathtt c}(\mu,\nu) := \bigg\{ \pi \in \Pi(\mu,\nu) : \int_{X \times X} \mathtt c \pi = \inf_{\pi' \in \Pi(\mu,\nu)} \int_{X \times X} \mathtt c \pi' \bigg\}.
\end{equation*}

If $\Gamma \subset X \times X$, then an \emph{axial path with base points $(z_i,z_i') \in \Gamma$, $i=1,\dots,I$} starting in $z = z_1$ and ending in $z''$ is the sequence of points
\begin{equation}
\label{axial_path_explicit}
	(z,z'_1) = (z_1,z'_1),(z_2,z'_1),\dots,(z_i,z'_{i-1}),(z_i,z'_i),(z_{i+1},z'_i),\dots,(z_I,z'_I),(z'',z'_I).
\end{equation}
We will say that the axial path \emph{connects} $z$ to $z''$: note that $z \in \mathtt p_1 \Gamma$. A \emph{closed axial path or cycle} is an axial path such that $z=z''$. The axial path has finite cost if it is contained in $\{\mathtt c < \infty\}$.

We say that $A \subset X$ is \emph{$(\Gamma,\mathtt c)$-cyclically connected} if for any $z,z'' \in A$ there exists an axial path with finite cost connecting $z$ to $z''$: equivalently we can say that there exists a closed axial path whose projection on $X$ contains $z$, $z''$. From the definition it follows that $A \subset \mathtt p_1 \Gamma$.

The \emph{Souslin sets $\Sigma^1_1$} of a Polish space $X$ are the projections of the Borel sets of $X \times X$. The $\sigma$-algebra generated by the Souslin sets will be denoted by $\varTheta$.

\section{Directed locally affine partitions}
\label{S_directed_locally_affine_partitions}

The key element in our proof is the definition of locally affine partition: this definition is not exactly the one given given in \cite{biadan} because we require that if the cone has linear dimension $h+1$, then its intersection with $t=1$ is a compact convex set of linear dimension $h$.

%
%

\begin{definition}
\label{D_loca_aff_part}
A \emph{directed locally affine partition} in $\RR^d$ is a partition into locally affine sets $\{Z^h_\mathfrak a\}_{\nfrac{h = 0,\dots,{d}}{\mathfrak a \in \mathfrak A^{d-h}}}$, $Z^h_\a  \subset \RR^d$ and $\mathfrak A^{d-h} \subset \R^{d-h}$, together with a map
\begin{equation*}
\mathbf d : \bigcup_{h=0}^{d} \{h\} \times \mathfrak A^{d-h} \to \bigcup_{h=0}^{d} \mhC(h,\RR^{d})
\end{equation*}
satisfying the following properties:
\begin{enumerate}
\item \label{Point_1_graph_mathcal_D_sigma_compact} the set
\begin{equation*}
\mathbf D = \bigg\{ \big( h,\mathfrak a,z,\mathbf d(h,\mathfrak a) \big): k \in \{0,\dots,d\}, \a \in \A^{d-h}, z \in Z^h_\mathfrak a \bigg\}, 
\end{equation*}
is $\sigma$-compact in $\cup_h \big( \{h\} \times \R^{d-h} \times (\RR^d) \times \mhC(h,\RR^{d})\big)$, i.e. there exists a family of compact sets
\[
	K_n \subset \bigcup_h \big( \{h\} \times \R^{d-h} \times (\RR^d)  \big)
\]
such that $Z^h_\a  \cap \mathtt p_z K_n$ is compact and
\[
\mathtt p_{h,\a} K_n \ni (h,\mathfrak a) \mapsto (Z^h_\mathfrak a \cap \mathtt p_z K_n,\mathbf d(h,\mathfrak a)) 
\]
is continuous w.r.t. the Hausdorff topology;


\item denoting $C^h_\mathfrak a := \mathbf d(h,\mathfrak a)$, then
\[
\forall z \in Z^h_\mathfrak a \ \Big( \aff\, Z^h_\mathfrak a = \aff (z + C^h_\mathfrak a) \Big);
\]

\item \label{Point_3_def_direc_part} the plane $\aff\, Z^h_\a $ satisfies $\aff\, Z^h_\a  \in \mhA(h,\RR^d)$.

\end{enumerate}
\end{definition}

\begin{remark}
\label{R_3rd_point_nondeg}
Using the fact that $C^h_\a$ is not degenerate, one sees immediately that Point \ref{Point_3_def_direc_part} is unnecessary.
\end{remark}

The map $\mathbf d$ will be called \emph{direction map} of the partition, or \emph{direction vector field} for $h=0$. Sometimes in the following we will write
\[
\mathbf d(z) = \mathbf d(h,\mathfrak a) \quad \text{for} \quad z \in Z^h_\mathfrak a,
\]
being $Z^h_\mathfrak a$ a partition, or we will use also the notation $\{Z^h_\mathfrak a,C^h_\mathfrak a\}_{h,\mathfrak a}$. For shortness we will write 
\begin{equation}
\label{E_mathbf_Z_base_partition}
\mathbf Z^h := \mathtt p_z \mathbf D(h) = \bigcup_{\mathfrak a \in \mathfrak A_h} Z^h_\mathfrak a, \qquad \mathbf Z := \mathtt p_z \mathbf D = \bigcup_{h=0}^{d} \mathbf Z^h = \bigcup_{h=0}^{d} \bigcup_{\mathfrak a \in \mathfrak A_h} Z^h_\mathfrak a.
\end{equation}


For each $C^h_n \in \mhC(h,\RR^{d})$,  (the index $n$ is because of the proposition below), consider a family $\mathrm e^h(n)$ of vectors $\{\mathrm e^h_i(n), i =0,\dots,h\}$ in $\R^d$ such that
\begin{equation}
\label{E_span_mathrm_e_h_i_n}
\mathcal C(h,\RR^d) \ni C(\{\mathrm e^h(n)\}) := \bigg\{ \sum_{i=0}^h t_i (1,\mathrm e^h_i(n)), t_i \in [0,\infty) \bigg\} \subset \mathring C^h_n(-r_n).
\end{equation}
Define also
\begin{equation}
\label{E_U_square_mathrm_e}
U(\{\mathrm e^h(n)\}) := \{t=0\} \times \conv\,\mathrm e^h(n). 
\end{equation}

Note that
\[
\aff\, \big\{ (0,0),(1,\mathrm e^h_0(n)),\dots,(1,\mathrm e^h_h(n)) \big\} \in \mhA(h,\tRd),
\]
so that $C(\{\mathrm e^h(n)\}) \in \mhC(h,\RR^{d})$.

\begin{figure}
	\centering
	\includegraphics[scale=0.85]{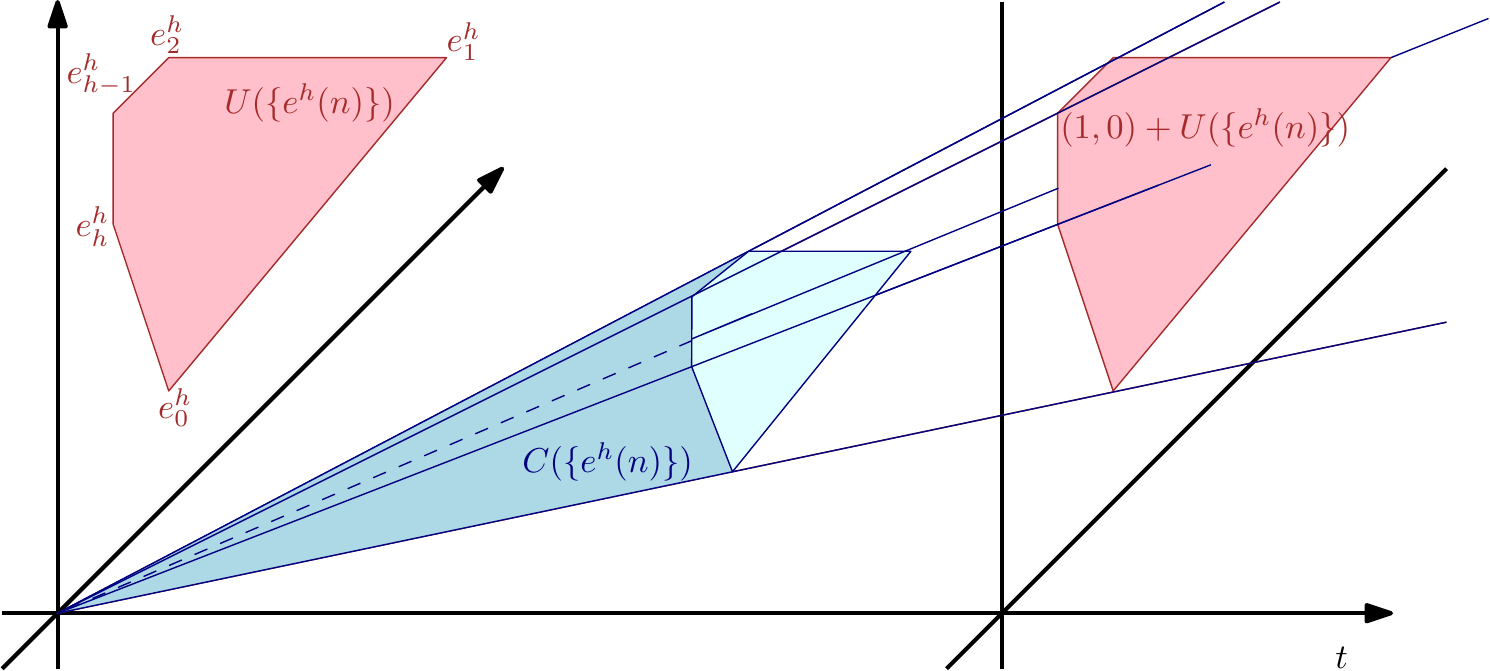}
	\caption{Definition of $C(\{\mathrm e^h(n)\})$ and $U(\{\mathrm e^h(n)\})$, formulas \eqref{E_span_mathrm_e_h_i_n} and \eqref{E_U_square_mathrm_e}.}
	\label{U_C_picture}
\end{figure}


The following proposition is the adaptation of Proposition 3.15 of \cite{biadan} to the present situation.


\begin{proposition}
\label{P_countable_partition_in_reference_directed_planes}
There exists a countable covering of $\mathbf D$ into disjoint $\sigma$-compact sets $\mathbf D(h,n)$, $h = 0,\dots,d$ and $n \in \N$, with the following properties: there exist
\begin{itemize}
\item vectors $\{\mathrm e^h_i(n)\}_{i=0}^h \subset \mathbb R^d$, with linear span
\[
V^h_n = \mathrm{span} \big\{ (1,\mathrm e_0^h(n)),\dots,(1,\mathrm e_h^h(n)) \big\} \in \mathcal A(h,\RR^d),
\]
\item a cone $C^h_n \in \mhC(h,V^h_n)$,
\item a given point $z^h_n \in V^h_n $,
\item constants $r^h_n,\lambda^h_n \in (0,\infty)$,
\end{itemize}
such that, setting
\[
\mathfrak A^h_n := \mathtt p_\mathfrak a \mathbf D(h,n),  \qquad  C^h_\a = \mathtt p_{\mathcal C(h,[0,\infty] \times \R^d)} \bD(h,n)(\a),
\]
it holds:
\begin{enumerate}
\item \label{Point_0_Proposition_P_countable_partition} $\mathtt p_{\{0,\dots,d\}} \mathbf D(h,n) =\{h\}$ for all $n \in \N$, i.e. the intersections of the elements $Z^h_\a$, $C^h_\a$ with $\{t=1\}$ have linear dimension $h$, for $\a \in \A^h_n$;
\item \label{Point_1_Proposition_P_countable_partition} the cone generated by $\{\mathrm e_i^h(n)\}$ is not degenerate and strictly contained in $C^h_n$,
\[
C(\{\mathrm e^h_i(n)\}) \in \mhC(h,V^h_n ), \qquad C(\{\mathrm e^h_i(n)\}) \subset \mathring C^h_n(-r^h_n);
\]
\item \label{Point_2_Proposition_P_countable_partition} the cones $C^h_\a $, $\a \in \A^{d-h}_n$, have a uniform opening,
\[
C^h_n(-r^h_n) \subset \prj^{\bar t}_{V^h_n} \mathring C^h_\mathfrak a;
\]
\item \label{Point_3_Proposition_P_countable_partition} the projections of cones $C^h_\a $, $\a \in \A^{d-h}_n$, are strictly contained in $C^h_n$,
\[
\prj^{\bar t}_{V^h_n} C^h_\mathfrak a \subset \mathring C^h_n; 
\]
\item \label{Point_4_Proposition_P_countable_partition} the projection at constant $t$ on $V^h_n$ is not degenerate: there is a constant $\kappa>0$ such that
\[
\big| \tpt_{V^h_n} (z - z') \big| \geq \kappa |z - z'| \quad \text{ for all }  z,z' \in C^h_\mathfrak a \cap \{t=\bar t\}, \ \mathfrak a \in \mathfrak A^h_n, \bar t \geq 0;
\]
\item the projection at constant $t$ of $Z^h_\a $ on $V^h_n$ contains a given cube,
\[
z^h_n +  \, U(\{\mathrm e^h_i(n)\}) \subset \prj^{\bar t}_{V^h_n} Z^h_\mathfrak a. 
\]
\end{enumerate}
\end{proposition}


Note that clearly the $Z^h_\a $ are transversal to $\{t= \text{constant}\}$.

\begin{figure}
\centering
\resizebox{14cm}{7cm}{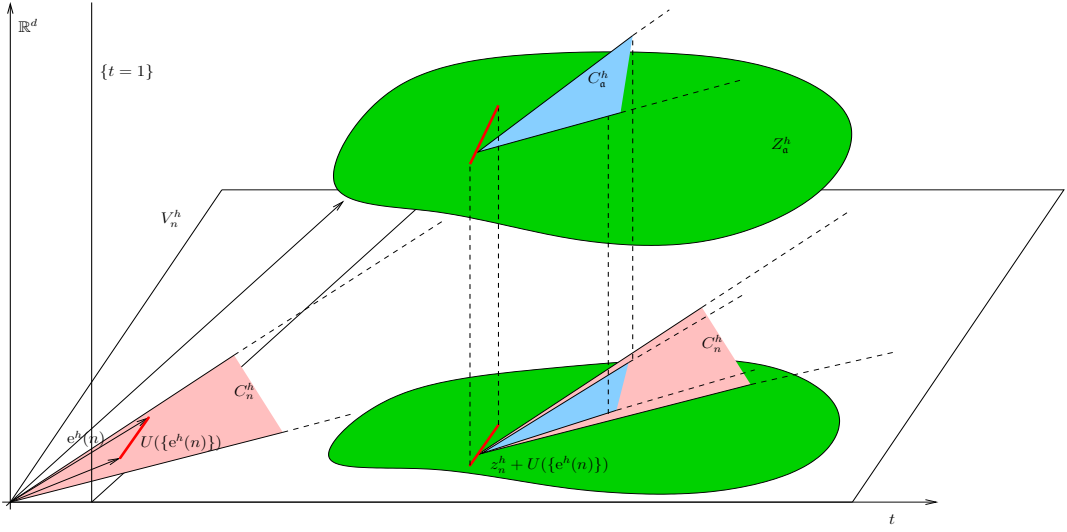}
\caption{The decomposition presented in Proposition \ref{P_countable_partition_in_reference_directed_planes}.}
\end{figure}

\begin{proof}
The only difference w.r.t. the analysis done in \cite{biadan} is the fact that we are using projections with $t$ constant, instead of projecting on $V^h_n$. However the assumption of Point \ref{Point_3_def_direc_part} of Definition \ref{D_loca_aff_part} gives that the projection of $Z^h_\a $, $C^h_\a$ at $t$ fixed is a set of linear dimension $h$, and thus we can take as a base for the partitions sets of the form \eqref{E_span_mathrm_e_h_i_n}, \eqref{E_U_square_mathrm_e}.
\end{proof}

Following the same convention of \eqref{E_mathbf_Z_base_partition}, we will use the notation $\mathbf Z^h_n := \mathtt p_z \mathbf D(h,n)$.

By the above proposition and the transversality to $\{t = \bar t\}$, the sets $\mathfrak A^h_n$ can be now chosen to be 
\begin{equation}
\label{E_mathfrak_A_h_def}
\mathfrak A^h_n := \mathbf Z^h_n \cap (\tpt_{V_h})^{-1}(z_n), \quad \mathfrak A^h := \bigcup_{n \in \N} \mathfrak A^h_n.
\end{equation}

%

\begin{definition}
\label{D_shaef_set}
We will call a directed locally affine partition $\mathbf D(h,n)$ a \emph{$h$-dimensional directed sheaf set with base directions $C^h_n$, $C^h_n(-r_n)$ and base rectangle $z_n + \lambda_n \, U(\mathrm e^h_i)$} if it satisfies the properties listed in Proposition \ref{P_countable_partition_in_reference_directed_planes} for some $\{\mathrm e^h_i(n)\}_{i=0}^h \subset \R^d$, $V^h_n = \mathrm{span} \{(1,\mathrm e^h_0),(1,\mathrm e^h_i),(1,\mathrm e^h_n) \}$, $C^h_n \in \mhC(h,V^h_n)$, $z_n \in \R^{d+1}$, $r_n,\lambda_n \in (0,\infty)$.
\end{definition}

\begin{remark}
\label{R_choice_t=1_loc_aff}
In the following we are only interested in the sets $Z^h_\a $ such that $Z^h_\a  \cap \{t=1\} \not= \emptyset$. Thus, the definition of $\mathbf D$ could be restricted to these sets, and the quotient space $\A^{d-h}$ can be taken to be a subset of an affine subspace $\{t=1\} \times \R^{d-h}$.
\end{remark}

\section{Construction of the first directed locally affine partition }
\label{S_fdlap}

In this section we show how to use the potential $\extphi$ to find a directed locally affine partition in the sense of the previous section. The approach follows closely \cite{Dan:PhD}: the main variations are in proving regularity, Sections \ref{Ss_back_forw_regul_phi} and \ref{Ss_regul_disi_phi}.
	

\begin{definition}
	\label{D_sub_super_differential_ext_phi}
	We define the \textit{sub-differential} of $ {\extphi}$ at $z$ as \[ \partial ^-   {\extphi} (z) := \big\{ z' \in \RR^d :  {\extphi}(z)- {\extphi}(z')=\bar{\mathtt c}(z-z') \big\}, \] and the \textit{super-differential} of $ {\extphi}$ at $z$ as \[ \partial ^+  {\extphi}(z) := \big\{ z'\in \RR^d :  {\extphi}(z')- {\extphi}(z)=\bar{\mathtt c}(z'-z) \big\}.\]
\end{definition}

\begin{figure}
	\centering
	\includegraphics[scale=0.65]{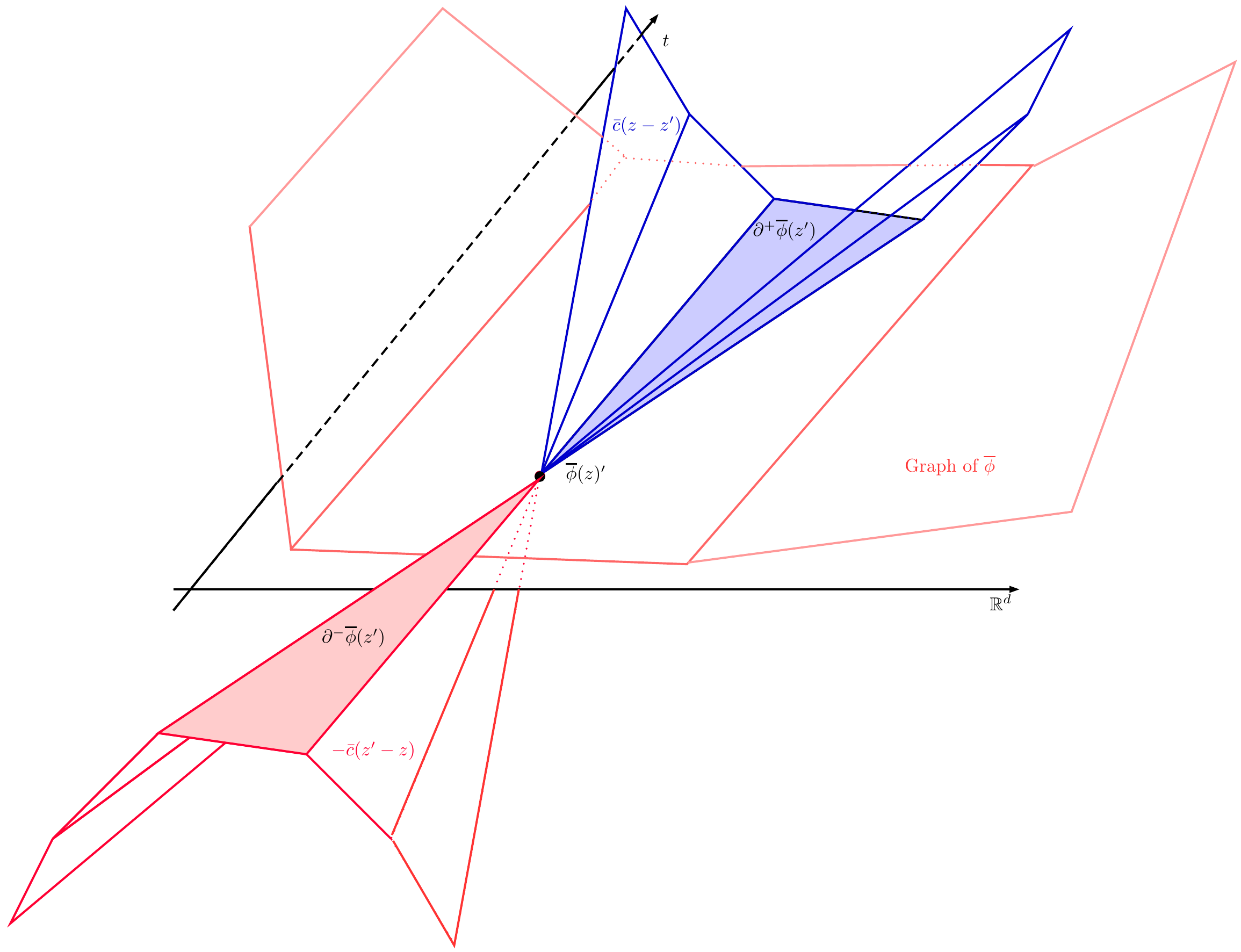}
	\caption{The sets $\partial^-\extphi(z)$, $\partial^+\extphi(z)$ of Definition \ref{D_sub_super_differential_ext_phi} are obtained intersecting $\epi\,\extc$, $-\epi\,\extc$ with $\Graph\, \extphi$, respectively.}
\end{figure}

\begin{definition}
	\label{D_opt_ray_phi}
	We say that a segment $\segment{z,z'}$ is an \textit{optimal ray} for $\extphi$ if
	\[ \extphi(z')- {\extphi}(z)=\bar{\mathtt c}(z'-z). \]
	We say that a segment $\segment{z,z'}$ is a \textit{maximal optimal ray} if it is maximal with respect to set inclusion.
\end{definition}

\begin{definition}
	The \textit{backward direction multifunction} is given by 
		\[\mathcal{D}^-  {\extphi}(z) = \bigg\{\frac{z-z'}{\prj_t(z-z')}: z'\in\partial^-  {\extphi}(z)\setminus \{z\} \bigg\},\] 
	and \textit{forward direction multifunction} is given by 
		\[\mathcal{D}^+  {\extphi}(z) = \bigg\{\frac{z'-z}{\prj_t(z'-z)}: z'\in\partial^+  {\extphi}(z)\setminus \{z\} \bigg\}.\]
\end{definition}

\begin{definition}
\label{D_F_pm_bar_phi}
	The \emph{convex cone generated} by $\mathcal{D}^-  {\extphi} $ (resp. by $\mathcal{D}^+ {\extphi}$) is the cone 
	\[F_{\extphi}^-(z) = \mathbb R^+\cdot \conv\, \mathcal{D}^-  {\extphi}(z) \qquad \big( \textrm{resp. } F_{\extphi}^+(z) = \mathbb R^+ \cdot\conv\, \mathcal{D}^+  {\extphi}(z) \big).\] 
\end{definition}

\begin{definition}
	The \textit{backward transport set} is defined respectively by 
	\[T_{\extphi}^- := \big\{ z : \partial^-  {\extphi}(z) \not = \{z\} \big\},\] the \textit{forward transport set} by 
	\[T_{\extphi}^+ := \big\{ z : \partial^+  {\extphi}(z) \not = \{z\} \big\},\] and the \textit{transport set} by \[T_{\extphi}=T_{\extphi}^-\cap T_{\extphi}^+.\]
\end{definition}
	
\begin{definition}\label{def_regular_points}
	The \textit{h-dimensional backward/forward regular transport sets} are defined for $h=0,\ldots,d $ respectively as
	\begin{eqnarray*}
		R_{\extphi}^{-,h} := \left\{ \begin{array}{lcl} 
										& (i)   & \mathcal{D}^-  {\extphi} (z) = \conv\, \mathcal{D}^-  {\extphi}(z) \\
										z\in T_{\extphi}^- :		& (ii)  & \textrm{dim}(\conv\, \mathcal{D}^-  {\extphi} (z) )=h \\
										& (iii) & \exists z'\in T_{\extphi}^-\cap(z+\interr F_{\extphi}^-(z) )\\ & & \textrm{ such that } {\extphi}(z)= {\extphi}(z')+\bar{\mathtt c}(z'-z)\textrm{ and }(i),(ii)\textrm{ hold for }z'
	        			\end{array}\right\},
	\end{eqnarray*}
	and
	\begin{eqnarray*}
		R_{\extphi}^{+,h} := \left\{ \begin{array}{lcl} 
				 						& (i)   & \mathcal{D}^+  {\extphi} (z) = \conv\, \mathcal{D}^+  {\extphi}(z) \\
										z\in T_{\extphi}^+ :	& (ii)  & \textrm{dim}(\conv\,\mathcal{D}^+  {\extphi} (z) )=h \\
										& (iii) & \exists z'\in T_{\extphi}^+\cap(z-\interr F_{\extphi}^+(z) )\\ & & \textrm{ such that } {\extphi}(z')= {\extphi}(z) +\bar{\mathtt c}(z-z')\textrm{ and }(i),(ii)\textrm{ hold for }z'
	        			\end{array}\right\}.
	\end{eqnarray*}
\end{definition}
		
\begin{figure}
	\centering
	\includegraphics[scale=0.70]{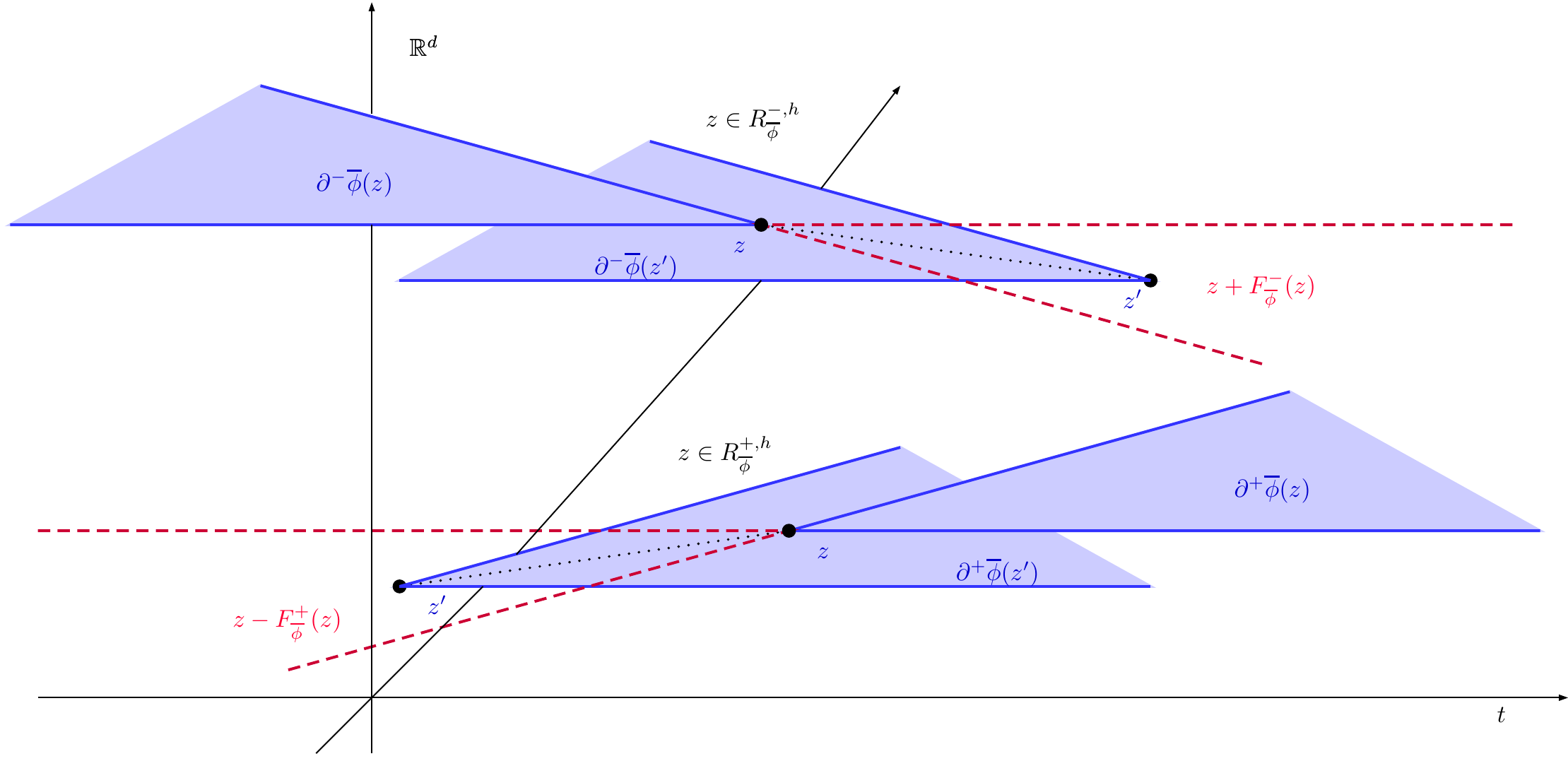}
	\caption{The sets $R_{\extphi}^{-}$ and $R_{\extphi}^{+}$ of Definition \ref{def_regular_points}.}
\end{figure}
		
Define the \textit{backward (resp. forward) transport regular set} as		
\[R_{\extphi}^- := \bigcup_{h=0}^{d} R_{\extphi}^{-,h} \qquad \bigg(\textrm{resp. } R_{\extphi}^+ := \bigcup_{h=0}^{d} R_{\extphi}^{+,h} \bigg),\]
and the \textit{regular transport set} as
\[R_{\extphi} := R_{\extphi}^+\cap R_{\extphi}^-.\]
Finally define the \emph{residual set} $N$ by \[N_{\extphi}  := T_{\extphi} \setminus R_{\extphi}.\]

\begin{proposition}
\label{p_sigma_continuity_extphi}
	The set $\partial^\pm \extphi$, $T^\pm_{\extphi}$, $\mathcal D^\pm\extphi$, $F^\pm_{\bar \phi}$, $R_{\extphi}^{\pm,h}$, $R_{\extphi}^\pm$, $R_{\extphi}$ are $\sigma-$compact.
\end{proposition}
\begin{proof}
	$\partial^\pm \extphi$. The map \[(z,z')\mapsto \Phi(z,z'):= \extphi(z') - \extphi(z) - \extc(z'-z)\] is continuous. Therefore, $\partial^\pm \extphi = \Phi^{-1}(0)$ is $\sigma$-compact.
	
	$T^\pm_{\extphi}$. The set $T^-_{\extphi}$ is the projection of the $\sigma$-compact set
		\[\bigcup_{n} \Big\{   \partial^-\extphi \cap \big\{ \big( z,z' \big):\vert z-z'\vert\geq 2^{-n} \big\} \Big\},\]
		and hence $\sigma$-compact. The same reasoning can be used for $T^+$.
	
	$\mathcal D^\pm\extphi$. Since
		\[ \{(z,z'):t(z) > t(z')\} \ni (z,z')\mapsto \frac{z-z'}{t(z)-t(z')}\in \{t=1\} \]
		is continuous, it follows that $\mathcal D^-\extphi$ is $\sigma$-compact, being the image of a $\sigma$-compact set by a continuous function.
		The same reasoning holds for $\mathcal D^+\extphi$.
		
		A similar analysis can be carried out for the $\sigma$-compactness of $F_{\extphi}^\pm$.
	
	$R_{\extphi}^{\pm,h}$.  Since the maps $A\mapsto \conv\, A$ is continuous with respect to the Hausdorff topology, and the dimension of a convex set is a lower semicontinuous map, the only point to prove is that the set 
	\[\Big\{ (z,z',C)\in \RR^h\times \RR^h\times \mathcal C(h, \RR^h): z'\in z - \interr C \Big\}\] is $\sigma$-compact. This follows by taking considering the cones $C(-r)$ 
	and writing the previous set as the countable union of $\sigma$-compact sets as follows
	\[\bigcup_{n}\Big\{ (z,z', C)\in \RR^h\times\RR^h\times\mathcal C(h,\RR^h):z'\in z+C(-2^{-n}) \setminus B(0,2^{-n}) \Big\}.\]
	Hence the set
	\begin{equation*}
		\begin{array}{lcrl}
			\bigg\{(z,z',C) & : &(i)		& z,z'\in T^-_{\extphi}\\
					&&(ii)	& C=F^-_{\extphi}(z) \\
					&&(iii)	& z'\in z+\interr C\\
					&&(iv)	& \dim\,(\conv \mathcal D^-\extphi(z)) = \dim\,(\conv \mathcal D^-\extphi(z')) = h\\
					&&(v)	& \mathcal D^-\extphi(z) = \conv\, \mathcal D^-\extphi(z), \mathcal D^-\extphi(z') = \conv\, \mathcal D^-\extphi(z')\bigg\}
			\end{array}
	\end{equation*}
	is $\sigma$-compact, and thus $R^{-,h}$ is $\sigma$-compact, too.
	The proof for $R_{\extphi}^+$ is analogous, and hence the regularity for $R_{\extphi}$ follows.
\end{proof}

\begin{proposition}\label{phi_transitivity_property}
	Let $z,z',z''\in [0,+\infty)\times\R^d$, then the following statements hold:
	\begin{enumerate}
		\item $z'\in\partial^- {\extphi}(z)$ and $z\in\partial^-  {\extphi}(z'')$ imply $z'\in\partial^-  {\extphi}(z'')$; 
		\item $z''\in\partial^+ {\extphi}(z)$ and $z\in\partial^+  {\extphi}(z')$ imply $z''\in\partial^+  {\extphi}(z')$. 		
	\end{enumerate}
\end{proposition}
\begin{proof}
	It easily follows from Definition \ref{D_sub_super_differential_ext_phi}.
\end{proof}
Moreover, it is easy to prove that:
\[z' \in \partial^\pm  {\extphi}(z) \quad \Longrightarrow \quad \partial^\pm  {\extphi}(z')\subset \partial^\pm  {\extphi}(z). \]


\begin{definition}
  	Let $z$ and $z'$ such that $  {\extphi}(z') -  {\extphi}(z) = \bar{\mathtt c}(z'-z)$ and define
  	\begin{equation}
  	\label{E_cal_Q_def}
  	\mathcal Q_{ {\extphi}}(z,z') := {\rm p}_{\RR^d} \big\{    \big( (z, {\extphi}(z)) + \epi\,{\bar{\mathtt c}} \big) \cap \big((z', {\extphi}(z')) - \epi\,{\bar{\mathtt c}}\big)    \big\}.
  	\end{equation}
\end{definition}


\begin{lemma}\label{definition_of_Q}
	It holds, \[ \mathcal Q_{ {\extphi}}(z,z') \subseteq \partial^- {\extphi}(z')\cap \partial^+ {\extphi}(z).\]
	
	Moreover
	\begin{equation}\label{E_minimal_extremal_face_containing_z_z_prime_phi}
		\R^+ \big( \mathcal Q_{ {\extphi}} (z,z') - z \big)= \R^+ \big( z'- \mathcal Q_{ {\extphi}} (z,z')\big) = F(z,z').
	\end{equation}
	where $F(z,z')$ is the projection of the minimal extremal face of $\epi\,{\bar{\mathtt c}}$ containing of $(z'-z,  {\extphi}(z')- {\extphi}(z))$.
\end{lemma}

\begin{proof}
	Let $(\bar z, \bar r) \in \big( (z, {\extphi}(z)) + \epi\,{\bar{\mathtt c}} \big) \cap \big((z', {\extphi}(z')) - \epi\,{\bar{\mathtt c}}\big) $: by definition,
	\[\bar r -  {\extphi}(z) \geq \bar{\mathtt c}(\overline z-z) \qquad \textrm{and} \qquad  {\extphi}(z') -\bar r \geq \bar{\mathtt c}(z'-\overline  z).\]
	Hence, from $   {\extphi}(z') -  {\extphi}(\bar z)\leq \bar{\mathtt c}( z' -\bar z )$,
	\begin{align*}
		{\extphi}(\bar z) -  {\extphi}(z) 	&~	\geq	{\extphi}(\bar z) -\bar r + \bar{\mathtt c}(\bar z-z)\crcr
								&~	\geq  {\extphi}(\bar z) -  {\extphi}(z') + \bar{\mathtt c}(z'-\bar z )+ \bar{\mathtt c}(\bar z-z) 
								\geq \bar{\mathtt c}(\bar z-z).
	\end{align*}
	Then $\bar z \in \partial^+  {\extphi}(z) $ and similarly one can prove $\bar z \in \partial^-  {\extphi}(z')$.
	
	The second part of the statement is an elementary property of convex sets: if $K$ is a compact convex set and $0 \in K$, then
	\begin{equation*}
	K \cap \mathrm{span}\,(K\cap (-K))
	\end{equation*}
	is the extremal face of $K$ containing $0$ in its relative interior. Since for us $K$ is a cone, the particular form \eqref{E_minimal_extremal_face_containing_z_z_prime_phi} follows.
\end{proof}  

\begin{figure}
	\centering
		\subfloat[][The set $\mathcal Q(z,z')$ and Lemma \ref{definition_of_Q}.]
			{\includegraphics[width=.40\columnwidth]{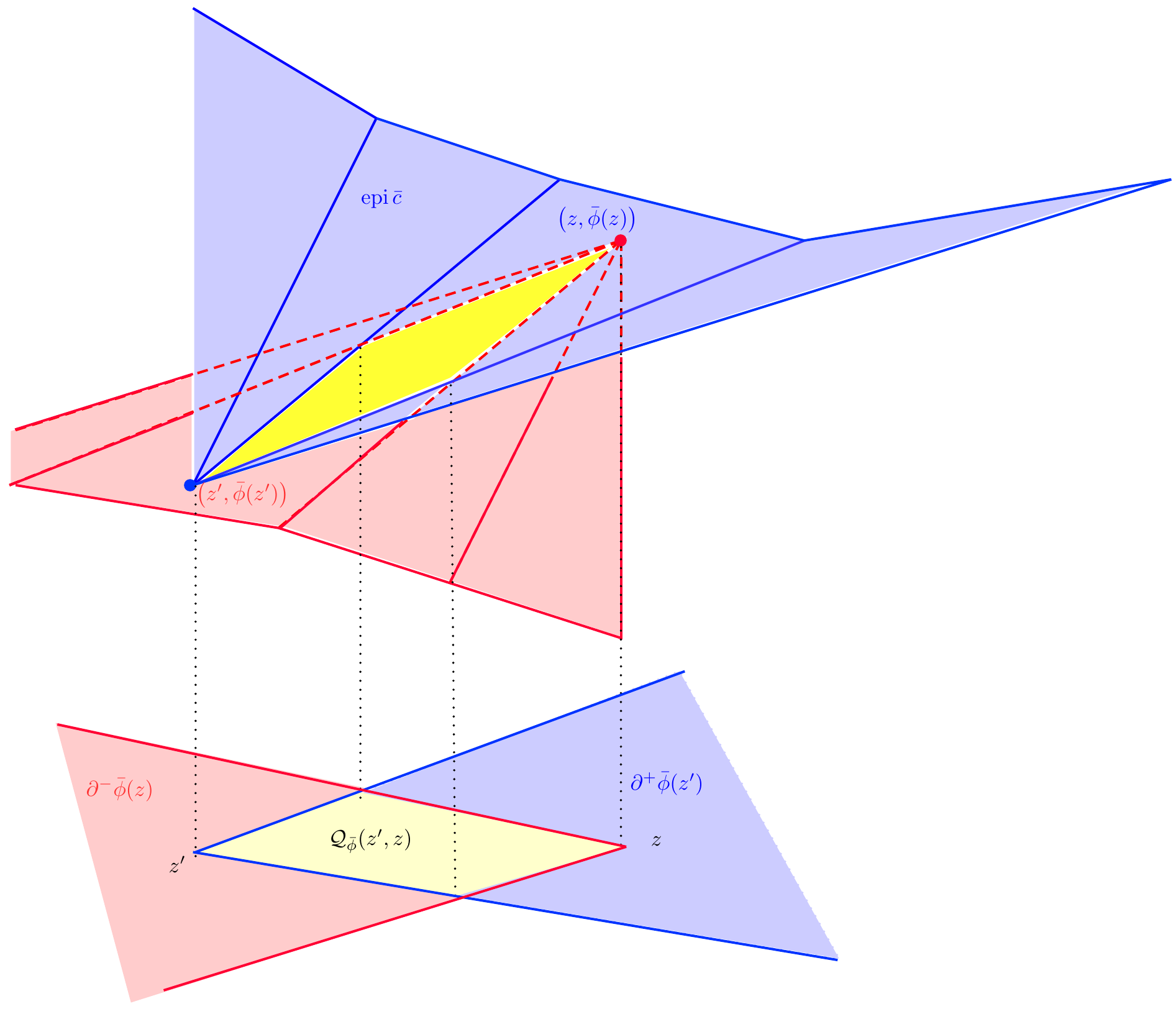}} \quad
		\subfloat[][The representation formula \eqref{E_part_phi_sum_Q} for $\partial^+  {\extphi} (z)$.]
			{\includegraphics[width=.45\columnwidth]{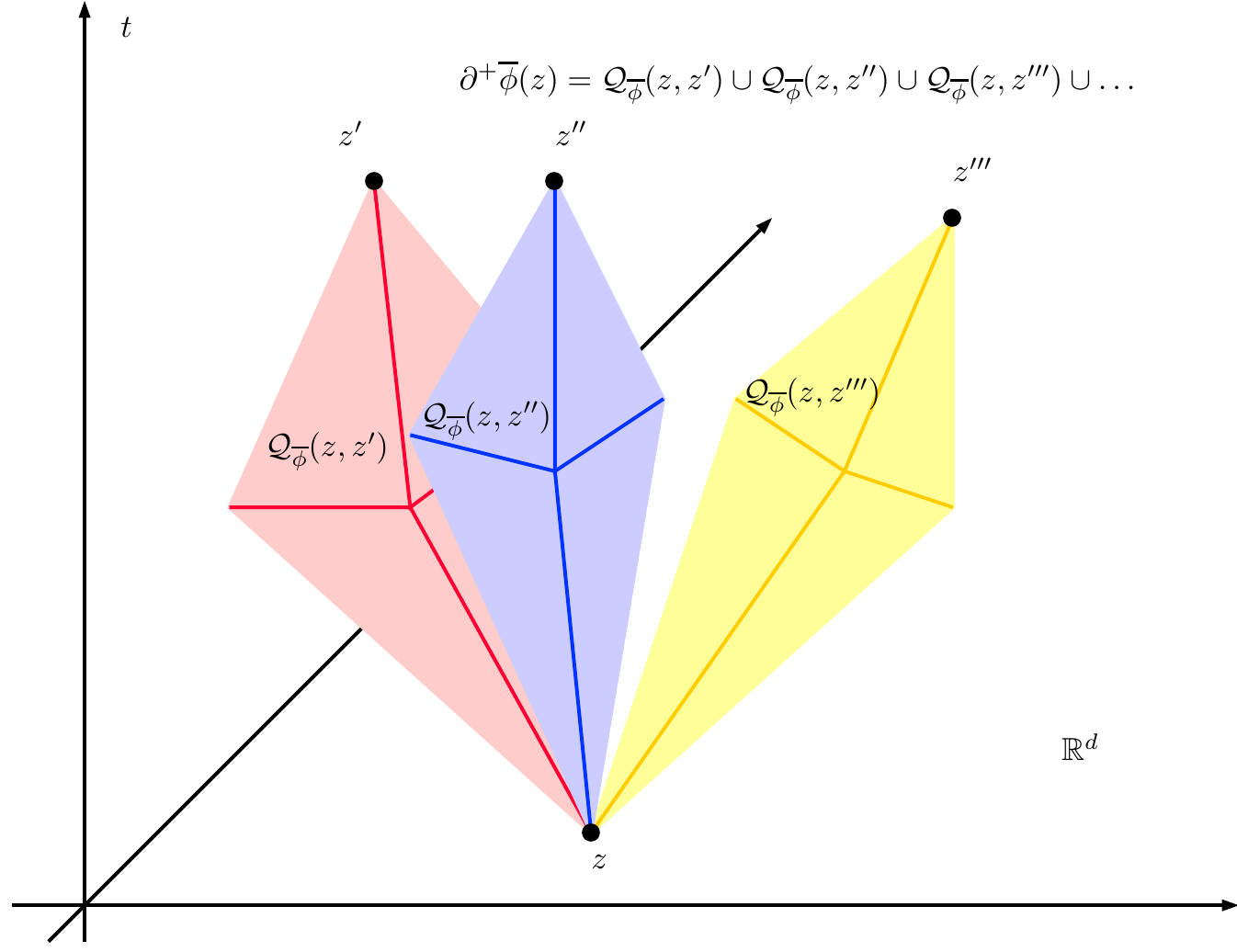}}
\caption{}\label{F_def_q}
\end{figure}  

In particular, one deduces immediately that $\partial^\pm \extphi$ is the union of sets of the form \eqref{E_cal_Q_def}, Figure \ref{F_def_q}:
\begin{equation}
\label{E_part_phi_sum_Q}
\partial^-  {\extphi} (z) = \bigcup_{z'\in\partial^-  {\extphi}(z)} \mathcal Q_{ {\extphi}} (z',z), \qquad   \partial^+  {\extphi} (z) = \bigcup_{z'\in\partial^+  {\extphi}(z)} \mathcal Q_{ {\extphi}} (z,z').
\end{equation}

\begin{proposition}\label{Pdirectionfaces1_phi}
	Let $F$ be the projection on ${\RR^d}$ of an extremal face of $\epi\,{\bar{\mathtt c}}$. The following holds:
	\begin{enumerate}
		\item\label{1_faces_phi} $ F\cap\{t=1\} \subseteq \mathcal{D}^-  {\extphi}(z) \ \iff \ \exists\delta >0 \textrm{ such that } B(z,\delta) \cap (z-F)\subseteq \partial^-  {\extphi}(z). $
		\item\label{2_faces_phi} If $F\cap\{t=1\} \subseteq \mathcal{D}^-  {\extphi}(z)$ is maximal w.r.t. set inclusion, then \[ \forall z'\in B(z,\delta) \cap (z-\interr F) \ \big( \mathcal{D}^-  {\extphi}(z') = F\cap\{t=1\} \big), \] with $\delta>0$ given by the previous point.
		\item\label{3_faces_phi} The following conditions are equivalent:
				\begin{enumerate}
					\item\label{3a_faces_phi} $\mathcal{D}^-  {\extphi}(z) = F_{\extphi}^- (z)\cap\{t=1\}$;
					\item\label{3b_faces_phi} the family of cones \[ \big\{ \mathbb R^+ \cdot\big( z- \mathcal  Q_{ {\extphi}} (z',z)\big), z'\in \partial^-  {\extphi}(z)\big\}  \]
							has a unique maximal element w.r.t. set inclusion, which coincides with $F_{\extphi}^-(z)$;
					\item\label{3c_faces_phi} 
					$\partial^- {\extphi}(z)\cap \interr (z - F_{\extphi}^-(z)) \not= \emptyset$;
					\item\label{3d_faces_phi} $\mathcal{D}^-  {\extphi}(z) = \conv\, \mathcal{D}^-  {\extphi}(z)$.
				\end{enumerate}
	\end{enumerate}
\end{proposition}

We recall that $F^-_{\bar \phi}$ is defined in Definition \ref{D_F_pm_bar_phi}. 

\begin{proof}
	\emph{Point \eqref{1_faces_phi}.} Only the first implication has to be proved. The assumption implies that there exists a point
	\[ z'\in \big( z- \interr F\big) \cap \partial^-  {\extphi}(z) \]
	and thus $\partial^-  {\extphi}(z)$ contains $\mathcal Q_{ {\extphi}} (z',z)$ by Lemma \ref{definition_of_Q}. It is fairly easy to see that this yields the conclusion, because there exists $\delta>0$ such that
	\[ B(z,\delta) \cap \big( z-F \big) \subseteq \mathcal Q_{ {\extphi}} (z',z). \]
	
	\emph{Point \eqref{2_faces_phi}.} The transitivity property of Lemma \ref{phi_transitivity_property} implies one inclusion. The opposite one follows because $\bar z$ is an inner point of $\mathcal Q_{ {\extphi}} (z',z)$.
	
	\emph{Point \eqref{3_faces_phi}.} \eqref{3b_faces_phi} implies \eqref{3a_faces_phi}: by Lemma \ref{definition_of_Q} it follows that the set $\mathcal{D}^-  {\extphi}(z)$ can be decomposed as the union of extremal faces with inner directions: since the dimension of extremal faces must increase by one at each strict inclusion, every increasing sequence of extremal faces has a maximum.
	If the maximal face $F^\mathrm{max}$ is unique, we apply Lemma \ref{definition_of_Q} to a point $\bar z$ in an inner direction, obtaining that $F^\mathrm{max} = F_{ {\extphi}}^+(z)$.
	
	\eqref{3a_faces_phi} implies \eqref{3d_faces_phi} and \eqref{3d_faces_phi} implies \eqref{3c_faces_phi}: these implications follow immediately from the definition of $\mathcal{D}^-  {\extphi}$.
	
	\eqref{3c_faces_phi} implies \eqref{3b_faces_phi}: if there is a direction in the interior of an extremal face, than by Lemma \ref{definition_of_Q} we conclude that the whole face is contained in $\mathcal{D}^- {\extphi}(z)$.
\end{proof}

A completely similar proposition can be proved for $\partial^+  {\extphi}$: we state it without proof.	

\begin{proposition}\label{Pdirectionfaces1_phi+}
	Let $F$ be the projection on ${\RR^d}$ of an extremal face of $\epi\,{\bar{\mathtt c}}$. The following holds:
	\begin{enumerate}
		\item $ F\cap\{t=1\} \subseteq \mathcal{D}^+  {\extphi}(z) \ \iff \ \exists\delta >0 \textrm{ such that } B(z,\delta) \cap (z+F)\subseteq \partial^+  {\extphi}(z). $
		\item If $F\cap\{t=1\} \subseteq \mathcal{D}^+  {\extphi}(z)$ is maximal w.r.t. set inclusion, then \[ \forall z'\in B(z,\delta) \cap (z+\interr F) \ \big( \mathcal{D}^+  {\extphi}(z') = F\cap\{t=1\}\big), \] with $\delta>0$ given by the previous point.
		\item The following conditions are equivalent:
				\begin{enumerate}
					\item $\mathcal{D}^+  {\extphi}(z) = F_{\extphi}^+ (z)\cap\{t=1\}$;
					\item the family of cones \[ \big\{ \mathbb R^+ \cdot\big( z+ \mathcal  Q_{ {\extphi}} (z',z)\big), z'\in \partial^+  {\extphi}(z)\big\}  \]
							has a unique maximal element by set inclusion, which coincides with $F_{\extphi}^+(z')$;
					\item 
					$\partial^+ {\extphi}(z)\cap \interr (z + F_{\extphi}^+(z)) \not = \emptyset$;
					\item $\mathcal{D}^+  {\extphi}(z) = \conv \mathcal{D}^+  {\extphi}(z)$.
				\end{enumerate}
	\end{enumerate}
\end{proposition}

As a consequence of Point \eqref{3_faces_phi} of the previous propositions, we will call sometimes $F_{\extphi}^-(z)$, $F_{\extphi}^+(z)$ the \emph{maximal backward/forward extremal face}.

Now we construct a map which gives a directed affine partition in $\RR^d$ up to a residual set. Define first
\begin{equation*}
	\begin{array}{ccccc}
		\mathtt v_{ {\extphi}}^- &:& R_{\extphi}^- &\to&  \cup_{h = 0}^d \mathcal A(h, \RR^d) \\ [.5em]
		&& z &\mapsto& \mathtt v_{ {\extphi}}^-(z) := \aff\, \partial^-  {\extphi}(z)
	\end{array}
\end{equation*}

\begin{lemma}
	The map $\mathtt v_{ {\extphi}}^-$ is $\sigma$-continuous.
\end{lemma}
\begin{proof}
	Since $\partial^-  {\extphi} (z)$ is $\sigma$-continuous by Proposition \ref{p_sigma_continuity_extphi} and the map $A \mapsto \aff\, A$ is $\sigma$-continuous in the Hausdorff topology, the conclusion follows.
\end{proof}

Notice that we are assuming the convention $\R^0 = \N$.

\begin{theorem}
\label{T_decomp_phi_R-}
	The map $\mathtt v_{ {\extphi}}^-$ induces a partition
	\[ \bigcup_{h=0}^d \Big\{ Z^{h,-}_{\mathfrak a} \subset \RR^d, \mathfrak a \in \R^{d-h}\Big\} \]
	on $R^-_{\extphi}$ such that the following holds:
	\begin{enumerate}
		\item\label{1_part} the sets $Z^{h,-}_{\mathfrak a}$ are locally affine;
		\item\label{2_part} there exists a projection $F^{h,-}_{\mathfrak a}$ of an extremal face $F^{h,-}_{\mathfrak a}$ with dimension $h+1$ of the cone $\epi\,{\bar{\mathtt c}}$ such that \[\forall z\in Z^{h,-}_{\mathfrak a},\quad \aff\, Z^{h,-}_{\mathfrak a} = \aff(z-F^{h,-}_{\mathfrak a}) \quad\textrm{and}\quad \mathcal{D}^-  {\extphi}(z)=(F^{h,-}_{\mathfrak a})\cap\{t=1\};\]
		\item\label{3_part} for all $z\in T^-$ there exists $r>0$, $F^{h,-}_{\mathfrak a}$ such that \[ B(z,r) \cap (z-\interr F^{h,-}_{\mathfrak a}) \subseteq Z^{h,-}_{\mathfrak a}. \]
	\end{enumerate}
\end{theorem}

The choice of $\mathfrak a$ is in the spirit of Proposition \ref{E_mathfrak_A_h_def}.

\begin{proof}
	Being a map, $\mathtt v_{ {\extphi}}^-$ induced clearly a partition $\{Z^{h,-}_{\mathfrak a}, h = 0,\ldots,d , \mathfrak a \in \R^{d-h}\}$.
	
	\emph{Point \eqref{1_part}.} Let $z \in Z^{h,-}_{\mathfrak a}$. By assumption, $z \in R_{\extphi}^{-}$ (or more precisely $z \in R_{\extphi}^{h,-}$ for some $h$), so that by Point (i) of Definition \ref{def_regular_points} of $R^{-,h}_{\extphi}$ there exists $z'$ such that \[ z' \in z - \interr \partial^-  {\extphi}(z).\]
	In the same way, by Point (iii) of Definition \ref{def_regular_points} of $R^{-,h}_{\extphi}$ there exists $z''$ such that \[z'' \in z + \interr \partial^-  {\extphi}(z).\]
	By Lemma \ref{definition_of_Q} we conclude that $z$ is contained in the interior of $\mathcal Q_{ {\extphi}} (z',z'')$, and this is a relatively open subset of $Z^{h,-}_{\mathfrak a}$, being of dimension \[\dim\,\partial^-   {\extphi}(z)=h+1.\]
	
	\emph{Point \eqref{2_part}.} Since $z \in R^{-, h}(\mathfrak a)$, then the maximal backward extremal face $F^{h,-}_{\mathfrak a}$ is given by $F_{\extphi}^-(z)$. Using the fact that $z$ is contained in a relatively open set of $Z^{h,+}_{\mathfrak a}$, the statements are a consequence of Proposition \ref{Pdirectionfaces1_phi}.
	
	\emph{Point \eqref{3_part}.} 
	If $z \in T^-_{\extphi}$, then $\partial^-  {\extphi}(z) \not= \emptyset$. We can thus take a maximal cone of the family \[ \big\{ \R^+ \cdot \mathcal Q_{ {\extphi}} (z,z'), z' \in \partial^-  {\extphi}(z) \big\}, \] and the point $z' \in \partial^-  {\extphi}(z)$ such that $\mathcal Q_{ {\extphi}} (z,z')$ is maximal with respect to the set inclusion: it is thus fairly simple to verify that \[ \interr \mathcal Q_{ {\extphi}} (z,z') \subset Z^{h,-}_{\mathfrak a} \] for some $h \in\{ 0,\ldots,d\}$, $\mathfrak a \in \R^{d-h}$. Hence, if $F^{h,+}_{\mathfrak a}$ is a projection on $\RR^d$ of an extremal face of a cone for $z \in \interr \mathcal Q_{ {\extphi}} (z,z')$, then from \eqref{E_minimal_extremal_face_containing_z_z_prime_phi} the conclusion follows.
\end{proof}

A completely similar statement holds for $R^+$, by considering of $\sigma$-continuous map

\begin{equation*}
	\begin{array}{ccccc}
		\mathtt v_{ {\extphi}}^+ &:& R_{\extphi}^+ &\to&  \cup_{h = 0}^d \mathcal A(h, \RR^d) \\ [.5em]
		&& z &\mapsto& \mathtt v_{ {\extphi}}^+(z) := \aff\, \partial^+  {\extphi}(z)
	\end{array}
\end{equation*}

\begin{theorem}
\label{T_decomp_phi_R+}
	The map $\mathtt v_{ {\extphi}}^+$ induces a partition
	\[ \bigcup_{h'=0}^d \Big\{ Z^{h',+}_{\mathfrak a'} \subset \RR^d, \mathfrak a' \in \R^{d-h'}\Big\} \]
	on $R^+_{\extphi}$ such that the following holds:
	\begin{enumerate}
		\item\label{1_part+} the sets $Z^{h',+}_{\mathfrak a'}$ are locally affine;
		\item\label{2_part+} there exists a projection $F^{h',+}_{\mathfrak a'}$ of an extremal face with dimension $h'+1$ of the cone $\epi\,{\bar{\mathtt c}}$ such that \[\forall z\in Z^{h',+}_{\mathfrak a'},\quad \aff\, Z^{h',+}_{\mathfrak a'} = \aff(z+F^{h',+}_{\mathfrak a'}) \quad\textrm{and}\quad \mathcal{D}^-  {\extphi}(z)=F^{h',+}_{\mathfrak a'} \cap\{t=1\};\]
		\item\label{3_part+} for all $z\in T^+$ there exists $r>0$, $F^{h',+}_{\mathfrak a'}$ such that \[ B(z,r) \cap (z+\interr F^{h',+}_{\mathfrak a'}) \subseteq Z^{h',+}_{\mathfrak a'}. \]
	\end{enumerate}
\end{theorem}

	In general $h \not= h'$, but on $R_{\extphi}$ the two dimensions (and hence the affine spaces $\aff\,\partial^\pm \extphi(z)$) coincide.

\begin{proposition}
\label{P_equal_R-R+}
	If $z \in R_{\extphi}$ then \[ \mathtt v_{ {\extphi}}^-(z) = \mathtt v_{ {\extphi}}^+(z).\]
\end{proposition}
\begin{proof}
	By the definition of $R_{\extphi}$, it follows that $h=h'$ because we have inner directions both forward and backward, and since each $z$ is in the relatively open set \[ \interr \big( Z^{h,-}_{\mathfrak a} \cap Z^{h',+}_{\mathfrak a'} \big), \] then $\aff\, \partial^-  {\extphi}(z) = \aff\, \partial^+  {\extphi}(z)$, i.e. $\mathtt v_{ {\extphi}}^-(z) = \mathtt v_{ {\extphi}}^+(z)$.
\end{proof}

Define thus on $R_{\extphi}$ \[ \mathtt v_{ {\extphi}} := \mathtt v_{ {\extphi}}^- \llcorner_R = \mathtt v_{ {\extphi}}^+ \llcorner_R,\] and let \[\Big\{ Z^{h}_{\mathfrak a}, \mathfrak a\in \R^{d-h}\Big\} \] be the partition induced by $\mathtt v_{ {\extphi}}$: since $R_{\extphi} = \cup_h (R_{\extphi}^{-,h} \cap R_{\extphi}^{+,h})$, it follows that \[ Z^{h}_{\mathfrak a} = Z^{h,-}_{\mathfrak a} \cap Z^{h,+}_{\mathfrak a},\] once the parametrization of $\mathcal A(h,\aff\, Z^h_\mathfrak a)$ is chosen in a compatible way.
We can then introduce the extremal cones of $\epi\,\extphi$
\[C_a^h := \epi\,\extphi \cap \big( \mathtt v_{\extphi}(z)-z \big) = F_a^{h,+}  = F_a^{h,-}.\]
Finally, define the set \[\mathbf D_{ {\extphi}} \subset \bigcup_{h=0,\ldots,d} \Big( \{h\} \times \R^{d-h} \times \R^h \times \mathcal C(h,\RR^d)\Big)\] by
\begin{equation}
\label{dlap_phi}
 	\mathbf D_{ {\extphi}} := \Big\{ \big( h,\mathfrak a,z,C \big) : C =C^h_\mathfrak a, z \in Z^h_{\mathfrak a} \Big\}.
\end{equation} 

\begin{lemma}
	The set $\mathbf D_{ {\extphi}}$ is $\sigma$-compact.
\end{lemma}

\begin{proof}
	Since $\mathtt v_{ {\extphi}}$ is $\sigma$-continuous, the conclusion follows.
\end{proof}

The next two sections will prove that this partition satisfies the condition of Theorem \ref{T_potential_deco}.

\subsection{Backward and forward regularity}
\label{Ss_back_forw_regul_phi}

The first point we need to prove is that $\mathcal H^d$-almost every point in $\{t=1\}$ is regular, i.e. it belongs to $R_{\extphi}$. 

We recall below the result obtained in \cite{biadan,Dan:PhD}, rewritten in our settings. 

\begin{proposition}[Theorem 5.21, \cite{biadan}]
\label{Theorem521}
	$\mathcal L^{d+1}$-almost every point in $\RR^d$ is regular.
\end{proposition}

Next we introduce a key tool for proving the regularity: the area estimate.  

\begin{lemma}
\label{L_inner_area_estimate_phi}
	Let $\bar t> s> \varepsilon>0$, and consider a Borel and bounded subset $S \subset \{ t=\bar t \}$ made of backward regular points.
	Then for every $(\bar t, x)\in S$ there exists a point $\sigma_s(\bar t, x) \in \interr \big( \partial^-\extphi(\bar t, x)\cap\{t=s\} \big)$ such that
	\begin{equation} 
	\label{F_inner_area_estimate_phi}
		\mathcal H^d (\sigma_s (S)) \geq \bigg( \frac{s-\varepsilon}{\bar t-\varepsilon} \bigg)^d \mathcal H^d(S).
	\end{equation}
\end{lemma}
\begin{proof}
	First of all we recall that from \eqref{HJ_def_extphi} every point has always an optimal ray reaching $\{t=0\}$. Using the assumption that the points in $S$ are backward regular and the transitivity property stated in Proposition \ref{phi_transitivity_property}, it follows that
	\begin{equation*}
	\dim\,\partial^- \extphi(z) \cap \{t=\varepsilon\} = h,  \qquad z \in S \cap R^{-.h}_{\extphi}.
	\end{equation*}
	In particular, it contains a given cone $z - K$ made of inner rays of $\partial^- \bar \phi(z)$.
	
	
	Using the fact that $\mathcal C(h,\RR^d)$ is separable and a decomposition analogous to the one of Proposition \ref{P_countable_partition_in_reference_directed_planes}, we can assume that there is a fixed $h$-dimensional cone $K'$ such that
	\begin{equation*}
	K' \cap \{t=\epsilon\} \subset \mathtt p^{\bar t}_{\mathrm{span}\,K'} \big( (z - \partial^- \extphi(z)) \cap \{t=\varepsilon\} \big).
	\end{equation*}
	Hence we can slice the sets $\partial^- \extphi(S)$ by a family of parallel planes in $\mathcal A(d-h,\RR^d)$ whose intersection with (a suitable translate of) $K'$ is an inner direction of $K'$.
%
%

	In this way, we find a $(d-h)$-dimensional problem one each affine plane $A$ such that for every $(\bar t, x)\in S \cap A$ there exists a unique point in $\interr\,\partial^-\extphi(\bar t,x) \cap \{t=\varepsilon\} \cap A$.
%
	We can now follow the strategy adopted in \cite[Lemma 2.13]{Car:strictly} and obtain the area formula.
\end{proof}

\begin{figure}
	\centering
	\includegraphics[scale=0.65]{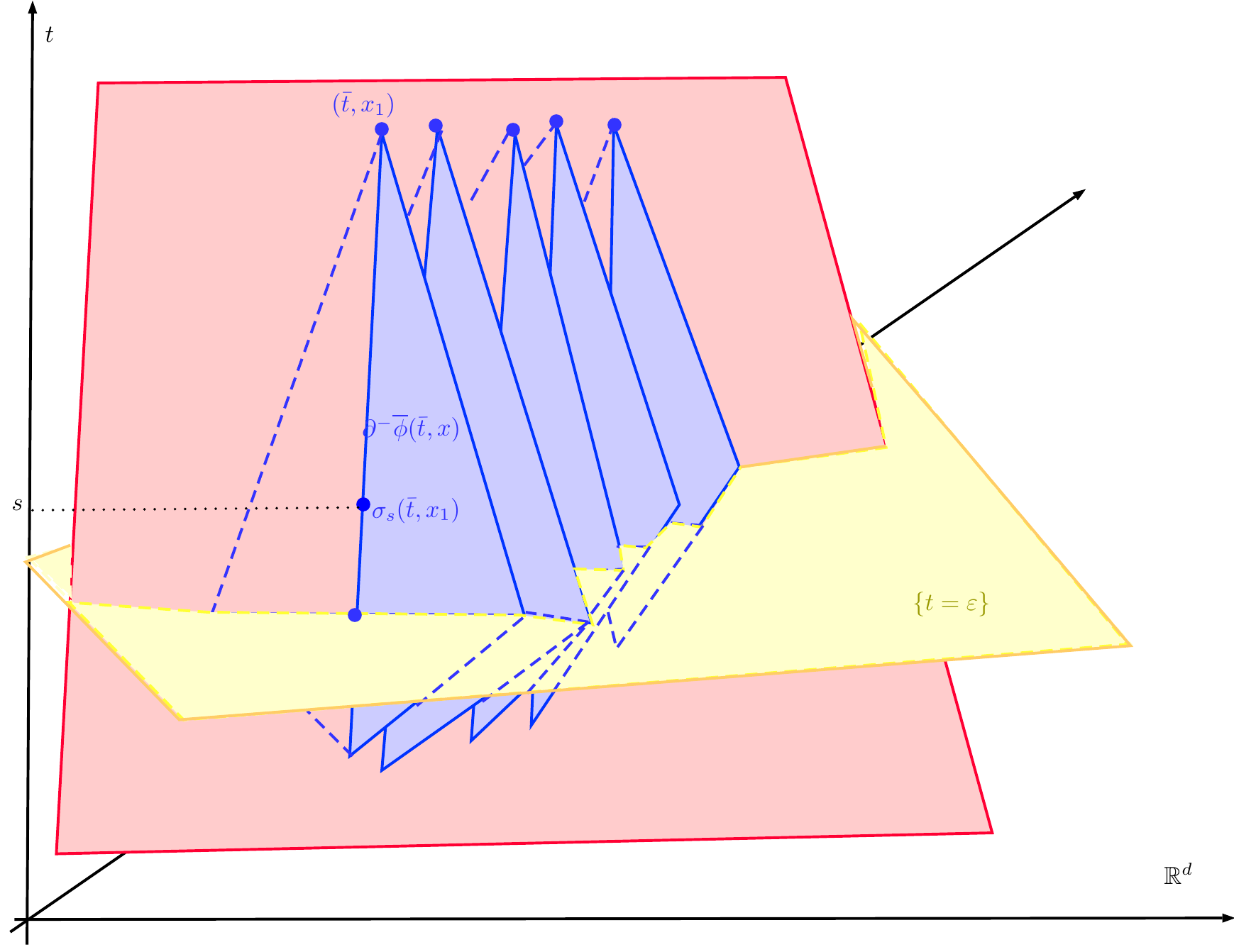}
	\caption{The strategy to prove Lemma \ref{L_inner_area_estimate_phi}: the pink plane is the transversal plane where $\partial^- \extphi(z)$ has a unique inner ray.}
	 \label{L_inner_area_estimate_phi_figure}
\end{figure}

\begin{remark}
\label{R_from_back_to_fowr_extphi}
We underline that the dimension of $\partial^- \extphi(z)$ is constant along the inner ray selected in the proof of the previous lemma. A similar property holds along inner rays of $\partial^+ \extphi(z)$, $z \in R^+_{\extphi}$.
%
%
\end{remark}

We can now prove the regularity of $\mathcal H^d \llcorner_{\{t=1\}}$-a.e. point.

\begin{proposition}
	\label{P_back_regul_phi}
	$\mathcal H^d-$almost every point in $\{t=1\}$ is regular for $\extphi$.
\end{proposition}

\begin{proof}
By Proposition \ref{Theorem521} and Fubini theorem there is $\varepsilon>0$ arbitrary small such that $\mathcal H^d$-a.e. point $z$ of $\{t=1 \pm \varepsilon\}$ is a  regular point for $\extphi$.
%
%

Let $\varepsilon'>0$ be fixed according to Lemma \ref{L_inner_area_estimate_phi}. The area estimate \ref{F_inner_area_estimate_phi} gives that the measure of points in $\{t=1-\varepsilon\}$ which belong to an inner ray of a backward regular point in $\{t=1+\varepsilon\}$ is larger than
\[ \bigg( \frac{1-\varepsilon-\varepsilon'}{1+\varepsilon -\varepsilon'} \bigg)^d \mathcal H^d (S).\]
By assumption these points are also regular (and thus forward regular).

Observe that an inner optimal ray starting from a backward regular point and arriving in a regular point is made of regular points, see Figure \ref{regular_propo_figure}. 
Therefore, by the arbitrariness of $\varepsilon$ and $\varepsilon'$ we conclude the proof.
\end{proof}

\begin{figure}
	\centering
	\includegraphics[scale=0.75]{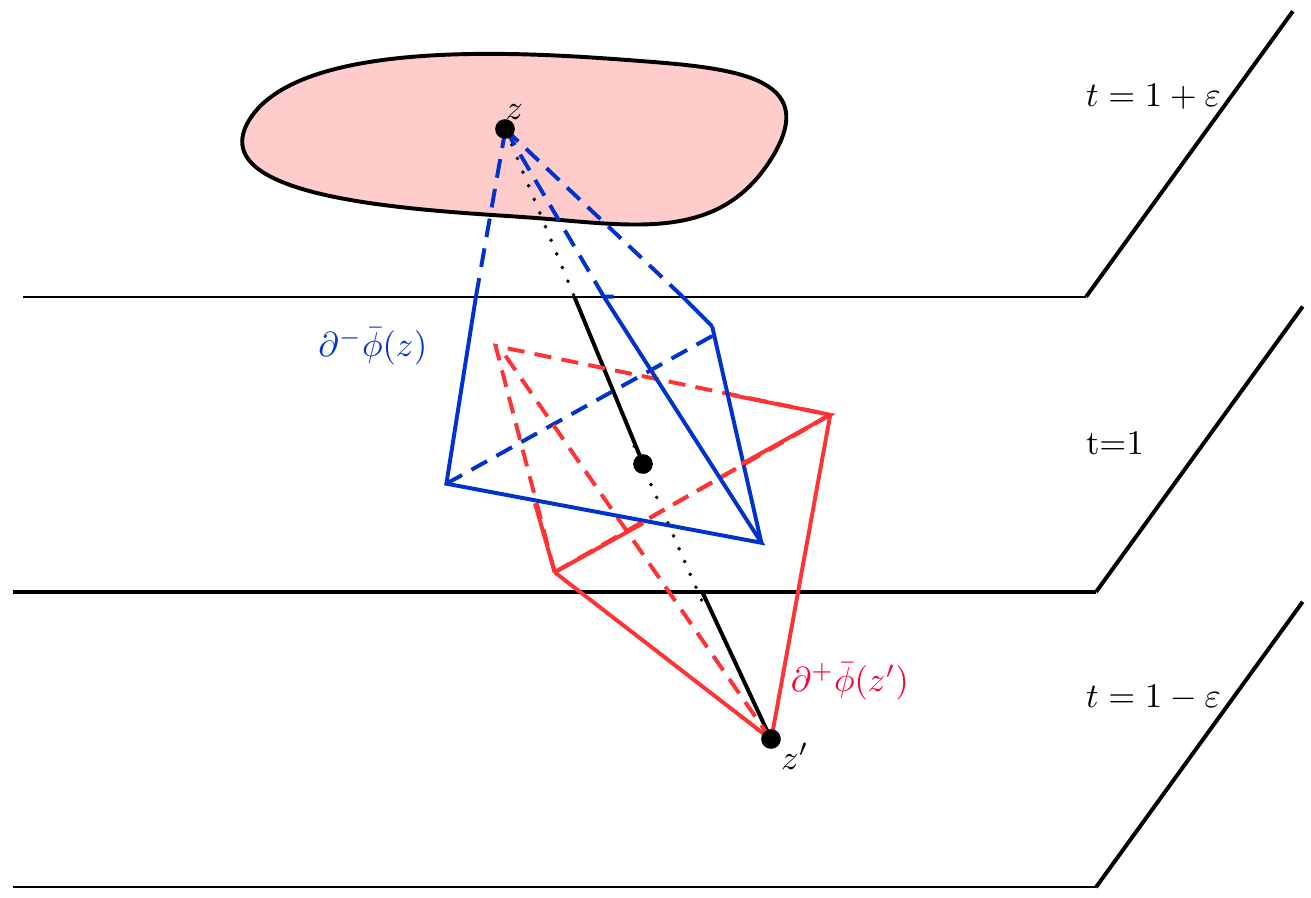}
	\caption{If $z,z'$ are regular points, then also the inner ray $\segment{z,z'}$ is made of regular points (proof of Proposition \ref{P_back_regul_phi}).}
	\label{regular_propo_figure}
\end{figure}

Hence Point \eqref{Point_1_pote_deco} of Theorem \ref{T_potential_deco} is proved.

\subsection{Regularity of the disintegration}
\label{Ss_regul_disi_phi}


By \cite[(3) of Theorem 1.1]{biadan} we know that
\[ \mathcal L^{d+1}\llcorner_{\bigcup_\a Z_\a^h} = \int f(\a,z) \mathcal H^{h+1}\llcorner_{Z_\a^h}(dz) \eta^h({d}\a).\]
so that by Fubini Theorem 
\[ \mathcal H^{d}\llcorner_{\{t=1+\varepsilon\}\cap\bigcup_\a Z_\a^h} = \int f(\a, x) \mathcal H^h\llcorner_{\{t=1+\varepsilon\}\cap Z_\a^h}(dx) \eta^h({d}\a)\qquad\text{for a.e. }\varepsilon>0.\]

Recalling the decomposition of Lemma \ref{L_inner_area_estimate_phi}, we fix the set of indexes
\begin{equation*}
\A^h_{\varepsilon,K} := \Big\{ \a \in \A^h : Z^h_\a \cap \{t=1+\varepsilon\} \not= \emptyset \ \text{and} \ K \subset \mathtt p_{\aff\,K} C^h_\a \Big\},
\end{equation*}
with $K \in \mathcal C(h,\RR^d)$ given.

An easy argument based on the push forward of $\mathcal H^d$ along the rays selected in the proof of Lemma \ref{L_inner_area_estimate_phi} (see for example \cite[Section 5]{biadan}) gives that there is
\begin{equation*}
c(\a, x) \in \big( (1 - \varepsilon/2)^d,2^d \big)
\end{equation*}
such that 
\[ \mathcal H^{d}\llcorner_{\{t=1\}\cap\bigcup_\a Z_\a^h} = \int_{\A^h_{\varepsilon,K}} c(\a, x) f(\a, x) \mathcal H^h\llcorner_{\{t=1+\varepsilon\}\cap Z_\a^h}(dx) \eta^h({d}\a).\]
The lower estimate of $c$ is given immediately by Lemma \ref{L_inner_area_estimate_phi} for $\bar t= 1 + \varepsilon/2$, $\varepsilon' = \varepsilon/2$ and $s=1$. \\
The upper estimate follows by inverting the roles of $\bar t=1+\varepsilon$ and $s=1$: in this case the ray starts in $Z^h_\a \cap \{t=1\}$ and ends in $Z^h_\a \cap \{t=1+\varepsilon\}$, and we are estimating the area between $t=1$ and $t=1+\varepsilon/2$. Using the same rays of Lemma \ref{L_inner_area_estimate_phi} in the backward direction and applying \eqref{F_inner_area_estimate_phi}, one obtains the second bound. 

Notice now that in the partition of the proof of Lemma \ref{L_inner_area_estimate_phi} the inner rays are parallel inside the elements of the partition: once the cone $K$ and the transversal planes $V_K$ are chosen, in each element $Z^h_\a$ the rays $Z^h_\a \cap V_K$ are parallel, so that the map along
\begin{equation}
\label{E_parall_transl_phi}
\begin{array}{ccccc}
\mathtt t_{V_K} &:& \underset{\A^h_{\varepsilon,K}}{\cup} Z^h_\a \cap \{t=1+\varepsilon/2\} &\to& \underset{\A^h_{\varepsilon,K}}{\cup} Z^h_\a \cap \{t=1\} \\ [1em]
&& Z^h_\a \ni x &\mapsto& \mathtt t_{V_k}(x) := (x + V_K) \cap Z^h_\a \cap \{t=1\}
\end{array}
\end{equation}
is just a translation (see Figure \ref{Fi_area_est_phi}). We thus deduce that
\begin{equation*}
(\mathtt t_{V_k})_\sharp \mathcal H^h \llcorner_{Z^h_\a \cap \{t=1+\varepsilon/2\}} = \mathcal H^h \llcorner_{\mathtt t_{V_K}(Z^h_\a \cap \{t=1+\varepsilon/2\}},
\end{equation*}
and that $c(\a,x) = c(\a)$.

Define
\begin{equation*}
f'(\a, \mathtt t_{V_K}(x)) :=  c(\a) f(\a,x),
\end{equation*}
so that we can write
%
	\[ \mathcal H^{d}\llcorner_{\{t=1\}\cap \cup_{\A^h_{\varepsilon,K}} Z_\a^h} = \int_{\A^h_{\varepsilon,K}} f'(\a, x) \mathcal H^h\llcorner_{\{t=1\} \cap Z_\a^h}(dx) \eta^h({d}\a).\]

\begin{figure}
\centering
\resizebox{14cm}{7cm}{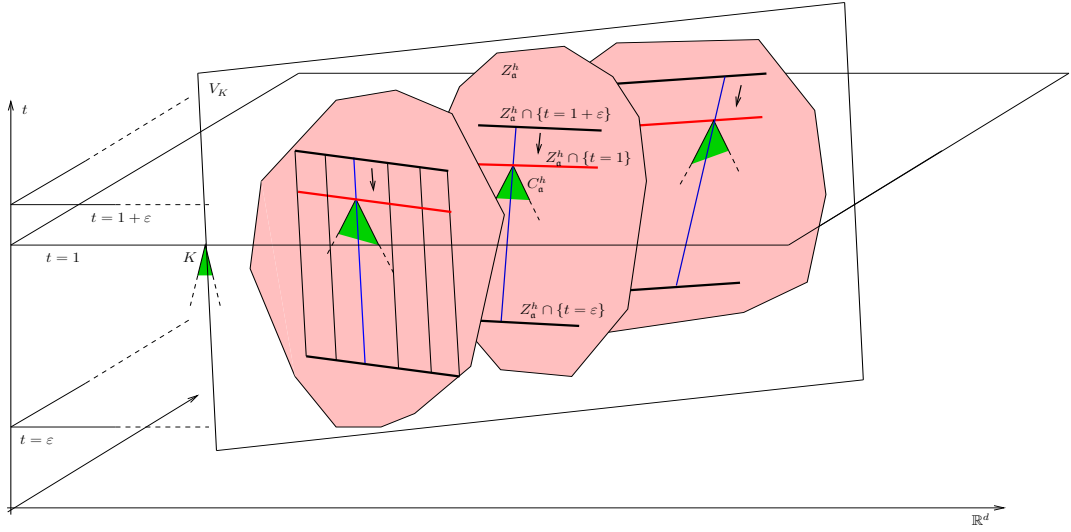}
\caption{The parallel translation of \eqref{E_parall_transl_phi} along the direction $C^h_\a \cap V_K$.}
\label{Fi_area_est_phi}
\end{figure}

By the uniqueness of the disintegration, the previous formula gives the regularity of the disintegration of $\mathcal H^d \llcorner_{\mathtt t_{V_K}(\underset{A^h_{\varepsilon,K}}{\cup} Z^h_\a \cap \{t=1\})}$. By varying $K$ and $\epsilon$ and using the fact that $Z^h_\a$ are transversal to $\{t=1\}$ and relatively open, we obtain the following proposition:
	
	\begin{proposition}
		\label{P_ac_disi}
		 The disintegration 
		 \begin{equation*}
		 \mathcal H^d \llcorner_{\cup_{h,\a} Z^h_\a \cap \{t=1\}} = \sum_h \int v^h_\a \eta^h(d\a)
		 \end{equation*}
		 w.r.t. the partition $\{Z_h^\a \cap\{t=1\} \}_{h,\a}$ is regular:
		\[v^h_\a \ll \mathcal H^h \llcorner_{Z^h_\a}.\]
	\end{proposition}

This concludes the proof of Point \eqref{Point_2_pote_deco} of Theorem \ref{T_potential_deco}. The last point of Theorem \ref{T_potential_deco} is an immediate consequence of the fact that $\extphi$ is a potential, and thus the mass is moving only along optimal rays $\Graph\,\extphi \cap (z - \epi\,\extc)$, and for all regular points $z$
\begin{equation*}
\mathtt p_{\RR^d} \big( \Graph\,\extphi \cap (z - \epi\,\extc) \big) \subset z - C^h_\a.
\end{equation*}

\begin{remark}
\label{R_form_image_meas}
The fact that $\eta^h \simeq \mathcal H^{d-h} \llcorner_{\A^h}$, with $\A^h$ chosen as in Remark \ref{R_choice_t=1_loc_aff}, is again a simple consequence of the estimate on the push-forward along optimal rays and Fubini Theorem. This result is exactly the same as the one stated in \cite[Theorem 5.18]{biadan}: we refer to that paper for the proof, because the form of the image measure is not essential in the construction and can be seen as an additional regularity of the partition.
\end{remark}

\section{Optimal transport and disintegration of measures on directed locally affine partitions}
\label{S_disintegration_locally_affine}


In this section and the following three ones we show how to refine a directed locally affine partition $\mathbf D$ either to lower the dimension of the sets or to obtain indecomposable sets. This procedure will then be applied at most $d$-times in order to obtain the proof of Theorem \ref{T_main_theor}.

Following the structure of the first  directed locally affine partition $\mathbf D_{\extphi}$ constructed in the previous section, we will consider the following three measures:
\begin{enumerate}
\item the measure $\mathcal H^d \llcorner_{\{t=1\}}$, with $\mathcal H^d(\{t=1\} \setminus \cup_{h,\a} Z^h_\a) = 0$;
\item the probability measure $\bar \mu := \delta_{\{t=1\}} \times \mu$, such that $\mu \ll \mathcal L^d$, and thus in particular $\bar \mu \ll \mathcal H^d \llcorner_{\{t=1\}}$; 
\item a probability measure $\bar \nu$ supported on $\{t=0\}$.
\end{enumerate}
%

On $\R^{d+1} \times \R^{d+1}$ we can define the natural transference cost
\begin{equation}
	\label{E_mathtt_c_mathbf_Z}
	\mathtt c_{\mathbf Z}(z,z') :=
		\begin{cases}
			0 & z \in Z^h_\mathfrak a, z - z' \in C^h_\mathfrak a, \crcr
			\infty & \text{otherwise}.
		\end{cases}
\end{equation}
Since
\[
\{\mathtt c_{\mathbf Z} < \infty\} = \big\{ (z,z') : z \in \mathbf Z, z - z' \in \mathbf d(z) \big\},
\]
i.e. it coincides with the projection $(\mathtt p_z,(\mathtt p_{\R^{d+1}} \circ \mathtt p_C)) \mathbf D$ of $\mathbf D$, then it is $\sigma$-continuous.

From Point \eqref{Point_3_pote_deco} of Theorem \ref{T_potential_deco}, it follows for $\mathbf D$ each optimal transference plan $\bar \pi$ has finite transference cost w.r.t. $\mathtt c_{\mathbf Z}$, so that the set $\Pi^f_{\mathtt c_{\mathbf Z}}(\bar \mu,\bar \nu)$ is not empty. From the observation (see Example \ref{Ex_2ndmarg} below) that in general the construction \emph{depends} on the selected transference plan $\underline{\bar \pi}$ through the marginals $\{\underline{\bar \nu}^h_\a\}_{h,\a}$, we will consider transference plans $\bar \pi \in \Pi(\bar \nu,\{\underline{\bar \nu}^h_\a\})$ such that
\[
\int \mathtt c_{\mathbf Z} \bar \pi < \infty.
\]
i.e. $\bar \pi \in \Pi^f_{\mathtt c_{\mathbf Z}}(\bar \mu,\{\underline{\bar \nu}^h_\a\})$. 

Consider the disintegrations on the partition $\{Z^h_\mathfrak a\}_{h,\a}$: if $z \mapsto (h(z),\mathfrak a(z))$ is the $\sigma$-continuous function whose graph is the projection $\mathtt p_{h,\mathfrak a,z} \mathbf D$, then
\[
\bar \mu = \sum_{h=0}^d \int_{\mathfrak A^h}\bar  \mu^h_\mathfrak a \xi^h(d\mathfrak a), \qquad \xi^h := \mathfrak a_\sharp \bar \mu \llcorner_{\mathbf Z^h}.
\]
In the same way we can disintegrate $\bar \pi \in \Pi(\bar \nu,\{\underline{\bar \nu}^h_\a\})$ w.r.t. the partition $\{Z^h_\mathfrak a \times \R^{d+1}\}_{h,\a}$, 
\begin{equation*}
	\bar \pi = \sum_{h=0}^d \int_{\mathfrak A^h} \bar \pi^h_\mathfrak a \xi^h(d\mathfrak a), \qquad \bar \mu^h_\mathfrak a = (\mathtt p_1)_\sharp \bar \pi^h_\mathfrak a.
\end{equation*}

Write also
\[
\bar \nu = \sum_{h=0}^d \int_{\mathfrak A^h} \underline{\bar \nu}^h_\mathfrak a \xi^h(d\mathfrak a), \qquad \underline{\bar \nu}^h_\mathfrak a = (\mathtt p_2)_\sharp \bar \pi^h_\mathfrak a,
\]
even if the above formula does not correspond to a real disintegration. 
%
%

In the following example we show why in general the partition depends on the plan $\underline{\bar \pi}$.

\begin{example} 
\label{Ex_2ndmarg}
For $d=2$ let
\begin{equation*}
\bar \mu = \frac{1}{8} \mathcal H^2 \llcorner_A, \qquad A = \big\{ (1,x) : \big| x - (\pm 2,0) \big| \leq 1 \big\}, 
\end{equation*}
\begin{equation*}
\bar \nu = \frac{1}{8} \big( 2 - ||x| - 2| \big) \mathcal H^1 \llcorner_B, \qquad B = \big\{ (0,x): x \in \{0\} \times [-4,4] \big\}.
\end{equation*}
and let the transportation cost $\bar{\mathtt c}$ be
\begin{equation*}
\extc(t,x) = \begin{cases}
|x|_\infty & t > 0, \crcr
\ind_{\{0\}}(x) & t = 0, \crcr
+\infty & t < 0,
\end{cases} \qquad |(x_1,x_2)|_\infty = \max \{|x_1|,|x_2|\}.
\end{equation*}

An pair of optimal plans $\bar \pi^\pm$ are given by
\begin{equation*}
\bar \pi^\pm = (\Id,\mathtt T^\pm)_\sharp \bar \mu, \qquad \text{where} \quad \mathtt T^\pm(x) := \big( 0,0, (x_2 \pm x_1) \big),\ x = (x_1,x_2) \in \R^2,
\end{equation*}
and, taking as a potential
$\bar \phi(t,x) = |x_1|$,
the decomposition obtained by the first step can be easily checked to be
\begin{equation*}
\begin{split}
Z^2_{\a_1} =&~ \big\{ (t,x), x_1 < 0 \big\}, \qquad C^2_{\a_1} = \big\{ (t,x), |x_2| \leq - x_1 \big\}, \crcr
Z^2_{\a_1} =&~ \big\{ (t,x), x_1 > 0 \big\}, \qquad C^2_{\a_2} = \big\{ (t,x), |x_2| \leq + x_1 \big\}.
\end{split}
\end{equation*}

\begin{figure}
\centering
\resizebox{14cm}{7cm}{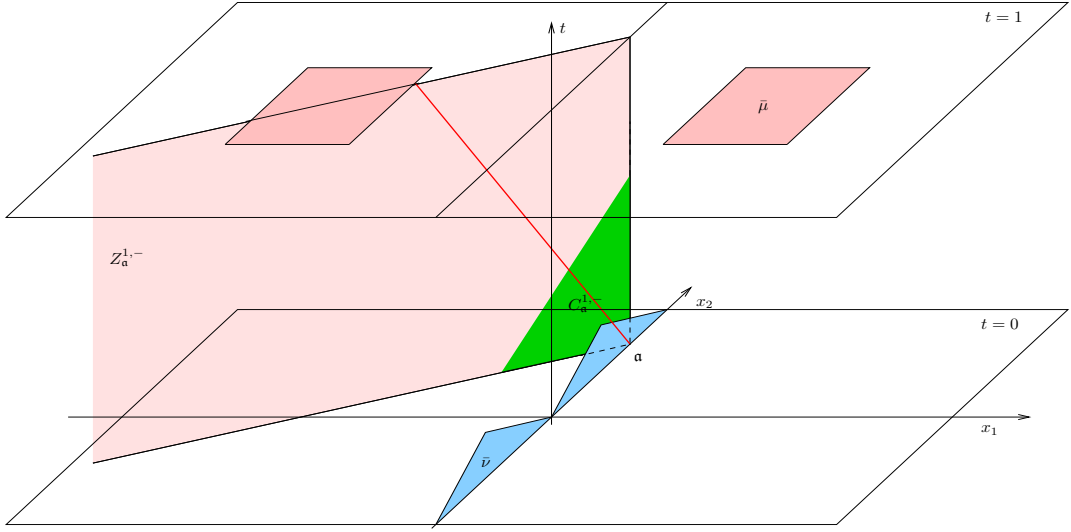}
\caption{The transport problem studied in Example \ref{Ex_2ndmarg}.}
\label{Fi_ex_2ndmarg}
\end{figure}

Being the second marginals of the disintegration of $\bar \pi^\pm = \bar \pi^{2,\pm}_{a_1} + \bar \pi^{2,\pm}_{a_2}$ w.r.t. the partition $\{Z^h_{\a_i}\}_{i=1,2}$ given by
\begin{equation*}
\bar \nu^{2,\pm}_{\a_i} = (\mathtt p_2)_\sharp \bar \pi^{2,\pm}_{\a_i} = \frac{1}{4} \big( 2 - ||x| - 2| \big) \mathcal H^1 \llcorner_{B^\pm_i}, \qquad B^\pm_i = \big\{ (1,0,x_2): \pm (-1)^i x_2 \in [0,4] \big\},
\end{equation*}
the sets $\Pi(\bar \mu,\{\bar \nu^{h,-}_{\a_i}\})$, $\Pi(\bar \mu,\{\bar \nu^{h,+}_{\a_i}\})$ are different.

If we further proceed with the decomposition, we will obtain that the indecomposable partition corresponding to $\bar \pi^\pm$ is
\begin{equation*}
\begin{split}
Z^{1,\pm}_{\a} =&~ \big\{ (t,x), x_2 = \a \mp x_1, \pm (\mathrm{sign}\, \a) x_1 > 0 \big\}, \crcr
C^{1,\pm}_{\a} =&~ \big\{ (t,x), x_2 = \mp x_1, \pm (\mathrm{sign}\, \a) x_1 \geq 0 \big\}.
\end{split}
\end{equation*}
The parameterization is such that
\begin{equation*}
Z^{1,\pm}_\a \cap \{x_1=0\} = \{(0,\a)\}, \qquad \a \in \R.
\end{equation*}

We conclude with the observation that if instead we consider the transference plan $\bar \pi = (\bar \pi^+ + \bar \pi^-)/2$, then the first decomposition is already indecomposable in the space $\Pi(\bar \mu,\{\bar \nu^2_{\a_i}\})$, because
$\bar \nu^2_{\a_i} = \bar \nu$.
%
%
\end{example}

Consider now the disintegration of the Hausdorff measure $\mathcal H^d \llcorner_{\{t=1\}}$ w.r.t. the partition $\{Z^h_\a\}_{h,\a}$: 
\[
\mathcal H^d \llcorner_{\{t=1\}} = \sum_{h=0}^d \int_{\mathfrak A^h} \upsilon^h_\mathfrak a \eta^h(d\mathfrak a), \qquad \eta^h := \mathfrak a_\sharp \mathcal H^d \llcorner_{\mathbf Z^h \cap \{t=1\}}.
\]
Following Point \eqref{Point_2_pote_deco} of Theorem \ref{T_potential_deco} on the regularity of the disintegration and taking into account the choice of the variable $\a$ considered in Remark \ref{R_choice_t=1_loc_aff}, we will recursively assume that the following
\begin{enumerate}
\item \label{Point_regularity_upsilon} the measures $\upsilon^h_\mathfrak a$ are equivalent to $\mathcal H^{h} \llcorner_{Z^h_\mathfrak a \cap \{t=1\}}$,
\item \label{Point_regularity_eta} the measures $\eta^h$ are equivalent to $\mathcal H^{d-h} \llcorner_{\mathfrak A^h}$ (see Remark \ref{R_form_image_meas}).
\end{enumerate}

We will write
\begin{equation*}
	\bar  \mu^h_\mathfrak a = {\mathtt f}_\mathfrak a \mathcal H^h, \quad {\mathtt f}_\mathfrak a \ \text{Borel}.
\end{equation*}


\subsection{Mapping a sheaf set to a fibration}
\label{Ss_mapping_sheaf_to_fibration}

Consider one of the $h$-dimensional directed sheaf sets $\mathbf D(h,n)$, $h=0,\dots,d$, constructed in Proposition \ref{P_countable_partition_in_reference_directed_planes}, and define the map
\begin{equation}
	\label{E_function_r_definition}
	\begin{array}{ccccc}
		\mathtt r &:& \mathbf D(h,n) &\to& \R^{d-h} \times \big( [0,+\infty) \times \R^h \big) \times \mhC(h,[0,+\infty) \times \R^h) \\ [.5em]
&& (\mathfrak a,z,C^{h}_\mathfrak a) &\mapsto& \mathtt r(\mathfrak a,z,C^{h}_\mathfrak a) := \big( \mathfrak a,\tpt_{\aff \, C^{h}_n} z, \tpt_{\aff \, C^{h}_n} C^{h}_\mathfrak a \big)
	\end{array}
\end{equation}
Remember that $C^h_n$ is the reference cone for each cone in $\mathbf D(h,n)$ and $\aff\, C^h_n$ is the reference plane for each $Z^h_\a$ in this sheaf.
 
Being the projection of a $\sigma$-compact set, $\mathtt r$ is $\sigma$-continuous. Clearly, since $z$ determines $\mathfrak a$ and $\mathfrak a$ determines $C^h_\mathfrak a$, also the maps
\begin{equation*}
\begin{array}{ccccl}
	\tilde{\mathtt r} &:& \mathbf Z^{h}_n &\to& \R^{d-h} \times \big( [0,+\infty) \times \R^h) \times \mhC(h,[0,+\infty) \times \R^h) \\ [.5em]
	&& z &\mapsto& \tilde{\mathtt r}(z) := \big( \mathfrak a(z),\tpt_{\aff \, C^{h}_n} z, \tpt_{\aff \, C^{h}_n} C^{h}_{\mathfrak a(z)} \big) \crcr
	&&&\crcr
	\hat{\mathtt r} &:& \mathfrak A^{h}_n &\to& \R^{d-h} \times {\mathcal C}(h,[0,+\infty) \times \R^h) \\ [.5em]
	&& \mathfrak a &\mapsto& \hat{\mathtt r}(z) := \big( \mathfrak a, \tpt_{\aff \, C^{h}_n} C^{h}_\mathfrak a \big)
\end{array}
\end{equation*}
are $\sigma$-continuous. We will use the notation
\begin{equation*}
w \in [0,+\infty) \times \R^h, \quad \tilde Z^{h}_\mathfrak a := \big( \mathtt i_{h} \circ \tpt_{\aff\,C^{h}_n} \big) Z^{h}_\mathfrak a \quad \text{and} \quad \tilde C^{h}_\mathfrak a := \big( \mathtt i_{h} \circ \tpt_{\aff\,C^{h}_n} \big) C^{h}_\mathfrak a, 
\end{equation*}
where $\mathtt i_{h} : V^{h}_n = \aff\, C^h_n \to \R^{h}$ is the identification map. Moreover set $\tilde{\mathbf Z}^{h}_n := \cup_\mathfrak a \{\mathfrak a\} \times \tilde Z^{h}_\mathfrak a$.

From Points \eqref{Point_2_Proposition_P_countable_partition} and \eqref{Point_3_Proposition_P_countable_partition} of Proposition \ref{P_countable_partition_in_reference_directed_planes} we deduce the following result.

\begin{figure}
	\centering
\resizebox{14cm}{10cm}{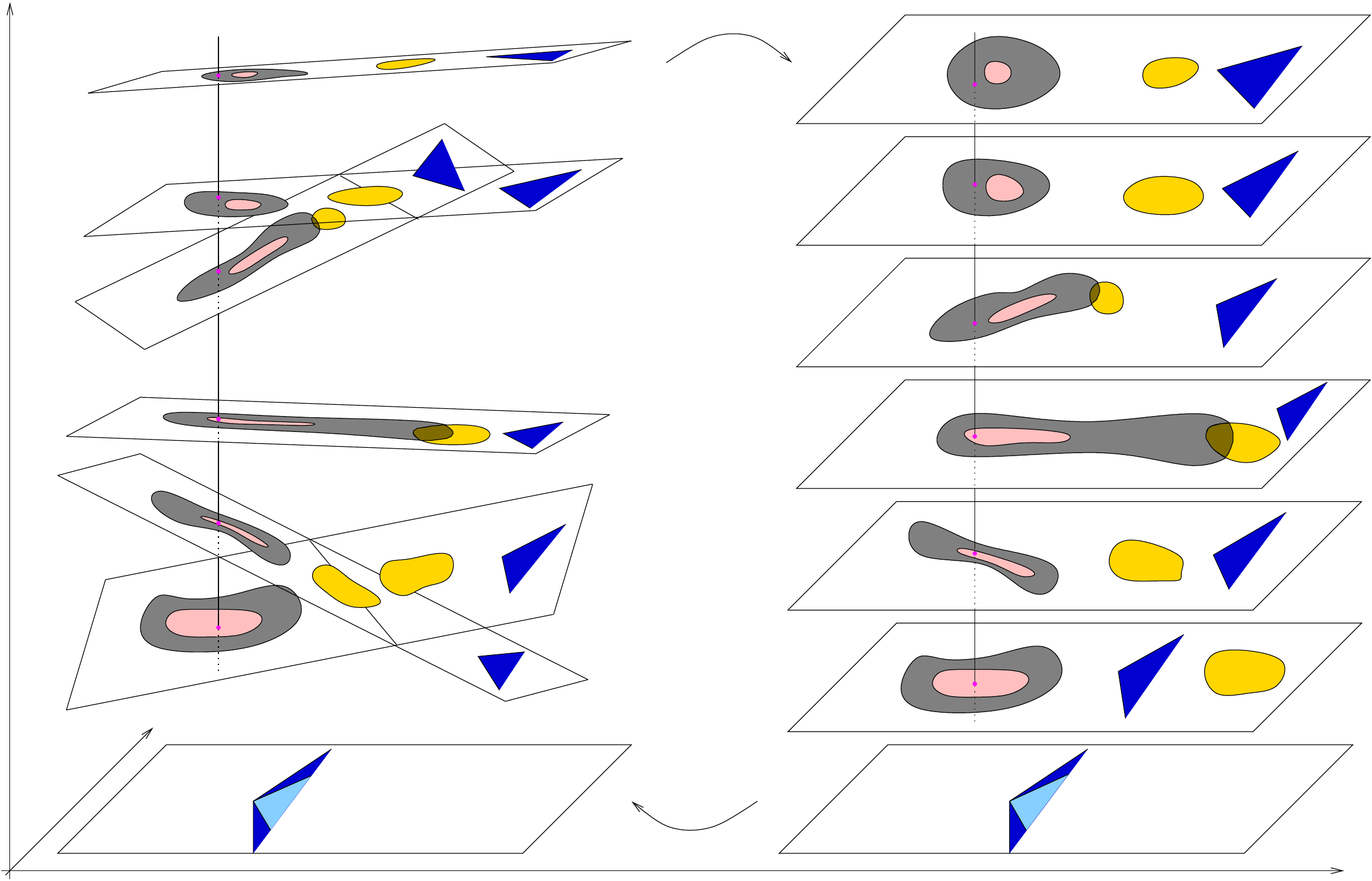}
	\caption{The map $\mathtt r$ defined in \ref{E_function_r_definition}.} 
\end{figure}

\begin{lemma}
\label{L_uniform_opening}
There exists two cones $C_n$, $C_n(-r_n)$ in ${\mathcal C}(h,[0,+\infty) \times \R^h)$ such that
\[
\forall \mathfrak a \in \mathfrak A^{h}_n \ \Big( C_n(-r_n) \subset \mathring{\tilde C}^{h}_\mathfrak a 
\ \wedge \ 
\tilde C^{h}_\mathfrak a \subset \mathring C_n \Big).
\]
\end{lemma}

\begin{proof}
	Take $C_n := ( \mathtt i_{h} \circ \tpt_{\aff\,C^{h}_n} ) C^{h}_n$. 
\end{proof}

\begin{definition}
\label{D_fibration}
A $\sigma$-compact subset $\tilde{\mathbf D}$ of $\R^{d-h} \times ([0,+\infty) \times \R^h) \times \mhC(h,[0,\infty) \times \R^h)$ such that
\begin{itemize}
\item the cone $\tilde{\mathbf D}(\mathfrak a,w)$ is independent of $w$,
\item there are two non degenerate cones $\tilde C, \tilde C' \in \mhC(h,[0,+\infty) \times \R^{h})$, $\tilde C \subset \tilde C'$, (replacing of $C_n(-r_n)$, $C_n$) satisfying Lemma \ref{L_uniform_opening},
\end{itemize}
will be called an \emph{$h$-dimensional directed fibration}.
\end{definition}

In particular we have that $\mathbf r(\mathbf D(h,n))$ is an $h$-dimensional directed fibration. 

The fact that we are considering transference problems in $\Pi(\bar \mu,\{\underline{\bar \nu}^h_\mathfrak a\})$ allows to rewrite them in the coordinates $(\mathfrak a,w) \in \R^{d-h} \times \big( [0,+\infty) \times \R^h \big)$. Indeed, consider the multifunction $\check{\mathbf r}$ whose inverse is the map
\begin{equation}
\label{E_check_mathbf_r_def}
\begin{array}{ccccc}
\check{\mathbf r}^{-1} &:& \mathfrak A^h_n \times \big( [0,+\infty) \times \R^h \big) &\to& [0,\infty) \times \R^d \\ [.5em]
&& (\mathfrak a,w) &\mapsto& \check{\mathbf r}^{-1}(\mathfrak a,w) := \aff\,Z^h_\mathfrak a \cap (\mathtt i_k \circ \tpt_{\aff\,C^{h}_n})^{-1}(w)
\end{array}
\end{equation}
and define the transport cost
\begin{equation}
\label{E_transference_cost_on_fibration}
\tilde{\mathtt c}^{h}_n \big(\mathfrak a,w,\mathfrak a',w') :=
\begin{cases}
0 & \mathfrak a = \mathfrak a', w - w' \in \tilde C^{h}_\mathfrak a, \crcr
\infty & \text{otherwise}.
\end{cases}
\end{equation}
It is clear that
\[
\mathtt c_{\mathbf Z} \big( \check{\mathtt r}^{-1}(\mathfrak a,w), \check{\mathtt r}^{-1}(\mathfrak a',w') \big) = \tilde{\mathtt c}^h_n \big( \mathfrak a, w, \mathfrak a,w'  \big). \]

Define the measures $\tilde \mu,\tilde \nu \in \mathcal P(\R^{d+1})$ by
\begin{equation*}
\tilde \mu := \int_{\mathfrak A^h_n} \tilde \mu_\mathfrak a \xi^{h}(d\mathfrak a), \quad \tilde \mu_\mathfrak a := (\mathtt i_{h} \circ \mathtt p^t_{\aff\,C^h_n})_\sharp \underline{\bar \mu}^{h}_\mathfrak a,
\end{equation*}
\begin{equation*}
\tilde \nu := \int_{\mathfrak A^h_n} \tilde \nu_\mathfrak a \xi^{h}(d\mathfrak a), \quad \tilde \nu_\mathfrak a := (\mathtt i_{h} \circ \mathtt p^t_{\aff\,C^h_n})_\sharp \bar \nu^{h}_\mathfrak a.
\end{equation*}

Since the marginals of the conditional probabilities $\bar \pi^{h}_\mathfrak a$ are fixed for all $\bar \pi \in \Pi(\bar \mu,\{\bar \nu^h_\mathfrak a\})$, then it is fairly easy to deduce the next proposition.

\begin{proposition}
\label{P_equivalence_transference_sheaf_fibration}
It holds
\[
\Pi^f_{\mathtt c_{\mathbf Z}}(\bar \mu,\{\underline{\bar \nu}^h_\mathfrak a\}) = \big( \check{\mathbf r}^{-1} \otimes \check{\mathbf r}^{-1} \big)_\sharp \Pi^f_{\tilde{\mathtt c}^h_n}(\tilde \mu,\tilde \nu).
\]
Moreover $\tilde \mu(\{t=1\}) = 1$, $\tilde \nu(\{t=0\}) = 1$.
\end{proposition}

By Point \eqref{Point_4_Proposition_P_countable_partition} of Proposition \ref{P_countable_partition_in_reference_directed_planes} one deduces that $(\mathtt p_{\aff\,C^{h}_n})_\sharp \upsilon^h_\mathfrak a$ are equivalent to $\mathcal H^{h} \llcorner_{\tilde Z^{h}_\mathfrak a \cap \{t=1\}}$, being $\upsilon^h_\a$ the conditional probabilities of the disintegration of $\mathcal H^d \llcorner_{\{t=1\}}$, so that in particular the measure $\check{\mathbf r}_\sharp \mathcal H^d \llcorner_{\mathbf Z^{h}_a \cap \{t=1\}}$ is equivalent to $\mathcal H^d \llcorner_{\tilde{\mathbf Z}^{h}_n \cap \{t=1\}}$. We have used the fact that $\check{\mathbf r}$ is single valued on $\mathbf Z^{h}_n$.

\section{Analysis of the cyclical monotone relation on a fibration}
\label{S_monotone_relation_on_fibration}

In this section we study the cyclical monotone relation generated by transference plans with finite cost on a fibration $\tilde{\mathbf D}$.

We recall that a directed fibration is a $\sigma$-compact subset of
\[
\Big\{ (\mathfrak a,w,C) \in \R^{d-h} \times \big( [0,+\infty) \times \R^h \big) \times \mhC(h,[0,+\infty) \times \R^h) \Big\}
\]
with the properties that $\mathtt p_{\mathfrak a,C} \tilde{\mathbf D}$ is the graph of a $\sigma$-compact map $\mathfrak a \mapsto  C_\mathfrak a \in \mhC(h,[0,+\infty) \times \R^h)$ and there are two cones $\tilde C,\tilde C' \in {\mathcal C}(h,\RR^h)$ such that
\begin{equation}
\label{E_uniform_opening_on_fibration}
\forall \mathfrak a \in \mathtt p_\mathfrak a \tilde{\mathbf D} \ \Big( \tilde C \subset \mathring{\tilde C}_\mathfrak a \ \wedge \ \tilde C_\mathfrak a \subset \mathring{\tilde C}' \Big).
\end{equation}
We will use the notation $\tilde{\mathfrak A} := \mathtt p_{\mathfrak a} \tilde{\mathbf D} \subset \R^{d-h}$, $\tilde Z_\mathfrak a := \mathtt p_w \tilde{\mathbf D}(\mathfrak a)$, $\tilde{\mathbf Z} := \mathtt p_{\mathfrak a,w} \tilde{\mathbf D}$: essentially the notation is the same for $\mathbf D(h,n)$, only neglecting the index $h$ and $n$.

The properties \eqref{E_uniform_opening_on_fibration} of the two cones $\tilde C, \tilde C' \in \mathcal C(h,[0,+\infty) \times \R^h)$ allows us to choose coordinates $w = (t,x) \in [0,+\infty) \times \R^{h}$ such that
\begin{equation*}
\tilde C = \epi\,\mathtt{co}, \qquad \tilde C' = \epi\,\mathtt{co}'
\end{equation*}
for two $1$-Lipschitz $1$-homogeneous convex functions $\mathtt{co}, \mathtt{co}': \R^{h} \to [0,+\infty)$ such that $\mathtt{co}'(x) < \mathtt{co}(x)$ for all $x \not= 0$. In the same way, let $\mathtt{co}_\mathfrak a : \R^{h} \to [0,+\infty)$ be $1$-Lipschitz $1$-homogeneous convex functions such that
\begin{equation*}
\tilde C_\mathfrak a = \epi\,\mathtt{co}_\mathfrak a.
\end{equation*}
Clearly from \eqref{E_uniform_opening_on_fibration} for $x \not= 0$ it holds $\mathtt{co}'(x) < \mathtt{co}_\mathfrak a(x) < \mathtt{co}(x)$. Moreover from the assumption that $\tilde C' \cap \{t=1\}$ is bounded, we have that $\mathtt{co}'(x) > 0$ for $x \not= 0$.


Define the transference cost $\tilde{\mathtt c}$ as in \eqref{E_transference_cost_on_fibration}
\begin{equation*}
\tilde{\mathtt c} \big(\mathfrak a,w,\mathfrak a',w' \big) :=
\begin{cases}
0 & \mathfrak a = \mathfrak a', w - w' \in \tilde C_\mathfrak a, \crcr
\infty & \text{otherwise}.
\end{cases}
\end{equation*}
Since $\tilde{\mathtt c} (\mathfrak a,w,\mathfrak a',w') < \infty$ implies $\mathfrak a = \mathfrak a'$, we will often write
\begin{equation*}
\tilde{\mathtt c}_\mathfrak a(w,w') := \tilde{\mathtt c} \big(\mathfrak a,w,\mathfrak a,w' \big).
\end{equation*}

From the straightforward geometric property of a convex cone $C$
\begin{equation}
\label{E_sum_of_cones_inside_cone}
w \in C \quad  \Longrightarrow \quad w + C \subset C,
\end{equation}
one deduces that
\begin{equation*}
\tilde{\mathtt c}(\mathfrak a,w,\mathfrak a',w'), \tilde{\mathtt c}(\mathfrak a',w',\mathfrak a'',w'') < \infty \quad \Longrightarrow \quad \tilde{\mathtt c}(\mathfrak a,w,\mathfrak a'',w'') < \infty.
\end{equation*}
Note that in particular $\mathfrak a = \mathfrak a' = \mathfrak a''$.

Consider two probability measures $\tilde \mu$, $\tilde \nu$ in $\tilde{\mathfrak A} \times ([0,+\infty) \times \R^h)$ such that their disintegrations
\begin{equation}
\label{E_disintegration_tilde_mu_nu}
\tilde \mu = \int_{\tilde{\mathfrak A}} \tilde \mu_\mathfrak a \tilde \xi(d\mathfrak a), \ \tilde \xi := (\mathtt p_{\tilde{\mathfrak A}})_\sharp \tilde \mu, \qquad \tilde \nu = \int_{\tilde{\mathfrak A}} \tilde \nu_\mathfrak a \tilde \xi'(d\mathfrak a),  \ \tilde \xi' = (\mathtt p_{\tilde{\mathfrak A}})_\sharp \tilde \nu,
\end{equation}
satisfy
\[
\tilde \xi = \tilde \xi' \qquad \text{and} \qquad \tilde \mu_\mathfrak a(\tilde Z_\mathfrak a \cap \{t=1\}) = \tilde \nu_\a(\tilde Z^h_\a \cap \{t=0\}) = 1.
\]
It is fairly easy to see that if $\tilde \pi \in \Pi^f(\mu,\nu)$, then
\begin{equation}
\label{E_disintegration_tilde_pi}
\tilde \pi = \int_{\tilde{\mathfrak A}} \tilde \pi_\mathfrak a \tilde \xi(d\mathfrak a) \qquad \text{with} \qquad \tilde \pi_\mathfrak a \in \Pi^f(\tilde \mu_\mathfrak a, \tilde \nu_\mathfrak a), 
\end{equation}
and conversely if $\mathfrak a \mapsto \tilde \pi_\mathfrak a \in \Pi^f(\tilde \mu_\mathfrak a,\tilde \nu_\mathfrak a)$ is an $\tilde \xi$-measurable function, then the transference plan given by the integration in \eqref{E_disintegration_tilde_pi} is in $\Pi^f(\tilde \mu,\tilde \nu)$.

We denote by $\varGamma(\tilde \pi)$ the family of $\sigma$-compact carriages $\tilde \Gamma$ of $\tilde \pi \in \Pi^f(\tilde \mu,\tilde \nu)$,
\begin{equation*}
\varGamma(\tilde \pi) := \Big\{ \tilde \Gamma \subset \{\tilde{\mathtt c} < \infty\} \cap \{t=1\} \times \{t=0\} : \tilde \pi(\tilde \Gamma) = 1 \Big\}, 
\end{equation*}
and set
\begin{equation*}
\varGamma := \bigcup_{\tilde \pi \in \Pi^f(\tilde \mu,\tilde \nu)} \varGamma(\tilde \pi).
\end{equation*}
The section of a set $\tilde \Gamma \in \varGamma$ at $(\mathfrak a,\mathfrak a)$ will be denoted by $\tilde \Gamma(\mathfrak a) \subset \{t=1\} \times \{t=0\}$.


\subsection{\texorpdfstring{A linear preorder on $\tilde{\mathfrak A} \times \RR^h$}{A linear preorder on the fibration}}
\label{Ss_linear_prorder_fibration}

Let $\tilde \Gamma \in \varGamma$. The following lemma is taken from \cite[Lemma 7.3]{biadan}: we omit the proof because it is completely similar.


\begin{lemma}
\label{L_z_n_dense_selections}
There exist a $\tilde \xi$-conegligible set $\tilde{\mathfrak A}' \subset \R^{d-h}$ and a countable family of $\sigma$-continuous functions $\mathtt w_n : \tilde{\mathfrak A}' \to \{t=1\} \times \R^h$, $n \in \N$, such that
\[
\forall n \in \N, \mathfrak a \in \tilde{\mathfrak A}' \ \Big( \mathtt w_n(\mathfrak a) \subset \mathtt p_1 \tilde \Gamma(\mathfrak a) \subset \clos\,\{\mathtt w_n(\mathfrak a)\}_{n \in \N} \Big). 
\]
\end{lemma}

Define the set $H_n \subset \tilde{\mathfrak A} \times \R^h$ by
\begin{equation}
\label{E_Hn_definition}
H_n := \bigg\{ (\mathfrak a,w): \exists I \in \N, \big\{ (w_i,w'_i) \big\}_{i = 1}^I \subset \tilde \Gamma(\mathfrak a) \ \Big( w_1 = \mathtt w_n(\mathfrak a) \ \wedge \ \tilde{\mathtt c}(\mathfrak a,w_{i+1},\mathfrak a,w'_i),\tilde{\mathtt c}(\mathfrak a,w,\mathfrak a,w'_I) < \infty \Big) \bigg\}.
\end{equation}
This set represents the points which can be reached from $\mathtt w_n(\mathfrak a)$ by means of axial path of finite costs (see \eqref{axial_path_explicit} and Figure \ref{Fi_Hn_def}).

\begin{figure}
	\centering
	\includegraphics[scale=0.85]{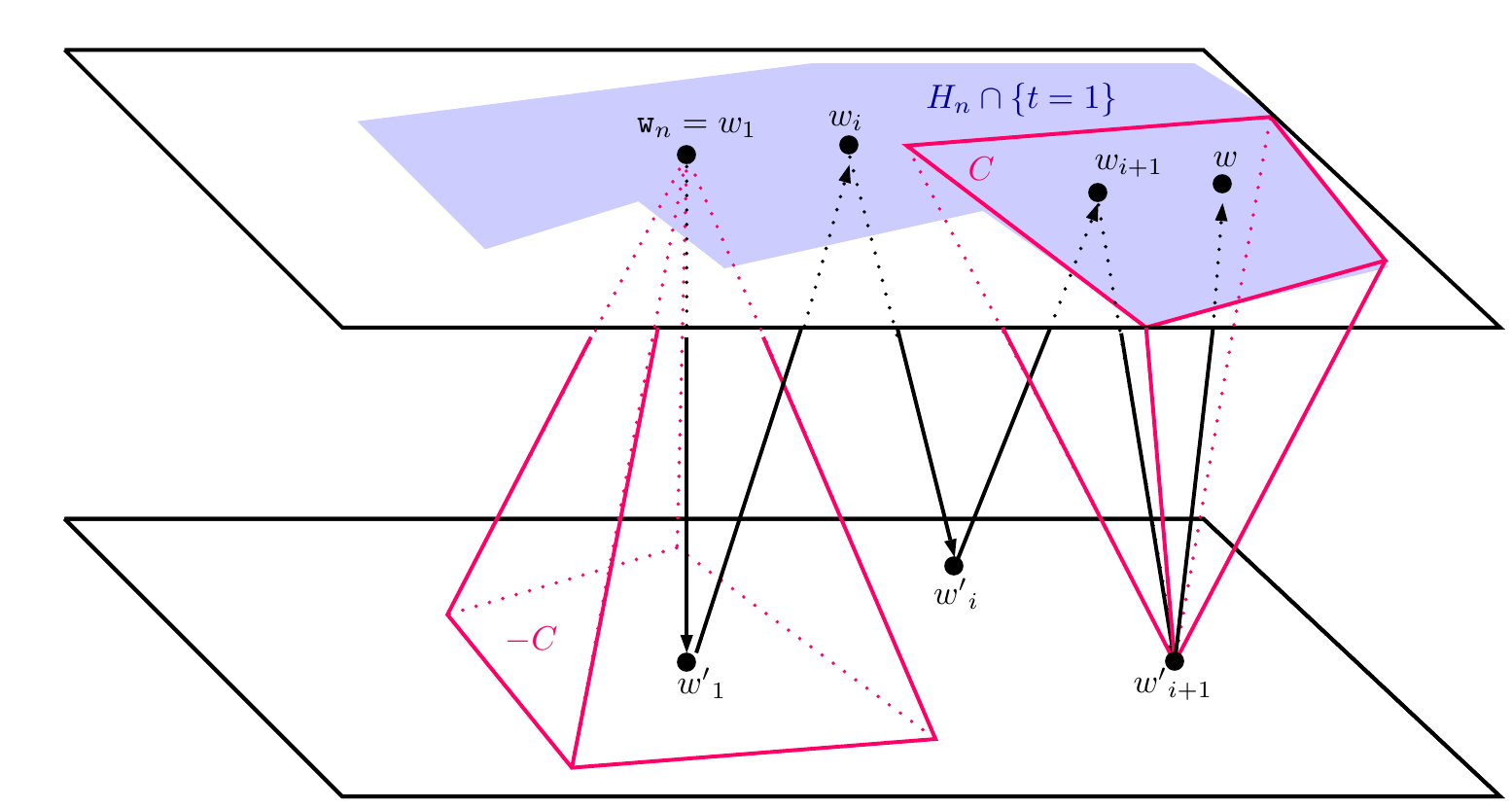}
	\caption{Construction of the set $H_n$, formula \ref{E_Hn_definition}.}
	\label{Fi_Hn_def}
\end{figure}

\begin{proposition}
\label{P_Hn_sigma_cpt_compatibility}
The set $H_n$ is $\sigma$-compact in $\R^{d-h} \times ([0,+\infty) \times \R^h)$, and moreover, defining the Borel set $\tilde{\mathfrak A}^\dag_n := \{\mathfrak a : H_n(\mathfrak a) \not= \emptyset\}$, then there exists a Borel function $\mathtt h_n : \tilde{\mathfrak A}^\dag \times \R^{h} \to [0,+\infty)$ such that for all $x,x' \in \R^{h}$
\begin{equation*}
\mathtt h_n(\mathfrak a,x') \leq \mathtt h_n(\mathfrak a,x) + \mathtt{co}_{\mathfrak a}(x'-x)
\end{equation*}
and
\begin{equation*}
\big\{ (\mathfrak a,t,x) : t > \mathtt h_n(\mathfrak a,x) \big\} \subset H_n \subset \big\{ (\mathfrak a,t,x) : t \geq \mathtt h_n(\mathfrak a,x) \big\}.
\end{equation*}
\end{proposition}

The above statement is the analog of Proposition 7.4 of \cite{biadan}, and we omit the proof. The function $\mathtt h_n$ is given explicitly by
\begin{align*}
\mathtt h_n(x) &~= \inf \big\{ \mathtt{co}_\a(x-y),y\in \mathtt i_d(H_n \cap \{t=0\}) \big\} \crcr&~= \min \Big\{ \mathtt{co}_\a(x-y),y \in \clos\,\big( \mathtt i_d(H_n \cap \{t=0\}) \big) \Big\}.
\end{align*}

The separability of $\R^d$ and the non degeneracy of the cone $\tilde C_\a$ yields the next lemma.

\begin{lemma}
\label{L_H_n_cap_t}
There exist countably many cones $\{w'_i + \tilde C_\a\}_{i \in \N}$, $\{w'_i\}_{i \in \N} \subset \mathtt i_h(\mathtt p_2 \tilde \Gamma(\a) \cap H_n(\a))$, such that
\begin{equation*}
\mathtt i_h \big( \mathring{H}_n(\a) \big) = \bigcup_{i \in \N} w'_i + \mathring{\tilde C}_\a.
\end{equation*}
Moreover, the set $\partial (\mathtt i_h(H_n(\a))) \cap \{t = \bar t\}$ is $(h-1)$-rectifiable for all $\bar t > 0$.
\end{lemma}

The estimate given in the proof below is well known, we give it for completeness.

\begin{proof}
We need to prove only the second part. Let $K = \tilde C_\a \cap \{t=\bar t\}$, and consider in $\R^h$ a set $H$  of the form
\begin{equation*}
H = \bigcup_{i \in \N} w'_i + \mathring{K}.
\end{equation*}
%
If a point $w$ belongs to $\partial H$, then it belongs to the boundary of $w' + K$ for a suitable $w'$.

Being $K = C_\a \cap \{t=1\}$ a compact convex set, the set $\partial K$ can be divided into finitely many $L$-Lipschitz graphs $O_i$, $i=1,\dots,I$. By restricting their domains, for all $\bar i$ fixed we can assume that if two points $w_j$, $j=1,2$, are such that 
\begin{equation*}
w_1 \in (w_1' + O_{\bar i}) \setminus ( w_2' + K ) \qquad \text{and} \qquad w_2 \in (w_2' + O_{\bar i} ) \setminus ( w_1' + K ),
\end{equation*}
then either they belong to a common $2L$-Lipschitz graph or their distance is of order $\mathrm{diam}\,K \approx \bar t$ (see Figure \ref{picture_rectificability}).

The previous assumption on the sets $O_i$ implies that the points in $\partial H_n \cap B(0,R)$ of the form $w' + O_{\bar i}$, with $\bar i$ fixed, can be arranged into at most $\frac{R}{\bar t}$ $2L$-Lipschitz graphs: hence we can estimate
%
%
%
\begin{equation*}
\mathcal H^{h-1} \big( \partial H \cap B(0,R) \big) \approx \max \bigg\{ \frac{R}{\bar t},1 \bigg\} \cdot R^{h-1} \approx \frac{R^h}{\bar t} + R^{h-1}.
\end{equation*}
For $R \ll \bar t$ we made use of the observation that there can be only $1$ Lipschitz graph inside the $B(0,R)$.
%
%
%
%
\end{proof}

\begin{figure}
	\centering
	\includegraphics[scale=0.85]{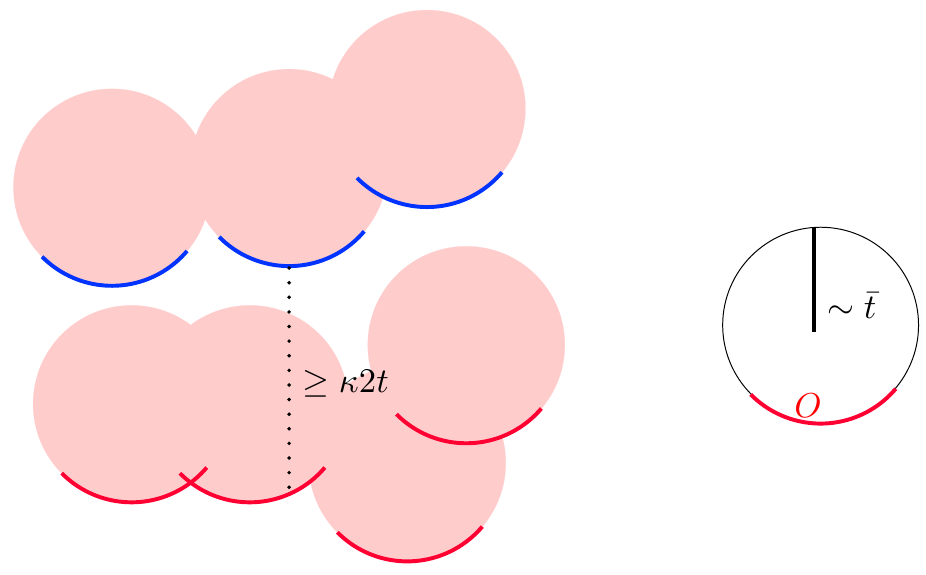}
	\caption{We can consider a straight line in $\{t=\bar t\}$ traversal to $O$. The distance between two points of $(\partial H_n) \cap B^h(z,r)$ on this line and belonging to some translations of $O$ is of the order of $\bar t$. (Lemma \ref{L_H_n_cap_t}).}
	\label{picture_rectificability}
\end{figure}

\subsubsection{\texorpdfstring{Construction of the linear preorder $\preccurlyeq_{\mathtt W}$}{Construction of a particular linear preorder}}
\label{Sss_constr_linear_preord}

Denote with $\mathtt W = \{\mathtt w_n\}_{n \in \N}$ the countable family of functions constructed in Lemma \ref{L_z_n_dense_selections}. 

Define first the function
\begin{equation}                   
\label{E_theta_Z_definition}
\begin{array}{ccccc}
\theta'_{\mathtt W,\tilde \Gamma} &:& \tilde{\mathfrak A}' \times ([0,+\infty) \times \R^h) &\to& \tilde{\mathfrak A}' \times [0,1] \\ [.5em]
&& (\mathfrak a,w) &\mapsto& \theta'_{\mathtt W,\tilde \Gamma}(\mathfrak a,w) := \Big( \mathfrak a, \max \Big\{ 0, \sum_n 2 \cdot 3^{-n} \chi_{H_n}(\mathfrak a,w) \Big\} \Big)
\end{array}
\end{equation}
It is fairly easy to show that $\theta'_{\mathtt W,\tilde \Gamma}$ is Borel. The dependence on $\tilde \Gamma$ occurs because the family $\mathtt W$ is chosen once $\tilde \Gamma$ has been selected.

Since we are interested only in the values of the functions on $\mathtt p_1 \tilde \Gamma$ and the measure $\tilde \mu$ is a.c., then once the function $\theta'_{\mathtt W,\tilde \Gamma}$ has been computed we define a new function $\theta_{\mathtt W,\tilde \Gamma}$ by
\begin{equation}
\label{E_level_theta'}
\big\{ w : \theta_{\mathtt W,\tilde \Gamma}(\a,w) \geq \underline{\theta} \big\} = \bigcup_{\nfrac{w' \in \mathtt p_2 \tilde \Gamma (\a)}{\theta'_{\mathtt W,\Gamma}(w') \geq \underline{\theta}}} w' + \tilde C_\a, \qquad \underline{\theta} \in [0,1].
\end{equation}
Being $\mathtt p_2 \tilde \Gamma(\a)$ $\sigma$-compact and $\a \mapsto \tilde C_a$ $\sigma$-compact, it is standard to prove that $\theta_{\mathtt W,\tilde \Gamma}$ is Borel if $\theta'_{\mathtt W,\tilde \Gamma}$ is.

The main reason for the introduction of the function $\theta$ will be clear in Section \ref{S_disintegration_on_affine}: indeed, $\theta$ and its upper semicontinuous envelope $\vartheta$ satisfy a Lax representation formula similar to the Lax formula for HJ equation (Remark \ref{R_HJ_prop}), so that the techniques used in order to prove regularity of the disintegration for $\bar \phi$ (Sections \ref{Ss_back_forw_regul_phi} and \ref{Ss_regul_disi_phi}) can be adapted to this context.

\begin{lemma}
\label{L_prop_theta_prim}
The functions $\theta_{\mathtt W,\tilde \Gamma}(\a),\theta'_{\mathtt W,\tilde \Gamma}(\a)$ are locally SBV on every section $\{t = \bar t\}$, and 
\[
(w,w') \in \tilde \Gamma(\a) \qquad \Longrightarrow \qquad \theta'_{\mathtt W,\tilde \Gamma}(\mathfrak a,w) = \theta'_{\mathtt W,\tilde \Gamma}(\a,w') = \theta_{\mathtt W,\tilde \Gamma}(\a,w) = \theta_{\mathtt W,\tilde \Gamma}(\a,w').
\]
\end{lemma}

Hence the function $\theta_{\mathtt W,\tilde \Gamma}$ has the same values of $\theta'_{\mathtt W,\tilde \Gamma}$ on $\mathtt p_1 \tilde \Gamma(\a) \cup \mathtt p_2 \tilde \Gamma(\a)$.

\begin{proof}
From the definition of $H_n(\a)$, formula \eqref{E_Hn_definition}, it is fairly easy to see that
\[
(w,w') \in \tilde \Gamma \ \Big( w \in H_n(\a) \ \Leftrightarrow \ w' \in H_n(\a) \Big),
\]
hence $\theta'_{\mathtt W,\tilde \Gamma}(\a,w) = \theta'_{\mathtt W,\tilde \Gamma}(\a,w')$. Moreover
\[
\theta'_{\mathtt W,\tilde \Gamma}(\a,[0,+\infty) \times \R^h) \subset \bigg\{ \sum_n s_n 3^{-n}, s_n \in \{0,2\} \bigg\},
\]
so that its range in $\mathcal L^1$-negligible (it is a subset of the ternary Cantor set).
By \eqref{E_Hn_definition}, the sets $H_n(\a) \cap \{t = \bar t\}$ is the union of compact convex sets containing a ball of radius $\mathcal O(\bar t)$, and then by Lemma \ref{L_H_n_cap_t} it is of locally finite perimeter: more precisely, in each ball in $\R^h$ of radius $r$ its perimeter is $\mathcal O(r^h/\bar t + r^{h-1})$.

Being $\theta'_{\mathtt W,\tilde \Gamma}(\a) \llcorner_{\{t = \bar t\}}$ given by the sum of the functions of $3^{-n} \chi_{H_n \cap \{t = \bar t\}}$ with (relative) perimeter $\approx 3^{-n} (r^h/\bar t +r^{h-1})$ in each ball of $\R^h$ of radius $r$, it follows that $\theta'_{\mathtt W,\tilde \Gamma}(\a) \llcorner_{\{t = \bar t\}}$ is locally BV with
\[
\TV(\theta'_{\mathtt W,\tilde \Gamma}(\a) \llcorner_{\{t=\bar t\}},B(x,r)) \approx \frac{r^h}{\bar t} + r^{h-1}.
\]
Being a countable sum of rectifiable sets,
it is locally SBV in each plane $\{t=\bar t\}$: more precisely, only the jump part of $D_w \theta'_{\mathtt W,\tilde \Gamma}(\a)$ is non zero.

The same analysis can be repeated for $\theta_{\mathtt W,\tilde \Gamma}$, using the definition \eqref{E_level_theta'}. This concludes the proof of the regularity.

The fact that $\theta_{\mathtt W,\tilde \Gamma}(\a,w) = \theta_{\mathtt W,\tilde \Gamma}(\a,w')$ if $(w,w') \in \tilde \Gamma(\a)$ is a fairly easy consequence of $\theta'_{\mathtt W,\tilde \Gamma}(\a,w) = \theta'_{\mathtt W,\tilde \Gamma}(\a,w')$ and the definition of $\theta$. Indeed, it is clear that $\theta \geq \theta'$; on the other hand, if $w'_i \in \mathtt p_2 \tilde \Gamma (\a)$ is a maximizing sequence for $w \in \mathtt p_1 \tilde \Gamma(\a)$, then the definition of $\theta'$ gives
\begin{equation*}
\theta'(w) \geq \theta'(w'_i),
\end{equation*}
and then $\theta'(w) = \theta'(w') = \theta(w)$. Since $\theta(w') \leq \theta(w)$ by \eqref{E_level_theta'}, the conclusion follows.
\end{proof}

\begin{lemma}
	$\theta_{\mathtt W,\tilde \Gamma}(\a)$ and $\theta'_{\mathtt W,\tilde \Gamma}(\a)$ are SBV in $[0,+\infty) \times \R^h$. 
\end{lemma}

\begin{proof}
Being every sub levels of $\theta_{\mathtt W,\tilde \Gamma}(\a)$ and $\theta'_{\mathtt W,\tilde \Gamma}(\a)$ the sum of cones $w' + \tilde C_\a$, the boundary of level sets is locally Lipschitz and the thesis follows. 
\end{proof}

To estimate the regularity of the disintegration of the locally affine partition generated by $\theta$ (Section \ref{S_decomposition_of_fibration}), we define the function $\vartheta_{\mathtt W,\tilde \Gamma}(\a)$ as the upper semicontinuous envelope of $\theta_{\mathtt W,\tilde \Gamma}(\a)$:
\begin{equation*}
\big\{ \vartheta(\a) \geq \underline{\theta} \big\} = \clos \big\{ \theta(\a) \geq \underline{\theta} \big\}.
\end{equation*}
Being the topological boundaries of level sets of $\theta_{\mathtt W,\tilde \Gamma}(\a)$ rectifiable, the $\mathcal H^h \llcorner_{\{t=\bar t\}}$-measure of the points where $\vartheta_{\mathtt W,\tilde \Gamma}(\a)$ and $\theta_{\mathtt W,\tilde \Gamma}(\a)$ are different is $0$.

\begin{remark}
\label{R_HJ_prop}
We observe here the relation with the Lax formula for Hamilton-Jacoby equation (with inverted time). In fact, if we define the Lagrangian
\[
L_\a(w) = \ind_{\tilde C_\a}(w),
\]
then formula \eqref{E_level_theta'} can be rewritten as
\[
\theta_{\mathtt W,\tilde \Gamma}(\a, w) = \sup \Big\{ \theta'_{\mathtt W,\tilde \Gamma}(\a, w') - L_\a(w - w'), w' \in \{t=0\} \Big\}.
\]
Moreover, the definition of $\vartheta_{\mathtt W,\tilde \Gamma}$ yields that
\[
\vartheta_{\mathtt W,\tilde \Gamma}(\a, w) = \max \Big\{ \vartheta_{\mathtt W,\tilde \Gamma}(\a, w') - L_\a(w - w'), w' \in \{t=0\} \Big\}.
\]

Being the maximum reached in some point, it follows that $\vartheta_{\mathtt W,\tilde \Gamma}$ in some sense replaces the potential $\bar \phi$. The advantages of using $\theta_{\mathtt W,\tilde \Gamma}$ instead of $\vartheta_{\mathtt W,\tilde \Gamma}$ will be clear in the following sections.

We remark here only that the disintegration of the Lebesgue measure $\mathcal H^d \llcorner_{\{t=1\}}$ on the sub levels of $\theta$ or of $\vartheta$ is equivalent, as observed above. 

\end{remark}

The space $\R^{d-h} \times [0,1]$ is naturally linearly ordered by the lexicographic ordering $\trianglelefteq$: set for $\mathfrak a = (\mathfrak a_1,\dots,\mathfrak a_{d-h})$, $s \in [0,1]$,
\begin{equation}
\label{E_lexico_on_R_d_h}
\begin{split}
(\mathfrak a,s) \triangleleft (\mathfrak a',s') \quad \Longleftrightarrow \quad & \bigg[ \exists i \in \{1,\dots,d-h\} \ \Big( \forall j < i \ \big( \mathfrak a_j = \mathfrak a'_j \big) \ \wedge \ \mathfrak a_i < \mathfrak a'_i \Big) \bigg] \crcr
& \vee \quad \big[ \a_i = {\a'}_i \ \wedge \ s < s' \big].
\end{split}
\end{equation}
The pull-back of $\trianglelefteq$ by $\theta_{\mathtt W,\tilde \Gamma}$ is the linear preorder $\preccurlyeq_{\mathtt W,\tilde \Gamma}$ defined by
\begin{equation*}
\preccurlyeq_{\mathtt W,\tilde \Gamma} := \big( \theta_{\mathtt W,\tilde \Gamma} \otimes \theta_{\mathtt W,\tilde \Gamma} \big)^{-1}(\trianglelefteq^{-1}),
\end{equation*}
and the corresponding equivalence relation on $\tilde{\mathfrak A}' \times \R^h$ is
\begin{equation*}
E_{\mathtt W,\tilde \Gamma} := \preccurlyeq_{\mathtt W,\tilde \Gamma} \cap \preccurlyeq^{-1}_{\mathtt W,\tilde \Gamma} = \Big\{ (w,w') : \theta_{\mathtt W,\tilde \Gamma}(w) = \theta_{\mathtt W,\tilde \Gamma}(w') \Big\}.
\end{equation*}
By construction $(\mathfrak a,w) \sim_{E_{\mathtt W,\tilde \Gamma}} (\mathfrak a',w')$ implies that $\mathfrak a = \mathfrak a'$. 
By convention we will also set
\[
E_{\mathtt W,\tilde \Gamma}(\mathfrak a) = \Big\{ (w,w'):  \theta_{\mathtt W,\tilde \Gamma}(\a,w) = \theta_{\mathtt W,\tilde \Gamma}(\a,w') \Big\}.
\]


\begin{lemma}
\label{L_closed_path_same_theta}
Assume that $(\mathfrak a,w), (\mathfrak a'',w'') \in \mathtt p_1 \tilde \Gamma$ can be connected by a closed axial path of finite cost. Then $(\mathfrak a,w) \sim_{E_{\mathtt W,\tilde \Gamma}} (\mathfrak a'',w'')$.
\end{lemma}

\begin{proof}
Clearly $\mathfrak a = \mathfrak a''$, and thus the condition can be stated as follows: there exist $I \in \N$, $(w_i,w'_i) \in \tilde \Gamma(\mathfrak a)$, $i=1,\dots,I$, such that $\tilde{\mathtt c}_{\mathfrak a}(w_{i+1},w'_i) < \infty$, $i = 1,\dots,I$ with $w_{I+1} = w_1$, and moreover $w = w_{i_1}$, $w'' = w_{i_2}$ for some $i_1,i_2 \in I$. This implies that
\[
\forall n \in \N \ \big( w \in H_n(\mathfrak a) \ \Leftrightarrow \ w'' \in H_n(\mathfrak a) \big),
\]
which proves that $\theta'_{\mathtt W,\tilde \Gamma}(\a,w) = \theta'_{\mathtt W,\tilde \Gamma}(\a'',w'')$. From Lemma \ref{L_prop_theta_prim} the conclusion follows.
\end{proof}


A consequence of Lemma \ref{L_prop_theta_prim} is thus that $\tilde \Gamma \subset E_{\mathtt W,\tilde \Gamma}$. If $\Gamma'$ is another carriage contained in $\{\tilde{\mathtt c} < \infty\}$, then
\[
(w,w') \in  \Gamma' (\a) \quad \Longrightarrow \quad w \preccurlyeq_{\mathtt W,\tilde \Gamma} w',
\]
because
\begin{equation}
\label{E_cone_add_theta}
\theta_{\mathtt W, \Gamma'}(w' + \tilde C_\a) \geq \theta_{\mathtt W, \Gamma'}(w')
\end{equation}
by construction. In particular from Theorem \ref{T_uniqueness} we deduce the following proposition.

\begin{proposition}
\label{P_equiv_all_plan}
If $\tilde \pi' \in \Pi^f(\tilde \mu,\tilde \nu)$, then $\tilde \pi'$ is concentrated on $E_{\mathtt W,\tilde \Gamma}$.
\end{proposition}

\subsubsection{\texorpdfstring{Construction of a $\sigma$-closed family of equivalence relations}{Construction of a countably closed family of equivalence relations}}
\label{Sss_constr_family_linear_preord}


The linear preorder $\preccurlyeq_{\mathtt W,\tilde \Gamma}$ depends on the family $\mathtt W$ of functions we choose and on the carriage $\tilde \Gamma$: by varying the $\tilde{\mathtt c}$-cyclically monotone carriage $\tilde \Gamma \in \varGamma$ and the family $\mathtt W$ dense in $\tilde \Gamma$ and we obtain in general different preorders. 

We can easily compose the linear preorders $\preccurlyeq_{\mathtt W_\beta,\tilde \Gamma_\beta}$, $\beta < \alpha$ countable ordinal number, by using the lexicographic preorder on $[0,1]^\alpha$: in fact, define the function (recall the notation $(\a,s) \in \R^{d-h} \times [0,1]$)
\begin{equation}
\label{E_theta_Zalpha_definition}
\begin{array}{ccccc}
\theta_{\{\mathtt W_\beta,\tilde \Gamma_\beta\}_{\beta < \alpha}} &:& \R^{d-h} \times [0,+\infty) \times \R^h &\to& \R^{d-h} \times [0,1]^{\alpha} \\ [.5em]
&& (\mathfrak a,w) &\mapsto& \theta_{\{\mathtt W_\beta,\tilde \Gamma_\beta\}_{\beta < \alpha}}(\mathfrak a,w) := \big( \mathfrak a, \{\mathtt p_s \theta_{\mathtt W_\beta,\tilde \Gamma_\beta}(\mathfrak a,w)\}_{\beta < \alpha} \big)
\end{array}
\end{equation}
As in the previous section $\theta_{\{\mathtt W_\beta,\tilde \Gamma_\beta\}_{\beta < \alpha}}$ is Borel, and the function should be considered defined in the domain $\cap_\beta \tilde{\mathfrak A}'_\beta$, where $\tilde{\mathfrak A}'_\beta$ is the domain of the family of functions $\mathtt W_\beta$.

If $\trianglelefteq$ is the lexicographic preorder in $\R^{d-h} \times [0,1]^\alpha$ as in \eqref{E_lexico_on_R_d_h}, then set
\begin{equation*}
\preccurlyeq_{\{\mathtt W_\beta,\tilde \Gamma_\beta\}_{\beta < \alpha}} := \big( \theta_{\{\mathtt W_\beta,\tilde \Gamma_\beta\}_{\beta < \alpha}} \otimes \theta_{\{\mathtt W_\beta,\tilde \Gamma_\beta\}_{\beta < \alpha}} \big)^{-1}(\trianglelefteq), \quad E_{\{\mathtt W_\beta,\tilde \Gamma_\beta\}_{\beta < \alpha}} := \preccurlyeq_{\{\mathtt W_\beta,\tilde \Gamma_\beta\}_{\beta < \alpha}} \cap \preccurlyeq_{\{\mathtt W_\beta,\tilde \Gamma_\beta\}_{\beta < \alpha}}^{-1}.
\end{equation*}
Clearly $\tilde \pi(E_{\{\mathtt W_\beta,\tilde \Gamma_\beta\}_{\beta < \alpha}}) = 1$, since $\tilde \pi(E_{\mathtt W_\beta,\tilde \Gamma_\beta})=1$ for all $\beta < \alpha$. To be an equivalence relation on $\R^{d-h} \times [0,+\infty) \times\R^h$, we can assume that $\Id \subset E_{\{\mathtt W_\beta,\tilde \Gamma_\beta\}_{\beta < \alpha}}$.

The next lemma is a simple consequence of the fact that a countable union of countable sets is countable. Its proof can be found in \cite[Proposition 7.5]{biadan}.

\begin{lemma}
\label{L_family_sigma_closed}
The family of equivalence relations
\begin{equation*}
\tilde{\mathcal E} := \Big\{ E_{\{\mathtt W_\beta,\tilde \Gamma_\beta\}_{\beta < \alpha}}, \mathtt W_\beta = \big\{ \mathtt w_{n,\beta} \big\}_{n \in \N}, \alpha \in \Omega \Big\}
\end{equation*}
is closed under countable intersection. Moreover, for ever $E_{\{\mathtt W_\beta,\tilde \Gamma_\beta\}_{\beta < \alpha}}$ there exists $\bar{\tilde \Gamma} \in \varGamma$ and $\bar{\mathtt W}$ such that
\begin{equation*}
E_{\bar{\mathtt W},\bar{\tilde \Gamma}} \subset E_{\{\mathtt W_\beta,\tilde \Gamma_\beta\}_{\beta < \alpha}}.
\end{equation*}
\end{lemma}

\subsection{Properties of the minimal equivalence relation}
\label{Ss_properties_minimal_equivalence_relation}


Let $\bar E_{\{\mathtt W_\beta,\tilde \Gamma_\beta\}_{\beta < \alpha}}$ be the minimal equivalence relation chosen as in Lemma \ref{L_family_sigma_closed} after a minimal equivalence relation of Theorem \ref{T_minimal_equival} in Appendix \ref{Aa_minimal_equivalence} has been selected.

Let $\bar \theta' : \R^{d-h} \times ([0,+\infty) \times \R^h) \to \R^{d-h} \times [0,1]$ be the function obtained through \eqref{E_theta_Z_definition} with the set $\bar{\tilde \Gamma}$ and the family of functions $\bar{\mathtt W}$, and let $\bar \theta$ be the corresponding function given by \eqref{E_level_theta'}. For shortness in the following we will use only the notation $\bar E$, $\bar \theta$ and $\bar \preccurlyeq$, and the convention is that $\bar \theta$ is defined on a $\sigma$-compact set $\tilde{\mathfrak A}' \times ([0,\infty) \times \R^h)$ as in the discussion following \eqref{E_theta_Zalpha_definition}.

Let $\tilde \Gamma \in \varGamma$ be a $\sigma$-compact cyclically monotone set, and let $\theta_{\mathtt W,\tilde \Gamma} : \R^{d-h} \times ([0,+\infty) \times\R^h) \to \R^{d-h} \times [0,1]^\N$ be constructed as in Section \ref{Sss_constr_linear_preord}. 

By Corollary \ref{C_constant_for_minimal_equivalence}, it follows that there exists a $\tilde \mu$-conegligible $\sigma$-compact set $\tilde B \subset \R^{d-h} \times ([0,+\infty) \times \R^h)$ and a Borel function $\mathtt s : \R^{d-h} \times [0,1] \to \R^{d-h} \times [0,1]$ such that $\theta_{\mathtt W} = \mathtt s \circ \bar \theta$ on $\tilde B$: since
\begin{equation*}
	\mathtt p_\a \bar \theta = \mathtt p_\a \theta_{\mathtt W,\tilde \Gamma} = \Id,
\end{equation*}
it follows that we can write $\mathtt s(\a,s) = (\a,\mathtt s(\a,s))$, with a slight abuse of notation. The set $\tilde B$ depends on $\theta_{\mathtt W,\tilde \Gamma}$.

Applying this result to the equivalence classes of positive $\tilde \mu_\mathfrak a$-measure, where $\tilde \mu_\mathfrak a$ are the conditional probabilities given by \eqref{E_disintegration_tilde_mu_nu}, we obtain the following proposition.

\begin{proposition}
\label{P_minimality_on_class}
There exists a set $\tilde{\mathfrak A}'' \subset \tilde{\mathfrak A}'$ of full $\tilde \xi$-measure such that
\[
\forall \mathfrak a \in \tilde{\mathfrak A}'', \underline{\theta} \in [0,1] \ \bigg( \tilde \mu_\mathfrak a( \bar \theta^{-1}(\underline{\theta}) ) > 0 \quad \Longrightarrow \quad \exists \underline{\theta}' \in [0,1] \ \Big( \tilde \mu_\mathfrak a \big( \bar \theta^{-1}(\underline{\theta}) \setminus \theta_{\mathtt W,\tilde \Gamma}^{-1}(\underline{\theta}') \big) = 0 \Big) \bigg).
\]
\end{proposition}

\begin{proof}
Since the equivalence classes under consideration have positive $\tilde \mu_\mathfrak a$-measure, the $\tilde \mu$-negligible set $(\tilde{\mathfrak A}' \times [0,+\infty) \times \R^h) \setminus \tilde B$ satisfies
\[
\tilde \xi \Big( \Big\{ \mathfrak a : \exists \underline{\theta} \ \big( \tilde \mu_\mathfrak a(\bar \theta^{-1}(\underline{t}) \setminus \tilde B) > 0 \big) \Big\} \Big) = 0.
\]
In the remaining $\tilde \xi$-conegligible subset $\tilde{\mathfrak A}''$ of $\tilde{\mathfrak A}'$ the value $\underline{\theta}' = \mathtt s(\underline{\theta})$ satisfies the statement.
\end{proof}

The essential cyclical connectedness of $\bar \theta^{-1}(\underline{\theta})$ now follows from the following lemma, valid for a generic $\theta_{\mathtt W,\tilde \Gamma}$. This lemma justify the choice of the density properties of the functions $\mathtt w_n$, Lemma \ref{L_z_n_dense_selections}.

\begin{lemma}
\label{L_indecompo_mini}
If $\tilde \mu_a(\theta^{-1}_{\mathtt W, \tilde \Gamma}(\underline{\theta})) > 0$, $\a \in \tilde{\A}'$, then it is $\mathcal H^h_{\{t=1\}}$-essentially $\tilde \Gamma$-cyclically connected.
\end{lemma}

\begin{proof}
Fix $\a \in \tilde{\A}'$ and assume the opposite. Then there are two sets $A_1,A_2$ in $\mathtt p_1(\tilde \Gamma(\a)) \cap \theta^{-1}_{\mathtt W, \tilde \Gamma}(\underline{\theta})$ of positive $\tilde \mu_\a$-measure such that each point of $A_1$ cannot reach any point of $A_2$.

If $(\bar w,\bar w') \in \tilde \Gamma(\a) \cup A_1 \times \{t=0\}$ is such that $\bar w$ is a $\mathcal H^h_{\{t=1\}}$-Lebesgue points of $A_1$, then using the non degeneracy of $\tilde C_\a$ and the density of $\mathtt W = \{\mathtt w_n\}_n$, we obtain that there exists a $\mathtt w_{\bar n}(\a) \in A_1 \cap (\bar w' + \tilde C_\a)$ with $H_{\bar n} \cap A_1 \not= \emptyset$. 

By the assumption that $\theta_{\mathtt W,\tilde \Gamma}$ is constant, we deduce that $A_2 \subset H_{\bar n}$, so that \emph{there is} an axial path connecting $\bar w$ to $A_2$.
%
%
%
%
%
%
%
%
%
%
%
%
%
%
\end{proof}

The next example shows that, differently from \cite[Theorem 7.2]{biadan}, the Lebesgue points of $\{\bar \theta_\a = t\}$ are not necessarily cyclically connected.

\begin{example}
\label{Ex_no_cycl_conn_inner}
Consider the sets in $\R^2$
\begin{equation*}
A_0 := \big\{ x_1 \geq 0, x_2 = 0 \big\}, \quad A_1 := A_0 + {B(0,1)},
\end{equation*}
and the map
\begin{equation*}
\mathtt T : \R^2 \setminus A_0 \to \R^2 \setminus A_1, \qquad \mathtt T(x) = x + \bigg( 1 - \frac{1}{2} \dist(x,A_0) \bigg)^+ \frac{x - \dist(x,A_0)}{|x - \dist(x,A_0)|}.
\end{equation*}

\begin{figure}
	\centering
	\includegraphics[scale=0.5]{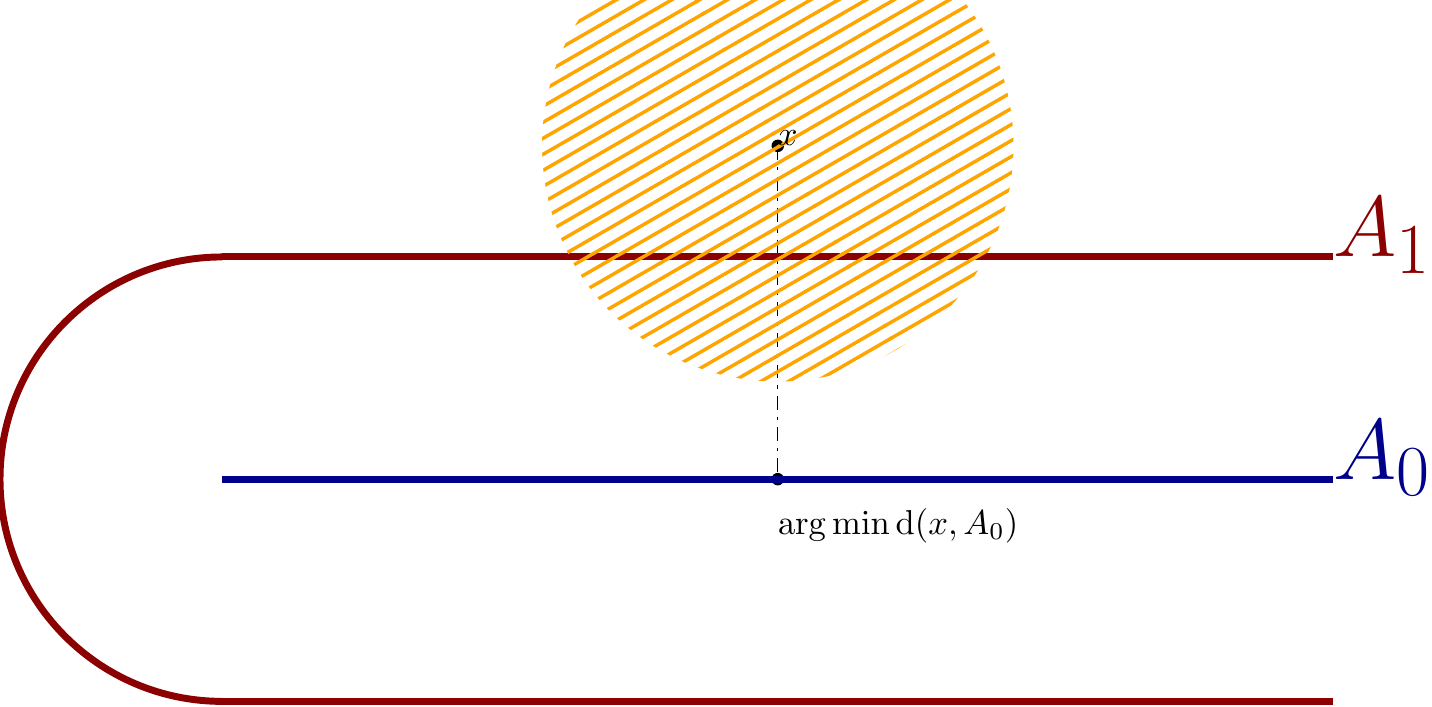}
	\caption{Example \ref{Ex_no_cycl_conn_inner}.}
\end{figure}

It is immediate to see that $\mathtt T$ is optimal for the cost $\mathtt c(x-x') = \ind_{|x|\leq 1}(x)$ and the measures $\mu$, $\mathtt T_\sharp \mu$ for all $\mu \in \mathcal P(A_1 \setminus A_0)$, and the sets and functions
\begin{equation*}
\begin{array}{lll}
\tilde \Gamma = \Graph(\mathtt T), & \mathtt W \cap \setminus A_0 = \emptyset, & \theta_{\mathtt W,\tilde \Gamma}(x) = \begin{cases} 2/3 & x \notin A_0, \crcr 0 & x \in A_0, \end{cases} \\ [2em]
\bar{\tilde \Gamma} = \tilde \Gamma \cup \Big\{ (x,x') : x \in A_0, x' \in (x + {B(0,1)}) \Big\}, & \clos(\bar{\mathtt W} \cap A_0) = A_0, & \bar \theta_{\bar{\mathtt W},\bar{\tilde \Gamma}} (x) = \frac{2}{3},
\end{array}
\end{equation*}
satisfy Proposition \ref{P_minimality_on_class}, the equivalence class for $\bar \theta$ being $\R^2$ but for $\theta$ is $\R^2 \setminus A_0$.
\end{example}

\section{Decomposition of a fibration into a directed locally affine partition}
\label{S_decomposition_of_fibration}

In this section we use the function $\mintheta$ constructed in the previous section to obtain a partition of subsets of $\R^{d-h} \times [0,+\infty) \times \R^h$ which is locally affine and satisfies some regularity properties: these properties are needed to prove the disintegration theorem of the next section.

The decomposition presented in this section can be performed using any $\sigma$-continuous function with the property \eqref{E_cone_add_theta}. In particular, in Section \ref{S_disintegration_on_affine} we will use the function $\bar \vartheta$.

Using Lusin Theorem (134Yd of \cite{MR2462519}) we can assume that $\mintheta$ is $\sigma$-continuous up to a $(\tilde \mu+\mathcal H^d\llcorner_{\{t=1\}})$-negligible set.


%

%
%
%

\subsection{Definitions of transport sets and related sets}
\label{Ss_definition_transport_initial_etc}

We define the following sets: they correspond to the sets used in \cite{biadan}, Section 4, adapted to the space $\R^{d-h} \times [0,+\infty) \times \R^h$, and are the analog of the sets used in Section \ref{S_fdlap} for the potential $\bar \phi$ and the cost $\bar{\mathtt c}$, replaced by $\bar \theta$ and $\tilde{\mathtt c}_\a$. 


\begin{description}
\item[Sub/super differential of $\mintheta$] we define the \emph{cone sub/super-differential} of $\mintheta$ at $(\a,w)$ as
\begin{align}
\label{E_sub_diff_bar_theta}
\partial^- \mintheta(\a) := \Big\{ (w,w') : \tilde{\mathtt c}_\a(w,w') < +\infty, \mintheta(\a,w) = \mintheta(\a,w') \Big\},
\end{align}
\begin{align*}
\partial^+ \mintheta(\a) := \Big\{ (w,w') : \tilde{\mathtt c}_\a(w',w) < +\infty, \mintheta(\a,w) = \mintheta(\a,w') \Big\}.
\end{align*}
Note that $\partial^- \mintheta = (\partial^+ \mintheta)^{-1}$. 
\item[Optimal Ray] the \emph{optimal rays} are the segments whose end points $(\a,w),(\a,w')$ satisfy 
\[ \mintheta(\a,w) = \mintheta(\a', w') \qquad \textrm{and} \qquad w \in w' + \tilde C_\a. \]

\item[Backward/forward transport set] the \emph{forward/backward transport set} are defined by
\begin{equation*}
	T_\mintheta^-(\a) := \big\{ w : \partial^- \mintheta(\a,w) \not= \{w\} \big\} = \mathtt p_1 \big( \partial^- \mintheta(\a) \setminus \Id \big),
\end{equation*}
\begin{equation*}
	T_\mintheta^+(\a) := \big\{ w : \partial^+ \mintheta(\a,w) \not= \{w\} \big\} = \mathtt p_1 \big( \partial^+ \mintheta(\a) \setminus \Id \big).
\end{equation*}

\item[Set of fixed points] the \emph{set of fixed points} is given by
\begin{equation*}
	F_\mintheta(\a) := \R^h \setminus \big( T_\mintheta^-(\a) \cup T_\mintheta^+(\a) \big).
\end{equation*}

\item[Backward/forward direction multifunction] The \emph{backward/forward direction multifunction} is given by
	\begin{equation*}
		\mathcal D^- \mintheta(\a) := \bigg\{ \bigg( w, \frac{w-w'}{\prj_t(w-w')} \bigg), w = (t,x) \in T_\mintheta^-(\a), w' = (t',x') \in \partial^- \mintheta(\a,w) \setminus \{w\} \bigg\},
	\end{equation*}
	\begin{equation*}
		\mathcal D^+ \mintheta(\a) := \bigg\{ \bigg( w, \frac{w'-w}{\prj_t(w'-w)} \bigg), w = (t,x) \in T_\mintheta^+(\a), w' = (t',x') \in \partial^+ \mintheta(\a,w) \setminus \{w\} \bigg\},
	\end{equation*}
normalized such that $\mathtt p_t \mathcal D^{\pm} \mintheta(\a,w) = 1$.

\item[Convex cone generated by $\mathcal D^\pm \mintheta$] define
\begin{equation*}
C_\mintheta^-(\a,w) := \R^+ \cdot \conv\,\mathcal D^- \mintheta(\a,w), \qquad C_\mintheta^+(\a,w) := \R^+ \cdot \conv\,\mathcal D^+ \mintheta(\a,w).
\end{equation*}

\item[Backward/forward regular transport set] the \emph{$\ell$-dimensional backward/forward regular transport sets} are defined for $\ell = 0,\dots,h$ respectively as
\begin{align*}
R_\mintheta^{-,\ell}(\a) := \bigg\{ w \in T^-(\a) :\ (i)&~ \mathcal D^- \mintheta(\a,w) = \conv \, \mathcal D^- \mintheta(\a,w) \crcr
						  (ii)&~ \dim \big( \conv\, \mathcal D^- \mintheta(\a,w) \big) = \ell \crcr
						  (iii)&~ \exists w' \in T_\mintheta^-(\a) \cap \big( w + \inter_{\mathrm{rel}}\, C_\mintheta^-(\a,w) \big) \ \bigl( (i),(ii) \ \text{hold for} \ w' \big) \bigg\},
\end{align*}
\begin{align*}
R_\mintheta^{+,\ell}(\a) := \bigg\{ w \in T^+(\a) :\ (i)&~ \mathcal D^+ \mintheta(\a,w) = \conv \, \mathcal D^+ \mintheta(\a,w) \crcr
						  (ii)&~ \dim \big( \conv\, \mathcal D^+ \mintheta(\a,w) \big) = \ell \crcr
						  (iii)&~ \exists w' \in T_\mintheta^+(\a) \cap \big( w - \inter_{\mathrm{rel}}\, C_\mintheta^+(\a,w) \big) \ \bigl( (i),(ii) \ \text{hold for} \ z' \big) \bigg\}.
\end{align*}
The \emph{backward/forward regular transport sets} and the \emph{regular transport set} are defined respectively by
\begin{equation*}
R_\mintheta^-(\a) := \bigcup_{\ell = 0}^h R_\mintheta^{-,\ell}(\a), \quad R_\mintheta^+(\a) := \bigcup_{\ell = 0}^h R_\mintheta^{-,\ell} \qquad \text{and} \qquad R_\mintheta(\a) := R_\mintheta^-(\a) \cap R_\mintheta^+(\a).
\end{equation*}
Finally define the \emph{residual set $N_\mintheta$} by
\begin{equation*}
N_\mintheta(\a) := T_\mintheta(\a) \setminus R_\mintheta(\a).
\end{equation*}

\end{description}



The next statements are completely analog to \cite[Section 4]{biadan}, and we will omit the proof. 

\begin{proposition}
\label{P_borel_regu_transport_sets}
The set $\partial^\pm \mintheta$, $T_\mintheta^\pm$, $F_\mintheta$, $\mathcal D^\pm \mintheta$, $C_\mintheta^\pm$, $R_\mintheta^{\pm,\ell}$, $R_\mintheta^\pm$, $R_\mintheta$ are $\sigma$-compact.
\end{proposition}

The next lemma follows easily from \eqref{E_cone_add_theta}: in the language of \cite{biadan}, we can say that the level sets of $\mintheta(\a)$ are complete $\tilde{\mathtt c}_\a$-Lipschitz graphs (\cite[Definition 4.1]{biadan}).

\begin{lemma}
\label{L_level_sets_contains_parallelogram}
If $(w,w') \in \partial^+ \mintheta(\a)$, then
\begin{equation}
	\label{E_Q_cone_minus_cone_intersection_z_z_prime}
	Q_{\mintheta,\a}(w,w') := (w + \tilde C_\a) \cap (w'- \tilde C_\a) \subset \partial^+ \mintheta(\a,w).
\end{equation}
Moreover,
\begin{equation*}
	\R^+\cdot (Q_{\mintheta,\a}(w,w') - w) = \R^+\cdot (w' - Q_{\mintheta,\a}(w,w')) =: O_\a(w,w')
\end{equation*}
where $O_\a(w,w')$ is the minimal extremal face of $\tilde C_\a$ containing $w'-w$.
\end{lemma}


In particular, one deduces that as in the potential case
\[
\partial^- \mintheta(\a,w) = \bigcup_{w' \in \partial^- \mintheta(\a,w)} Q_{\mintheta,\a}(w',w), \qquad \partial^+\mintheta(\a,w) = \bigcup_{w' \in \partial^+ \mintheta(\a,w)} Q_{\mintheta,\a}(w,w').
\]
Moreover, by \eqref{E_sum_of_cones_inside_cone},
\begin{equation*}
	w' \in \partial^\pm \mintheta(\a,w) \quad \Longrightarrow \quad \partial^\pm \mintheta(\a,w') \subset \partial^\pm \mintheta(\a,w).
\end{equation*}


Using the same ideas of the proof of Proposition \ref{Pdirectionfaces1_phi}, we obtain an analogous proposition for $\mathcal D^- \bar \theta$.

\begin{proposition}
\label{Pdirectionfaces1_minus}
Let $O_\a \subset \tilde C_\a$ be an extremal face. Then the following holds.

\begin{enumerate}
\item \label{item1Pdire_minus}
$
O_\a  \cap \{ t=1 \} \subset \mathcal D^- \mintheta(\a,w) \quad \Longleftrightarrow \quad \exists \delta > 0 \ \Big( B(w,\delta) \cap \big( w - O_\a \big) \subset \partial^- \mintheta(\a,w) \Big).
$

\item \label{item2Pdire_minus} If $O_\a \cap \{ t=1 \} \subset \mathcal D^- \mintheta(\a,w)$ is maximal w.r.t. set inclusion, then
\[
\forall w' \in B^h(w,\delta) \cap \big( w - \interr O_\a \big) \ \Big( \mathcal D^- \mintheta(\a,w') =  O_\a \cap \{ t=1 \} \Big),
\]
with $\delta > 0$ given by the previous point.

\item \label{item3Pdire_minus} The following conditions are equivalent:

\begin{enumerate}
\item \label{a_item3Pdire_minus} $\mathcal D^- \mintheta(\a,w) =  C^-_{\bar \theta}(\a,w) \cap \{ t=1 \}$;

\item \label{b_item3Pdire_minus} the family of cones
\[
\big\{ \R^+ \cdot Q_\a(w',w), w' \in \partial^- \mintheta(\a,w) \big\}
\]
has a unique maximum by set inclusion, which coincides with $C^-_{\bar \theta}(\a,w)$;

\item \label{c_item3Pdire_minus} $\partial^- \mintheta(\a,w) \cap (z - \interr C^-_{\bar \theta}(\a,w)) \not= \emptyset$;

\item \label{d_item3Pdire_minus} $\mathcal D^- \mintheta(\a,w) = \conv \mathcal D^- \mintheta(\a,w)$.

\end{enumerate}

\end{enumerate}

\end{proposition}

A completely similar proposition can be proved for $\mathcal D^+ \mintheta$. 

\begin{proposition}
\label{Pdirectionfaces1}
Let $O_\a \subset \tilde C_\a$ be an extremal face. 
Then the following holds.

\begin{enumerate}
\item \label{item1Pdire}
$
O_\a \cap \{ t=1 \} \subset \mathcal D^+ \mintheta(\a,w) \quad \Longleftrightarrow \quad \exists \delta > 0 \ \Big( B(w,\delta) \cap \big( w + O_\a \big) \subset \partial^+ \mintheta(\a,w) \Big).
$
\item \label{item2Pdire} If $O_\a  \cap \{ t=1 \}  \subset \mathcal D^+ \mintheta(\a,w)$ is maximal w.r.t. set inclusion, then
\[
\forall w' \in B(w,\delta) \cap \big( w + \interr O_\a \big) \ \Big( \mathcal D^+ \mintheta(\a,w') =  O_\a  \cap \{ t=1 \} \Big),
\]
with $\delta > 0$ given by the previous point.

\item \label{item3Pdire} The following conditions are equivalent:

\begin{enumerate}
\item \label{a_item3Pdire} $\mathcal D^+ \mintheta(\a,w) =  C_\mintheta^+(\a,w) \cap \{ t=1 \} $;

\item \label{b_item3Pdire} the family of cones
\[
\big\{ \R^+ \cdot Q_\a(w,w'), w' \in \partial^+ \mintheta(\a,w) \big\}
\]
has a unique maximum by set inclusion, which coincides with $C_\mintheta^+(\a,w)$;

\item \label{c_item3Pdire} $\partial^+ \mintheta(\a,w) \cap \interr (z + C_\mintheta^+(\a,w)) \not= \emptyset$;

\item \label{d_item3Pdire} $\mathcal D^+ \mintheta(\a,w) = \conv \mathcal D^+ \mintheta(\a,w)$.

\end{enumerate}

\end{enumerate}

\end{proposition}

As a consequence of Point \eqref{item3Pdire} of the previous propositions, we will call sometimes $C^-_{\bar \theta}(\a,w)$, $C^+_{\bar \theta}(\a,w)$ the \emph{maximal backward/forward extremal face}.

\begin{remark}
\label{R_vart_to_bar_vart}
In Section \ref{S_disintegration_on_affine} we will need to compute the same objects for the function $\bar \vartheta$. The definitions are exactly the same, as well as the statements of Propositions \ref{P_borel_regu_transport_sets}, \ref{Pdirectionfaces1_minus}, \ref{Pdirectionfaces1} and Lemma \ref{L_level_sets_contains_parallelogram}, just replacing the function $\bar \theta$ with $\bar \vartheta$. We thus will consider the sets
\begin{equation*}
\partial^\pm \bar \vartheta, \quad T^\pm_{\bar \vartheta}, \quad F_{\bar \vartheta}, \quad \mathcal D^\pm \bar \vartheta, \quad C^\pm_{\bar \vartheta}, \quad R^{\pm,\ell}_{\bar \vartheta}, \quad R^\pm_{\bar \vartheta}, \quad R_{\bar \vartheta}, \quad N_{\bar \vartheta},
\end{equation*}
and for the exact definition we refer to the analog for $\bar \theta$.
\end{remark}

\subsection{Partition of the transport set}
\label{Ss_partition_transport_set}

In this section we construct a map which give a directed locally affine partition in $\R^{d-h} \times [0,+\infty) \times \R^h$: more precisely, up to a residual set, we will find a directed locally affine partition on each fiber $\{\a\} \times [0,+\infty) \times \R^h$, and the dependence of this partition from the parameter $\a$ is $\sigma$-continuous.

Define the map
\begin{equation*}
\begin{array}{ccccc}
\mathtt v_\mintheta^- &:& R^-_{\bar \theta} &\to& \R^{d-h} \times \cup_{\ell = 0}^h \mathcal A(\ell,\R^h) \\ [.5em]
&& (\a,w) &\mapsto& \mathtt v_\mintheta^+(\a,w) := \big( \a,\aff\, \partial^- \mintheta(\a,w) \big)
\end{array}
\end{equation*}
\begin{lemma}
The map $\mathtt v_\mintheta^-$ is $\sigma$-continuous.
\end{lemma}

\begin{proof}
Since $\partial^- \mintheta(\a,w)$ is $\sigma$-continuous by Proposition \ref{P_borel_regu_transport_sets} and the map $A \mapsto \aff\, A$ is $\sigma$-continuous in the Hausdorff topology, the conclusion follows.
\end{proof}

We recall the convention $\R^0 = \N$.

\begin{theorem}
\label{TpartitionE-}
The map $\mathtt v_\mintheta^-$ induces a partition
\[
\bigcup_{\ell'=0}^h \Big\{ Z^{\ell',-}_{\a,\b'}, \a \in \R^{d-h}, \b' \in \R^{h-\ell'} \Big\}
\]
on $R^-_{\bar \theta}$ such that the following holds:
\begin{enumerate}
\item \label{item1TpartE-} each set $Z^{\ell',-}_{\a,\b'}$ is locally affine;

\item \label{item2TpartE-} there exists an extremal face $O^{\ell',-}_{\a,\b'}$ with dimension $\ell'$ of the cone $\tilde C_\a$ such that
\[
\forall w \in Z^{\ell',-}_{\a,\b'} \ \Big( \aff\,Z^{\ell',-}_{\a,\b'} = \aff (w + O^{\ell',-}_{\a,\b'}) \quad \wedge \quad \mathcal D^- \mintheta(\a,w) = O^{\ell',-}_{\a,\b'} \cap \{ t=1 \} \Big);
\]

\item \label{item3TpartE-} for all $w \in T^-_{\bar \theta}(\a)$ there exists $r > 0$, $O^{\ell',-}_{\a,\b'}$ such that
\[
B(w,r) \cap \big( w - \inter_\mathrm{rel} O^{\ell',-}_{\a,\b'} \big) \subset Z^{\ell',-}_{\a,\b'}.
\]
\end{enumerate}
\end{theorem}

The choice $\b \in \R^{h-\ell}$ is in the spirit of Proposition \ref{P_countable_partition_in_reference_directed_planes}.

\begin{proof}
	For the proof see \cite[Theorem 4.18]{biadan}.
\end{proof}

A completely similar statement holds for $R^+_{\bar \theta}$, by means of $\sigma$-continuous map 
\begin{equation*}
\begin{array}{ccccc}
\mathtt v_\mintheta^+ &:& R^+_{\bar \theta} &\to& \R^{d-h} \times \cup_{\ell = 0}^h \mathcal A(\ell,\RR^h) \\ [.5em]
&& (\a,w) &\mapsto& \mathtt v_\mintheta^+(\a,w) := \big( \a,\aff\, \partial^+ \mintheta(\a,w) \big)
\end{array}
\end{equation*}

\begin{theorem}
\label{TpartitionE+}
The map $\mathtt v_\mintheta^+$ induces a partition
\[
\bigcup_{\ell=0}^h \Big\{ Z^{\ell,+}_{\a,\b} \subset [0,+\infty) \times \R^h, \a \in \R^{d-h}, \b \in \R^{h-\ell} \Big\}
\]
on $R^+_{\bar \theta}$ such that the following holds:
\begin{enumerate}
\item \label{item1TpartE+} each set $Z^{\ell,+}_{\a,\b}$ is locally affine;

\item \label{item2TpartE+} there exists an extremal face $O^{\ell,+}_{\a,\b}$ with dimension $\ell$ of the cone $\tilde C_\a$ such that
\[
\forall w \in Z^{\ell,+}_{\a,\b} \ \Big( \aff\,Z^{\ell,+}_{\a,\b} = \aff ( z + O^{\ell,+}_{\a,\b}) \quad \wedge \quad \mathcal D^+ \mintheta(\a,w) =  O^{\ell,+}_{\a,\b} \cap \{ t=1 \}\Big);
\]

\item \label{item3TpartE+} for all $w \in T^+_{\bar \theta}(\a)$ there exists $r > 0$, $O^{\ell,+}_{\a,\b}$ such that
\[
B(w,r) \cap \big( w + \inter_\mathrm{rel}\, O^{\ell,+}_{\a,\b} \big) \subset Z^{\ell,+}_{\a,\b}.
\]
\end{enumerate}
\end{theorem}

In general $\ell \not= \ell'$, but on $R_{\bar \theta}$ the two dimensions (and hence the affine spaces $\mathtt p_2\, \mathtt v_\mintheta^\pm$) coincide.

\begin{proposition}
On the set $R_{\bar \theta}$ one has
\[
\mathtt v_\mintheta^-(\a,w) = \mathtt v_\mintheta^+(\a,w).
\]
\end{proposition}

\begin{proof}
	For the proof see \cite[Corollary 4.19]{biadan}.
\end{proof}
%
%
Define thus on $R_{\bar \theta}$
\begin{equation*}
\mathtt v_\mintheta := \mathtt v_\mintheta^+ \llcorner_{R_{\bar \theta}} = \mathtt v_\mintheta^- \llcorner_{R_{\bar \theta}},
\end{equation*}
and let
\begin{equation*}
\Big\{ Z^{\ell}_{\a,\b}, (\a,\b) \in \R^{d-h} \times \R^{h-\ell} \Big\}
\end{equation*}
be the partition induced by $\mathtt v_\mintheta$: since $R_{\bar \theta} = \cup_\ell R^{-,\ell}_{\bar \theta} \cap R^{+,\ell}_{\bar \theta}$, it follows that
\[
Z^{\ell}_{\a,\b} = Z^{\ell,-}_{\a,\b} \cap Z^{\ell,+}_{\a,\b},
\]
once the parametrization of $\mathcal A(\ell',\aff\,Z^\ell_\a)$ is fixed accordingly.

Finally, define the set $\tilde{\mathbf D}'$ by
\begin{align*}
	\tilde{\mathbf D}' :=&~ \Big\{ \big( \ell,\a,\b,w,C \big) : C = \mathtt p_2 \mathtt v_\mintheta(\a,w) \cap \tilde C_\a, w \in Z^\ell_{\a,\b} \Big\} \crcr
				 \subset&~ \bigcup_0^h \big\{ \{\ell\} \times \R^{d-h} \times \R^{h-\ell} \times ([0,\infty) \times \R^{\ell}) \times \mathcal C(\ell,\R^h) \big\}.
\end{align*}

\begin{lemma}
\label{L_tilde_mathbf_D_prime_sigma_compact}
The set $\tilde{\mathbf D}'$ is $\sigma$-compact.
\end{lemma}

\begin{proof}
Since $\mathtt v_\mintheta$, $\a \mapsto \tilde C_\a$ are $\sigma$-continuous, the conclusion follows.
\end{proof}

We thus conclude that $\tilde{\mathbf D}'$ corresponds the following directed locally affine partition of $\R^d$:
\begin{align}
\label{E_hat_mathbf_D_prime_def}
\hat{\mathbf D}' := \Big\{ (\ell,\c,z,C), \c = (\a,\b), z = (\a,w): (\ell,\a,\b,w,C) \in \tilde{\mathbf D}' \Big\}.
\end{align}
We will use the notation $\c = (\a,\b)$ and
\begin{equation*}
Z^\ell_\c := Z^\ell_{\a,\b} = \mathtt p_z \tilde{\mathbf D}'(\ell,\c), \quad C^\ell_\c := O^\ell_{\a,\b} = \mathtt p_C \tilde{\mathbf D}'(\ell,\c), \quad \tilde{\mathbf Z}'^\ell := \mathtt p_{\a,w} \tilde{\mathbf D}'(\ell),
\end{equation*}
where $O^\ell_{\a,\b}$ is the extremal fact of $\tilde C_\a$ corresponding to the space $\mathrm{span}\, \mathcal D \mintheta(\a,w)$.

\section{Disintegration on directed locally affine partitions}
\label{S_disintegration_on_affine}

In this section we show how to use the function $\bar \vartheta$ in order to prove that the directed locally affine partition $\tilde{\mathbf D}$ is regular w.r.t. the measure $\mathcal H^d \llcorner_{\{t=1\}}$. This is the main difference w.r.t. the analysis of \cite{biadan}, where the regularity is proved w.r.t. the measure $\mathcal L^{d+1}$. 


Let thus $\minvartheta$ be the upper semi continuous envelope of $\mintheta$.


\begin{lemma}
	\label{L_Lax_forlu_precise}
	For all $s\geq 0$, the following holds: if $t\geq s$, then
	\begin{equation*}
		{\minvartheta}(\a, t,x) = \max \Big\{ {\minvartheta}(\a, s,y) - \ind_{\tilde C_\a}(t-s,x-y), y \in \R^h \Big\}.
	\end{equation*}
\end{lemma}

\begin{proof}
Recalling that $\bar \theta$ is defined by
\[
\mintheta(\a,t,x) = \sup \Big\{ \mintheta(\a,s,y) - \ind_{\tilde C_\a} (t-s,x-y): y \in \R^h \Big\},
\]
the proof follows immediately by considering a sequence of maximizers $y_n$ for $\bar \theta(\a,t,x)$.
\end{proof}

Since $\bar \vartheta$ satisfies
\begin{equation}
\label{E_tilde_c_comp_varth}
\bar \vartheta(\a,w + \tilde C_\a) \geq \bar \vartheta(\a,w)
\end{equation}
i.e. it is a \emph{complete $\tilde c$-Lipschitz foliation} according to \cite{biadan}, the same completeness property \eqref{E_Q_cone_minus_cone_intersection_z_z_prime} holds: if $(w,w') \in \partial^+ \minvartheta(\a)$, then
	\begin{equation*}
		Q_{\minvartheta,\a}(w,w') := (w + \tilde C_\a) \cap (w'- \tilde C_\a) \subset \partial^+ \minvartheta(\a,w).
	\end{equation*}

Recalling for the notations Remark \ref{R_vart_to_bar_vart}, a first connection between $\bar \theta$ and $\bar \vartheta$ is shown in the following lemma.

\begin{lemma}
\label{L_ray_vartheta_theta}
If $\minvartheta(\a, t,x) = \mintheta(\a, t,x)$, then $\partial^- \mintheta(\a, t,x ) \subset \partial^- \minvartheta(\a, t,x)$.
\end{lemma}

\begin{proof}
Let $(s,y) \in \partial^- \mintheta(\a, t,x)$, so that $\mintheta(\a, s,y) =  \mintheta(\a, t,x)$. The inclusion $\partial^- \mintheta(\a,t,x) \subset \partial^- \minvartheta(\a,t,x)$ follows from the estimate:
\begin{equation*}
\mintheta(\a, s,y) \leq \minvartheta(\a, s,y) \leq \minvartheta(\a, t,x) = \mintheta(\a, t,x).
\end{equation*}
This concludes the proof. 
\end{proof}

\subsection{\texorpdfstring{Regularity of the partition $\tilde{\mathbf D}'$}{Regularity of the locally affine partition}}
\label{Ss_back_forw_regul}


The proof to show the regularity of $\mathcal H^d \llcorner_{\{t = 1\}}$-a.e. point for $\bar \vartheta$ is very similar to the analysis done in Section \ref{Ss_back_forw_regul_phi}: the two proofs differ because we have now to consider a family of HJ equations (one for each $\a \in \tilde \A$), and that the Lagrangian is the indicator function of a cone $\tilde C_\a$.

Once we have the regularity for $\bar \vartheta$, we use the fact that $\bar \theta(\bar t) = \bar \vartheta(\bar t)$ for $\mathcal H^d \llcorner_{\{t = \bar t\}}$-a.e. point in order to deduce that the same regularity holds for $\bar \theta$.

%
%

Consider a Borel bounded set $S \subset \{ t=\bar t \}$ made of backward regular points for $\minvartheta$. Since by the definition of $\bar \vartheta$ each point has an optimal ray reaching $t=0$, for all $s>0$ we can find \emph{inner} optimal rays, i.e. with directions belonging to the interior of $C^-_{\bar \vartheta}$.

\begin{lemma}
\label{L_inner_area_estimate}
Let $\bar t> s > \varepsilon>0$. Then for every $(\bar t, x)\in S$ there exists a point
\begin{equation*}
\sigma_s(\bar t,x) \in \interr \big( \partial^-\minvartheta(\a,\bar t, x)\cap\{t=s\} \big)
\end{equation*}
such that
\[
\mathcal H^h (\sigma_s (S)) \geq \bigg( \frac{s-\varepsilon}{\bar t-\varepsilon} \bigg)^h \mathcal H^h(S).
\]
\end{lemma}
\begin{proof}
For each fixed $\a$ the proof is the same as the one of $\bar \phi$, just replacing it with $\bar \vartheta$. In particular we obtain that for each fixed $\varepsilon > 0$ every point $z \in S$ has a cone of optimal backward directions $K_z$ such that
\begin{equation*}
K_z \in \mathcal C(\ell,\RR^h) \qquad \text{and} \qquad (z - K_z) \cap \{t=\varepsilon\} \subset \partial^- \bar \vartheta(z) \cap \{t=\varepsilon\},
\end{equation*}
where $\ell = \dim\,\mathcal D^- \bar \vartheta(z)$.

As in the proof of Lemma \ref{L_inner_area_estimate_phi}, we can thus partition the set $S$ according to the requirement that the projection of $K$ on a $(\ell+1)$-dimensional reference plane $V'$ contains a reference cone $K' \in \mathcal C(\ell,\RR^d)$.

Slicing the problem on $(d-\ell)$-dimensional planes $V''$ transversal to $V'$, it follows that
\begin{equation*}
\sigma_s(z) := \bar \vartheta(z) \cap V'' \cap \{t=s\}
\end{equation*}
is singleton for all $z \in S$. We can then use the same approach used in \cite[Section 8]{biadan}.

Consider the two measures $\hat \mu := \mathcal H^d \llcorner_S$ and its image measure $\hat \nu := (\sigma_s)_\sharp \hat \mu$. By \eqref{E_tilde_c_comp_varth} and Proposition \ref{P_equiv_all_plan} applied to $\bar \vartheta$, every transport $\hat \pi \in \Pi^f_{\tilde{\mathtt c}}(\hat \mu,\hat \nu)$ with finite cost w.r.t. $\tilde{\mathtt c}$ occurs on the level sets of $\bar \vartheta$: in particular, in each plane $V''$ there exists a unique transference plan with finite cost.

We can then use \cite[Lemma 8.4]{biadan} in order to obtain a family of cone vector fields converging to $\sigma_s$ for $\mathcal H^d$-a.e. point. Being the area estimate
\begin{equation*}
\mathcal H^h (\sigma^n_s (S)) \geq \bigg( \frac{s-\varepsilon}{\bar t-\varepsilon} \bigg)^h \mathcal H^h(S)
\end{equation*}
u.s.c. w.r.t. pointwise convergence $\sigma^n_s(z) \to \sigma_s(z)$ \cite[Lemma 5.6]{biadan}, we obtain the statement.
%
%
%
\end{proof}
%
%

We can now repeat the same proof of Proposition \ref{P_back_regul_phi} in order to obtain the regularity of $\mathcal H_{\{t=\bar t\}}$-a.e. point. The only variation w.r.t. the proof of Proposition \ref{P_back_regul_phi} is that we have to use the regularity of the disintegration of $\mathcal L^{d+1}$ on the directed locally affine partition $\tilde{\mathbf D}_{\bar \vartheta} = \{Z^\ell_{\a,\b}(\bar \vartheta),C^\ell_{\a,\b}(\bar \vartheta)\}_{\ell,\a,\b}$ induced by $\bar \vartheta$ through the map
\begin{equation*}
\begin{array}{ccccc}
\mathtt v_{\bar \vartheta} &:& R_{\bar \vartheta} &\to& \R^{d-h} \times \cup_{\ell = 0}^h \mathcal A(\ell,\R^h) \\ [.5em]
&& (\a,w) &\mapsto& \mathtt v_{\bar \vartheta}(\a,w) := \big( \a,\aff\, \partial^- \bar \vartheta(\a,w) \big)
\end{array}
\end{equation*}
(The fact that $\mathtt v_{\bar \vartheta}$ induces a locally affine directed partition is the same statement of Theorem \ref{T_decomp_phi_R-} or Theorem \ref{T_decomp_phi_R+}, see Remark \ref{R_vart_to_bar_vart}.) \\
The regularity of $\tilde{\mathbf D}_{\vartheta}$ is one of the fundamental results of \cite{biadan}, Theorem 8.1 and Corollary 8.2:

\begin{theorem}[Theorem 8.1 and Corollary 8.2 of \cite{biadan}]
\label{T_regula}
If $\{Z^\ell_{\a,\b}(\bar \vartheta),C^\ell_{\a,\b}(\bar \vartheta)\}_{\ell,\a,b}$ is the locally affine partition induced by the function $\bar \vartheta$, then $\mathcal L^{d+1}$-a.e. point is regular and the disintegration of $\mathcal L^{d+1}$ w.r.t. $\{Z^\ell_{\a,\b}(\bar \vartheta)\}_{\ell,\a,\b}$ is regular. \\
A similar statement holds for the directed locally affine partition $\tilde{\mathbf D}' = \{Z^\ell_{\a,\b},C^\ell_{\a,\b}\}_{\ell,\a,\b}$ obtained through the function $\bar \theta$.
\end{theorem}

Once we are given that $\mathcal L^{d+1}$-a.e. point is regular for $\bar \vartheta$, replacing the area estimate Lemma \ref{L_inner_area_estimate_phi} with Lemma \ref{L_inner_area_estimate} in the proof of Proposition \ref{P_back_regul_phi} yields the following result. 

\begin{proposition}
\label{P_back_regul_vartheta}
For all $\bar t>0$, the set of regular points for $\bar \vartheta$ in $\{t = \bar t\}$ is $\mathcal H^d \llcorner_{\{t=\bar t\}}$-conegligible.	
\end{proposition}

%
%
%
%
%
%

We now transfer the regularity w.r.t. $\bar \vartheta$ to the regularity w.r.t. $\bar \theta$. The sketch of the proof is as follows: since by Fubini theorem, for $\mathcal H^1$-a.e. $\bar t$ it holds that $\mathcal H^d \llcorner_{\{t=\bar t\}}$-a.e. point is regular for $\bar \theta$, and the same occurs for $\bar \vartheta$, we can use the fact that $\bar \theta(t,x) = \bar \vartheta(t,x)$ for $\mathcal H \llcorner_{\{t=\bar t\}}$-a.e. $x$ and every $\bar t>0$ in order to obtain that the points $z,\sigma_s(z)$ used in Lemma \ref{L_inner_area_estimate} are regular points for $\bar \theta$. The key observation is that the inner rays for $\bar \vartheta$ will be also inner rays for $\bar \theta$. 

\begin{proposition}
	\label{P_back_regul_theta}
	If $\bar t>0$ then $\mathcal H^h \llcorner_{\{t=\bar t\}}$-a.e. point is regular for $\mintheta$. 
\end{proposition}


\begin{proof}
By Theorem \ref{T_regula} we can fix $\varepsilon' > 0$ such that $\mathcal H^d_{t = \bar t \pm \varepsilon'}$-a.e. point is regular for both $\bar \theta$ and $\bar \vartheta$. By the area estimate of Lemma \ref{L_inner_area_estimate}, we can also assume that
\begin{equation*}
S \subset R_{\bar \theta} \cap R_{\bar \vartheta} \cap \{t = \bar t+ \epsilon\} \cap \{\bar \theta = \bar \vartheta\}
\end{equation*}
and for all $z \in S$
\begin{equation*}
\sigma_{\bar t - \varepsilon'}(z) \in R_{\bar \theta} \cap R_{\bar \vartheta} \cap \{t = \bar t - \varepsilon'\} \cap \{\bar \theta = \bar \vartheta\}.
\end{equation*}
In particular we deduce that
\begin{equation}
\label{E_same_val_varth_th}
\sigma_{\bar t- \varepsilon'}(z) \in \partial^- \bar \theta(z).
\end{equation}

If $z - \sigma_s(z)$ belongs to an inner direction of
\begin{equation*}
C_{\bar \theta}(z) = C_{\bar \theta}^+(z) = C^-_{\bar{\theta}}(z),
\end{equation*}
then the same observation at the end of the proof of Proposition \ref{P_back_regul_phi} will give immediately the statement: the arbitrariness of $\varepsilon'$ is used as in the proof of the proposition in order to obtain that $\mathcal H^d \llcorner_{\{t=1\}}$-a.e. point is regular.

We thus left with proving this last property of $\sigma_s(z)$, i.e. $z - \sigma_s(z) \in C_{\bar \theta}$.

Being $\sigma_s(z)- z$ an inner ray of $- C_{\bar \vartheta}(z)$ and $C_{\bar \theta}(z)$, $C_{\bar{\vartheta}}$ extremal cones of $\tilde C_{\mathtt p_\a z}$, it follows that if $C_{\bar \theta}(z) \subsetneq C_{\bar{\vartheta}}$, by the extremality property then for $s < \bar t + \varepsilon'$
\begin{equation*}
(z - \sigma_s(z)) \cap C_{\bar \theta}(z) = \emptyset,
\end{equation*}
contradicting \eqref{E_same_val_varth_th}.
%
%
%
%
%
%
%
\end{proof}

The proof of the regularity of the disintegration of $\mathcal H^d \llcorner_{\{t=1\}}$ w.r.t. the partition $\{Z^\ell_{\a,\b}\}_{\ell,\a,\b}$ follows the same line of Section \ref{Ss_regul_disi_phi}: just replace the use of Lemma \ref{L_inner_area_estimate_phi} with Lemma \ref{L_inner_area_estimate}.

The statement is analogous.

\begin{proposition}
\label{P_ac_disi_bar_theta}
 The disintegration 
		 \begin{equation*}
		 \mathcal H^d \llcorner_{\cup_{\ell,\a,\b} Z^\ell_{\a,\b} \cap \{t=1\}} = \sum_\ell \int v^\ell_{\a,\b} \eta^\ell(d\a d\b)
		 \end{equation*}
		 w.r.t. the partition $\{Z^h_{\a,\b} \cap\{t=1\} \}_{h,\a,\b}$ is regular:
		\[v^h_{\a,\b} \ll \mathcal H^\ell \llcorner_{Z^\ell_{\a,\b}}.\]
\end{proposition}

\subsection{Proof of Theorem \ref{T_step_theorem}}
\label{Ss_proof_th_step}

Let $\hat{\mathbf D}' = \{Z^\ell_\c,C^\ell_\c\}_{\ell,\c}$ be the directed locally affine given by \eqref{E_hat_mathbf_D_prime_def}. By Corollary \ref{P_back_regul_theta} we have that the complement of $\cup_{\ell,\c} Z^\ell_\c$ is $\mathcal H^d \llcorner_{\{t=1\}}$-negligible; Proposition \ref{P_ac_disi} yields that the disintegration is regular.

Consider now the map $\check{\mathtt r}$ defined in \eqref{E_check_mathbf_r_def}: $\check{\mathtt r}$ is invertible on $\hat{\mathbf D}'$ and, as observed at the end of Section \ref{Ss_mapping_sheaf_to_fibration}, the two measures $\check{\mathtt r}^{-1} \mathcal H^\ell \llcorner_{Z^\ell_\c}$ and $\mathcal H^{\ell} \llcorner_{\check{\mathtt r}^{-1}(Z^\ell_\c)}$ are equivalent. Hence the first three points of Theorem \ref{T_step_theorem} follows: namely
\begin{itemize}
\item the sets
\begin{equation*}
Z^{h,\ell}_{\a,\b} := \check{\mathtt r}^{-1}(Z^{\ell}_{\c}) \subset Z^h_\a
\end{equation*}
has affine dimension $\ell+1$ and
\begin{equation*}
C^{h,\ell}_{\a,\b} := \check{\mathtt r}^{-1}(C^{\ell}_{\c})
\end{equation*}
is an $(\ell+1)$-dimensional extremal cone of $C^h_\a$, $\ell \leq h$ and $\c = (\a,\b)$;
\item $\bar \mu(\cup_{h,\ell,\a,\b} Z^{h,\ell}_{\a,\b}) = 1$;
\item the disintegration of $\mathcal H^d \llcorner_{\{t = 1\}}$ w.r.t. the partition $\{Z^{h,\ell}_{\a,\b}\}_{h,\ell,\a,\b}$ is regular; 
\end{itemize}

Since the preorder $\bar \preccurlyeq$ induced by $\bar \theta$ is Borel and $\tilde \pi'(\bar \preccurlyeq) = 1$, then by Theorem \ref{T_uniqueness} the transference plan is concentrated on the diagonal $\bar E = \bar \preccurlyeq \cap \bar \preccurlyeq^{-1}$. Hence by construction the transference of mass occurs along the equivalence classes: these directions are exactly the optimal rays defined by \eqref{E_sub_diff_bar_theta}. On the regular set by definition these directions are the extremal cone $C^\ell_\c$ in $Z^\ell_\c$:
\begin{itemize}
\item if $\bar \pi \in \Pi(\bar \mu,\{\underline{\bar \nu}^h_\a\})$ with $\underline{\bar \nu}^h_\a = \mathtt p_2 \underline{\bar \pi}^h_\a$, then $\bar \pi$ satisfies \eqref{E_finite_cone_cost} iff
\begin{equation*}
\bar \pi  = \sum_{h,\ell} \int \bar \pi^{h,\ell}_{\a,\b} m^{h,\ell}(d\a d\b), \qquad \int \ind_{C^{h,\ell}_{\a,\b}}(x-x') \bar \pi^{h,\ell}_{\a,\b} < \infty;
\end{equation*}
\end{itemize}

The indecomposability of the sets $Z^{h,\ell}_{\a,\b}$ with $\ell = h$ is a consequence of Proposition \ref{P_minimality_on_class} and Lemma \ref{L_indecompo_mini}, stating that the function $\theta$ constructed by a given $\Gamma$ is constant on $Z^{h,h}_{\a,\b}$ and the sets $Z^{h,h}_{\a,\b}$ are indecomposable. 
This proves Point \eqref{Point_step_th_indeco} of Theorem \ref{T_step_theorem}:
\begin{itemize}
\item \label{Point_step_th_indeco_pr} if $\ell = h$, then for every carriage $\Gamma$ of $\bar \pi \in \Pi(\bar \mu,\{\underline{\bar \nu}^h_\a\})$ there exists a $\bar \mu$-negligible set $N$ such that each $Z^{h,h}_{\a,\b} \setminus N$ is $\ind_{C^{h,h}_{\a,\b}}$-cyclically connected.
\end{itemize}

In the next section we will use this theorem in order to prove Theorem \ref{T_main_theor}.

\section{Proof of Theorem \ref{T_main_theor}}
\label{S_transl_orig}

By Theorem \ref{T_potential_deco} we have a first directed locally affine decomposition $\mathbf D_{\bar \phi}$, and by Theorem \ref{T_step_theorem} a method of refining a given locally affine partition in order to obtain indecomposable sets or lower the dimension of the sets by at least $1$. It is thus clear that after at most $d$ steps we obtain a locally affine decomposition $\{Z^h_\a,C^h_\a\}$ with the properties stated in Point \eqref{Point_step_th_indeco} of Theorem \ref{T_step_theorem}. 

\begin{theorem}
\label{T_cone_ theorem}
Given a transference plan $\underline{\bar \pi} \in \Pi(\bar \mu,\bar \nu)$ optimal w.r.t. the cost $\tilde{\mathtt c}$, then there is a directed locally affine partition $\bar{\mathbf D} = \{Z^h_\a,C^h_\a\}_{h,\a}$ such that 
\begin{enumerate}
\item $Z^h_\a$ has affine dimension $h+1$ and $C^{h}_{\a}$ is an $(h+1)$-dimensional proper extremal cone of $\epi\,\bar{\mathtt c}$; moreover $\aff\,Z^h_a = \aff(z + C^h_\a)$ for all $z \in Z^h_\a$;
\item $\mathcal H^d(\{t=1\} \setminus \cup_{h,\a} Z^{h}_{\a}) = 0$;
\item the disintegration of $\mathcal H^d \llcorner_{\{t = 1\}}$ w.r.t. the partition $\{Z^h_\a\}_{h,\a}$ is regular, i.e.
\begin{equation*}
\mathcal H^d \llcorner_{\{t=1\}} = \sum_h \int \xi^h_\a \eta^h(d\a), \qquad \xi^h_\a \ll \mathcal H^h \llcorner_{Z^h_\a \cap \{t=1\}};
\end{equation*}
\item if $\bar \pi \in \Pi(\bar \mu,\{\underline{\bar \nu}^h_\a\})$ with $\underline{\bar \nu}^h_\a = \mathtt p_2 \underline{\bar \pi}^h_\a$, then $\bar \pi$ is optimal iff
\begin{equation*}
\bar \pi  = \sum_{h} \int \bar \pi^{h}_{\a} m^{h}(d\a), \qquad \int \ind_{C^{h}_{\a}}(z-z') \bar \pi^h_\a < \infty;
\end{equation*}
\item \label{Point_5_tt_c_indec} for every carriage $\Gamma$ of any $\bar \pi \in \Pi(\bar \mu,\{\underline{\bar \nu}^h_\a\})$ there exists a $\bar \mu$-negligible set $N$ such that each $Z^{h}_{\a,\b} \setminus N$ is $\ind_{C^{h}_{\a}}$-cyclically connected.
\end{enumerate}
\end{theorem}

The last step is to project back the decomposition for $\{t=1\} \times \R^d$ to $\R^d$, and cut the cones $C^h_\a$ at $\{t=1\}$.
\begin{itemize}
\item Take $S^h_\a := Z^h_\a \cap \{t=1\}$ and $O^h_\a = - C^h_\a \cap \{t=1\}$: the minus sign is because of the definition of $\bar{\mathtt c}$ in formula \eqref{E_def_bar_tt_c}. By the transversality of $C^h_\a$ if follows that $\dim\,O^h_\a = \dim\,O^h_\a = h$. Since $Z^h_\a$ is parallel to $C^h_\a$, then $O^h_\a$ is parallel to $S^h_\a$. Being $C^h_\a$ an extremal cone of $\bar{\mathtt c}$ and the latter $1$-homogeneous, it follows that $O^h_\a$ is an extremal face of $\mathtt c$.
\item The fact that the partition cover $\mathcal L^d$-a.e. point and that the disintegration is regular are straightforward.
\item Being $\bar \mu(\{t=1\}) = \bar \nu(\{t=0\}) = 1$, then it is clear that we can assume that every carriage $\Gamma$ is a subset of $\{t=1\} \times \{t=0\}$. This implies that when computing the cyclical indecomposability we use only vectors in $O^h_\a$, and thus the last point of Theorem \ref{T_main_theor} follows from Point \eqref{Point_5_tt_c_indec} of Theorem \ref{T_cone_ theorem}.
\end{itemize}
This concludes the proof of Theorem \ref{T_main_theor}.

\subsection{\texorpdfstring{The case of $\nu \ll \mathcal L^d$}{The case of a.c. second marginal}}
\label{Ss_further_remk}

In general the end points of optimal rays are in $Z^h_\a + C^h_\a \cap \{t=1\}$, which in general is larger that $Z^h_\a$. As an example, one can consider the case $\nu = \delta_{x=0}$, and verify that $Z^h_a = C^h_\a \setminus \{0\}$. However in the case $\nu \ll \mathcal L^d$, the partition is independent of $\pi$, i.e. following \cite{biadan} we call it \emph{universal}.

The key observation is that we can replace the roles of the measures $\bar \mu$, $\bar{\nu}$, obtaining then a decomposition $\{W^{h'}_{\a'},C^{h'}_{\a'}\}_{h',\a'}$ for $\{t=0\}$. Now recall that along optimal rays $\bar \theta$ is constant: being inner ray of the cones $C^h_\a$, $C^{h'}_{\a'}$, then it follows that $C^h_\a = C^{h'}_{\a'}$, and from this it is fairly easy to see that $Z^h_\a = W^{h'}_{\a'}$. In particular, for each optimal transference plan $\bar \pi$ it follows that its second marginals are given by the disintegration of $\bar \nu$ on $Z^h_\a$, i.e. they are independent of $\bar \pi$.

Translating this decomposition into the original setting, we can thus strengthen Theorem \ref{T_main_theor} as follows.

\begin{theorem}
\label{T_main_theor_ac}
Let $\mu,\nu \ll \mathcal L^d$. Then there exists a family of sets $\{S^h_\a,O^h_\a\}_{\nfrac{h = 0,\dots,d}{\a \in \A^h}}$ in $\R^d$ such that the following holds:
\begin{enumerate}
\item $S^h_\a$ is a locally affine set of dimension $h$;
\item $O^h_\a$ is a $h$-dimensional convex set contained in an affine subspace parallel to $\aff\,S^h_\a$ and given by the projection on $\R^d$ of a proper extremal face of $\epi\,\mathtt c$; 
\item $\mathcal L^d(\R^d \setminus \cup_{h,\a} S^h_\a) = 0$;
\item \label{Point_disint_lebe_main_ac} the partition is Lebesgue regular;
\item if $\pi \in \Pi(\mu,\nu)$ then optimality in \eqref{E_original_transp} is equivalent to 
\begin{equation*}
\sum_h \int \bigg[ \int \ind_{O^h_\a}(x'-x) \pi^h_\a(dxdx') \bigg] m^h(d\a) < \infty,
\end{equation*}
where $\pi = \sum_h \int_{\A^h} \pi^h_\a m^h(d\a)$ is the disintegration of $\pi$ w.r.t. the partition $\{S^h_\a \times \R^d\}_{h,\a}$;
\item \label{Point_indecomp_main_th_ac} for every carriage $\Gamma$ of $\pi \in \Pi(\mu,\nu)$ there exists a $\mu$-negligible set $N$ such that each $S^h_\a \setminus N$ is $\ind_{O^h_\a}$-cyclically connected.
\end{enumerate}
\end{theorem}

\newpage

\newpage

\appendix

\section{Equivalence relations, Disintegration and Uniqueness}
\label{A_appendix_1}

The following theorems have been proved in Section 4 of \cite{BiaCar}. For a more comprehensive analysis, see \cite{MR2462372}. 

\subsection{Disintegration of measures}
\label{Aa_disi_measure}

Let $E$ be an equivalence relation on $X$, and let $\mathtt h : X \mapsto X/E$ be the quotient map. The set $\mathfrak A := X/E$ can be equipped with the $\sigma$-algebra
\[
\mathtt A := \Big\{ A \subset \mathfrak A : \mathtt h^{-1}(A) \in \mathcal B(X) \Big\}.
\]

Let $\mu \in \mathcal P(X)$, and define $\xi := \mathtt h_\sharp \mu$.

A \emph{disintegration of $\mu$ consistent with $E$} is a map $\mathfrak A \ni \mathfrak a \mapsto \mu_\mathfrak a \in \mathcal P(X)$ such that
\begin{enumerate}
\item for all $B \in \mathcal B(X)$ the function $\mathfrak a \mapsto \mu_\mathfrak a(B)$ is $\xi$-measurable,

\item for all $B \in \mathcal B(X)$, $A \in \mathtt A$
\[
\mu(B \cap \mathtt h^{-1}(A)) = \int_A \mu_\mathfrak a(B) \xi(d\mathfrak a).
\]
\end{enumerate}

The disintegration is \emph{unique} if the \emph{conditional probabilities} $\mu_\mathfrak a$ are uniquely defined $\xi$-a.e.. It is \emph{strongly consistent} if $\mu_\mathfrak a(E_\mathfrak a) = 1$.

\begin{theorem}
\label{T_disint_gener}
Under the previous assumptions, there exists a unique consistent disintegration.

If the image space is a Polish space and $\mathtt h$ is Borel, then the disintegration is strongly consistent.
\end{theorem}

\subsection{Linear preorders and uniqueness of transference plans}
\label{Aa_linear_pre_unique_transf}

We now recall some results about uniqueness of transference plans. Let $A \subset X \times X$ be a Borel set such that
\begin{enumerate}
\item \label{Cond_1_linear_preorder} $A \cup A^{-1} = X$, where
\[
A^{-1} = \big\{ (x,x') : (x',x) \in A \big\};
\]
\item \label{Cond_2_transit_preorder} $(x,x'),(x',x'') \in A \ \Rightarrow \ (x,x'') \in A$.
\end{enumerate}
We will say that $A$ is (the graph of) a \emph{preorder} if Condition \eqref{Cond_2_transit_preorder} holds, and a \emph{linear preorder} if all points are comparable (Condition \ref{Cond_1_linear_preorder}). It is easy to see that
\[
E := A \cap A^{-1}
\]
is an equivalence relation. Let $\mathtt h : X \mapsto X/E$ be a quotient map.

\begin{theorem}
If $\mu \in \mathcal P(X)$, then the disintegration of $\mu$ w.r.t. $E$ is strongly consistent:
\[
\mu = \int \mu_\mathfrak a \xi(d\mathfrak a), \quad \xi := \mathtt h_\sharp \mu, \ \mu_\mathfrak a(E_\mathfrak a) = 1.
\]
\end{theorem}

Let $\bar \pi \in \mathcal P(X \times X)$ such that $\bar \pi(E) = 1$, and let $\bar \mu := (\mathtt p_1)_\sharp \bar \pi$, $\bar \nu := (\mathtt p_2)_\sharp \bar \pi$ be its marginals. Consider the disintegration
\[
\bar \pi = \int \bar \pi_\mathfrak a \bar \xi(d\mathfrak a), \quad \bar \xi = (\mathtt h \circ \mathtt p_1)_\sharp \bar \pi.
\]
Let $\bar \mu_\mathfrak a$, $\bar \nu_\mathfrak a$ be the conditional probabilities of $\bar \mu$, $\bar \nu$ w.r.t. $E$:
\[
\bar \mu = \int \bar \mu_\mathfrak a \bar \xi(d\mathfrak a) = \int (\mathtt p_1)_\sharp \bar \pi_\mathfrak a \bar \xi(d\mathfrak a), \quad \bar \nu = \int \bar \nu_\mathfrak a \bar \xi(d\mathfrak a) = \int (\mathtt p_2)_\sharp \bar \pi_\mathfrak a \bar \xi(d\mathfrak a),
\]

\begin{theorem}
\label{T_uniqueness}
If $\pi \in \Pi(\bar \mu,\bar \nu)$ satisfies
\[
\int \ind_A \pi < +\infty,
\]
then $\pi(E) = 1$, and moreover the disintegration of $\pi$ on $E$ satisfies
\[
\pi = \int \pi_\mathfrak a \bar \xi(d\mathfrak a), \quad \pi_\mathfrak a \in \Pi(\bar \mu_\mathfrak a,\bar \nu_\mathfrak a).
\]
\end{theorem}

\subsection{Minimality of equivalence relations}
\label{Aa_minimal_equivalence}

Consider a family of equivalence relations on $X$,
\[
\mathcal E = \Big\{ E_\mathfrak e \subset X \times X, \mathfrak e \in \mathfrak E \Big\}.
\]
Assume that $\mathcal E$ is closed under countable intersection
\[
\{E_{\mathfrak e_i}\}_{i \in \N} \subset \mathcal E \quad \Rightarrow \quad \bigcap_{i \in \N} E_{\mathfrak e_i} \in \mathcal E,
\]
and let $\mu \in \mathcal P(X)$.

By Theorem \ref{T_disint_gener}, we can construct the family of disintegrations
\[
\mu = \int_{\mathfrak A_\mathfrak e} \mu_\mathfrak a \xi_\mathfrak e(d\mathfrak a), \quad \mathfrak e \in \mathfrak E.
\]

\begin{theorem}
\label{T_minimal_equival}
There exists $E_{\bar{\mathfrak e}} \in \mathfrak{E}$ such that for all $E_\mathfrak e$, $\mathfrak e \in \mathcal{E}$, the following holds:
\begin{enumerate}
\item
\label{item:poifvojfivfi}
if $\mathtt A_\mathfrak e$, $\mathtt A_{\bar{\mathfrak e}}$ are the $\sigma$-subalgebras of the Borel sets of $X$ made of the saturated sets for $E_\mathfrak e$, $E_{\bar{\mathfrak e}}$ respectively, then for all $A \in \mathtt A_\mathfrak e$ there is $A' \in \mathtt A_{\bar{\mathfrak e}}$ s.t.~ $\mu(A \vartriangle A') = 0$;
\item 
\label{item:gfoaifvfffv}
if $\xi_\mathfrak e$, $\xi_{\bar{\mathfrak e}}$ are the restrictions of $\mu$ to $\mathtt A_\mathfrak e$, $\mathtt A_{\bar{\mathfrak e}}$ respectively, then $\mathtt A_\mathfrak e$ can be embedded (as measure algebra) in $\mathtt A_{\bar{\mathfrak e}}$ by Point \eqref{item:poifvojfivfi}: let
\[
\xi_{\bar{\mathfrak e}}  = \int \xi_{\bar{\mathfrak e},\mathfrak a} \xi_\mathfrak e(d\mathfrak a)
\]
be the disintegration of $\xi_{\bar{\mathfrak e}}$ consistent with the equivalence classes of $\mathtt A_\mathfrak e$ in $\mathtt A_{\bar{\mathfrak e}}$.
\item
\label{item:gfoaivdasvds}
If
\[
\mu = \int \mu_{\mathfrak e,\mathfrak a} \xi_\mathfrak e(d\mathfrak a), \quad \mu = \int \mu_{\bar{\mathfrak e},\beta} \xi_{\bar{\mathfrak e}}(d\beta)
\]
are the disintegration consistent with $E_\mathfrak e$, $E_{\bar{\mathfrak e}}$ respectively, then
\begin{equation*}
\mu_{\mathfrak e,\mathfrak a} = \int \mu_{\bar{\mathfrak e},\mathfrak b} \xi_{\bar{\mathfrak e},\mathfrak a}(d\mathfrak b).
\end{equation*}
for $\xi_\mathfrak e$-.a.e.~$\mathfrak a$.
\end{enumerate}
\end{theorem}

In particular, assume that each $E_\mathfrak e$ is given by
\[
E_\mathfrak e = \{ \theta_\mathfrak e(x) = \theta_\mathfrak e(x') \}, \quad \theta_\mathfrak e : X \to X', \ X' \ \text{Polish}, \ \theta_\mathfrak e \ \text{Borel}.
\]

\begin{corollary}
\label{C_constant_for_minimal_equivalence}
There exists a $\mu$-conegligible set $F \subset X$ such that $\theta_\mathfrak e$ is constant on $F \cap \theta_{\bar{\mathfrak e}}^{-1}(x')$, for all $x' \in X'$.
\end{corollary}

\begin{proof}
Consider the function $\vartheta := (\theta_\mathfrak e,\theta_{\bar{\mathfrak e}})$: by the minimality of $\theta_{\bar{\mathfrak e}}$, it follows that
\[
\xi_{\bar{\mathfrak e}} = \int \xi_{\bar{\mathfrak e},(x',x'')} \xi_\vartheta(dx'dx''), \quad \xi_\vartheta := \vartheta_\sharp \mu.
\]
Since $(\mathtt p_2)_\sharp \xi_\vartheta = \xi_{\bar{\mathfrak e}}$, then also
\[
\xi_\vartheta = \int \xi_{\vartheta,x'} \xi_{\bar{\mathfrak e}}(dx'),
\]
and thus
\[
\xi_{\bar{\mathfrak e}} = \int \bigg[ \int \xi_{\bar{\mathfrak e},(x',x'')} \xi_{\vartheta,x''}(dx'dx'') \bigg] \xi_{\bar{\mathfrak e}}(dx').
\]
This implies that $\xi_{\bar{\mathfrak e}}$-a.e. $x'$
\[
\int \xi_{\bar{\mathfrak e},(x',x'')} \xi_{\vartheta,x'}(dtds) = \delta_{x'},
\]
or equivalently that
\[
\xi_{\vartheta,x'''} = \delta_{\mathtt x(x'''),\mathtt x''(x''')}, \quad \xi_{\bar{\mathfrak e},(\mathtt x'(x'''),\mathtt x''(x'''))} = \delta_{x'''}.
\]
Hence $\xi_\vartheta$ is concentrated on a graph: by choosing $x'=x'''$, there exists $\mathtt s = \mathtt s(x')$ Borel such that $\xi_\vartheta = (\Id,\mathtt s)_\sharp \xi_{\bar{\mathfrak e}}$. This is equivalent to say that there exists a $\mu$-conegligible set $F$ such that $\theta_\mathfrak e = \mathtt s \circ \theta_{\bar{\mathfrak e}}$ on $F$.
\end{proof}



\bibliography{./biblio}

\end{document}